\documentclass[11pt]{article}

\usepackage{arxiv}

\usepackage[utf8]{inputenc} 
\usepackage[T1]{fontenc}    
\usepackage{hyperref}       
\usepackage{url}            
\usepackage{booktabs}       
\usepackage{amsfonts}       
\usepackage{nicefrac}       
\usepackage{microtype}      
\usepackage[skins]{tcolorbox}        
\usepackage{graphicx}
\usepackage{float}
\usepackage{mathtools}
\usepackage{tabularx}
\usepackage[table]{xcolor}

\usepackage{amssymb, amsmath, latexsym, bm}
\usepackage{nicefrac}
\usepackage{hhline}
\usepackage{multirow}
\usepackage{amsmath}
\usepackage{amsthm}
\usepackage{pifont}
\usepackage[customcolors]{hf-tikz}

\usepackage[colorinlistoftodos,bordercolor=orange,backgroundcolor=orange!20,linecolor=orange,textsize=scriptsize]{todonotes}
 \usepackage{algorithm}
 \usepackage{algpseudocode}
\usepackage{caption}
\usepackage{subcaption}
\usepackage{enumitem}
\usepackage{natbib}
\usepackage{sidecap}

\newcommand{\argmin}{\mathop{\arg\!\min}}
\newcommand{\argmax}{\mathop{\arg\!\max}}

\newcommand{\eqdef}{\overset{\text{def}}{=}}

\newcommand{\prox}{\mathrm{prox}}
  
\newcommand{\cA}{{\cal A}}
\newcommand{\cB}{{\cal B}}

\newcommand{\cO}{{\cal O}}

\newcommand{\cT}{{\cal T}}

\def\R{\mathbb{R}}

\def\R{\mathbb R}

\def\la{\langle}
\def\ra{\rangle}

\def\xs{x^{\star}}
\def\ys{y^{\star}}
\def\zs{z^{\star}}

\newtheorem{lemma}{Lemma}
\newtheorem{theorem}{Theorem}
\newtheorem{definition}{Definition}

\newtheorem{corollary}{Corollary}
\newtheorem{remark}{Remark}

\usepackage[flushleft]{threeparttable} 

\usepackage{caption}
\usepackage{multirow}
\usepackage{colortbl}
\definecolor{bgcolor}{rgb}{0.8,1,1}
\definecolor{bgcolor2}{rgb}{0.8,1,0.8}
\definecolor{niceblue}{rgb}{0.0,0.19,0.56}

\definecolor{mydarkred}{RGB}{192,20,20}
\definecolor{mydarkgreen}{RGB}{20,145,20}
\definecolor{mydarkblue}{RGB}{20,20,192}
\definecolor{mydarkpurple}{RGB}{110,20,180}
\newcommand{\myred}{\color{mydarkred}}
\newcommand{\mygreen}{\color{mydarkgreen}}
\newcommand{\myblue}{\color{niceblue}}
\newcommand{\mypuprle}{\color{mydarkpurple}}

\newcommand{\hstar}{{\myred h^{*}}}

\usepackage{hyperref}
\hypersetup{colorlinks,linkcolor={niceblue},citecolor={niceblue},urlcolor={niceblue}}

\usepackage{pifont}
\definecolor{PineGreen}{RGB}{0,110,51}
\definecolor{BrickRed}{RGB}{143,20,2}
\newcommand{\xmark}{{\color{BrickRed}\ding{55}}}%
\definecolor{linen}{HTML}{FAF0E6}

\usepackage{xspace}
\newcommand{\algname}[1]{{\color{BrickRed}\small\sf#1}\xspace}

\title{A Unified Primal--Dual Recipe for Accelerating Three-Operator Splitting Methods}

\author{ Abdurakhmon Sadiev \\
    KAUST, Saudi Arabia\\
    \texttt{abdurakhmon.sadiev@kaust.edu.sa}
    \And
    Laurent Condat \\
    KAUST, Saudi Arabia\\
    \texttt{https://lcondat.github.io/}
    \And
    Peter Richt\'arik \\
    KAUST, Saudi Arabia\\
    \texttt{peter.richtarik@kaust.edu.sa}
}

\begin{document}
\maketitle

\begin{abstract}
    Composite optimization problems, formulated as the minimization of three functions, are ubiquitous in large-scale machine learning and signal processing. While state-of-the-art splitting methods such as Condat--V\~{u} (\algname{CV}) \citep{condat2013primal,vu2013splitting}, Primal--Dual Davis--Yin (\algname{PDDY}) \citep{salim2022dualize}, and Primal--Dual Twice Reflected (\algname{PDTR}) \citep{malitsky2026first} are highly versatile, they inherently exhibit non-accelerated convergence rates. Existing accelerated primal--dual splitting results either focus on special structures like one of the functions being zero, or linearly constrained problems, or smooth regimes, or directly use Nesterov-type momentum. Our contribution is a unified Bregman primal--dual framework that yields four variants and a common Lyapunov analysis. By applying the Chambolle--Pock algorithm \citep{chambolle2011first} to primal--dual reformulations, we systematically derive four novel accelerated algorithms: Accelerated Condat--V\~{u} (\algname{ACV-I} and \algname{ACV-II}) and Accelerated Primal--Dual Twice Reflected (\algname{APDTR-I} and \algname{APDTR-II}). Through a simplified Lyapunov-based analysis, we establish iteration complexities for both smooth and nonsmooth cases, successfully removing the restrictive assumptions restrictions required by prior works. 
\end{abstract}

\section{Introduction}

Accelerating the convergence of algorithms has historically been a central focus in optimization theory. The standard Gradient Descent (\algname{GD}) method converges at an $\mathcal{O}(1/k)$ rate for smooth convex functions, and at a linear $\mathcal{O}\big(e^{-k/\kappa}\big)$ rate for smooth strongly convex functions, where $k$ is the iteration index and $\kappa$ is the condition number of the objective function. By contrast, Nesterov's accelerated gradient descent (\algname{AGD}) achieves a strictly superior, and indeed optimal, $\mathcal{O}(1/k^2)$ rate for smooth convex functions \citep{nesterov1983method}, and an optimal $\mathcal{O}\big(e^{-k/\sqrt{\kappa}}\big)$ rate for smooth strongly convex functions \citep{nesterov2013introductory}. Driven by the growing demands of machine learning and large-scale data processing, understanding the fundamental nature of this acceleration and extending it to complex optimization problems has become a primary research objective in the field.

A highly versatile framework for tackling such large-scale challenges is composite optimization, specifically formulated as the minimization of three convex functions involving a linear operator:
\begin{equation}
    \min_{x} {\myblue f}(x) +{\mygreen g}(x) + {\myred h}(\mathsf{K}x)\label{eqp1}
\end{equation}
(see Section \ref{secsetup} for a formal definition). To solve this class of problems, proximal splitting methods are widely utilized. These methods operate on the principle of iteratively evaluating the gradients, proximal operators, and linear operators individually \citep{par14, condat2023proximal}. A foundational method within this family is the Forward--Backward splitting method (or proximal gradient descent), which alternates between gradient (forward) and proximal (backward) steps. For more complex composite structures, algorithms such as Douglas--Rachford \citep{douglas1956numerical,lions1979splitting,eckstein1992douglas,combettes2004solving,svaiter2011weak} and the Chambolle--Pock algorithm (\algname{CP})  \citep{chambolle2011first} are frequently employed.

For problems involving the sum of three functions, state-of-the-art splitting methods include the Condat--V\~{u} (\algname{CV}), 
\citep{condat2013primal, vu2013splitting}, Primal--Dual Three-Operator Splitting (\algname{PD3O}) \citep{yan2018new}, Primal--Dual Davis--Yin (\algname{PDDY}) \citep{salim2022dualize}, and Primal--Dual Twice Reflected (\algname{PDTR})  \citep{malitsky2026first} methods. However, a significant limitation of these standard methods is that they inherently exhibit non-accelerated convergence rates, except in isolated special cases. While recent efforts have attempted to bridge this gap, achieving optimal accelerated rates typically relies on restrictive structural assumptions (such as requiring one of the components to be zero), which prevents solving Problem \eqref{eqp1} in full generality.

In this paper, we bridge this theoretical gap by introducing a unified recipe for accelerating gradient-based splitting methods via a primal--dual perspective. By applying \algname{CP}  to primal--dual reformulations, we systematically derive novel accelerated variants of well-known splitting algorithms. Crucially, these proposed methods achieve accelerated convergence guarantees 
without relying on restrictive functional assumptions.

\section{Related Work}

\textbf{The Nature of Acceleration.} In parallel with the algorithmic extensions of Accelerated Gradient Descent (\algname{AGD}) to various scenarios, there has been profound interest in understanding the fundamental nature of acceleration. Several distinct theoretical frameworks have been proposed to demystify Nesterov's momentum. \citet{su2016differential} studied the continuous-time limit dynamics of \algname{AGD}, modeling it as a second-order ordinary differential equation. From an optimization viewpoint, \citet{allen2017linear} interpreted accelerated gradient descent as a linear coupling of gradient descent and mirror descent. \citet{ahn2022understanding} investigated acceleration from the perspective of the proximal point method. Furthermore, \citet{wang2018acceleration, wang2024no} explained acceleration through the lens of no-regret dynamics in Fenchel games. 

Since Nesterov's original proposal \citep{nesterov1983method}, various formulations of \algname{AGD} have been developed to handle different problem classes, including extensions for strongly convex functions and composite proximal setups \citep{nesterov2013gradient, nesterov2018lectures}. While standard \algname{AGD} guarantees the optimal $\mathcal{O}(1/k^2)$ rate in terms of the objective function, it inherently sacrifices the monotonic decrease of both the objective function and the gradient norm. This lack of monotonicity often accompanies acceleration, arising from the underlying rotational dynamics in the primal--dual space. To refine the theoretical limits of these methods, optimized formulations have been developed. Notably, the Optimized Gradient Method (\algname{OGM}) proposed by \citet{kim2016optimized} provides an approach that achieves the exact optimal worst-case convergence bound, improving upon standard \algname{AGD}.

Finally, viewing acceleration through a primal--dual lens has proven highly fruitful. The Chambolle--Pock algorithm (\algname{CP})  \citep{chambolle2011first} inherently exhibits acceleration under strong convexity assumptions in both the primal and dual domains. Building on this, researchers have shown that by incorporating Bregman divergences or modifying the dual update steps within the \algname{CP} framework, it is possible to directly recover accelerated gradient methods \citep{lan2018optimal, lan2018random}. This demonstrates that primal--dual splitting is not merely a structural tool, but a mechanism that fundamentally captures the mechanics of acceleration.

\textbf{Splitting Methods.} For problems involving the sum of three operators, several splitting frameworks have been established, relying on access to the proximal operators of two nonsmooth functions and the gradient of one smooth function. Prominent examples include \algname{CV} (Forms I and II) \citep{condat2013primal, vu2013splitting}, \algname{PD3O} \citep{yan2018new}, \algname{PDDY} \citep{salim2022dualize}, and \algname{PDTR} \citep{malitsky2026first}. While highly flexible, these base algorithms do not inherently exhibit accelerated convergence rates.

\textbf{Accelerated Primal--Dual Methods.} The observation that \algname{CP} achieves an accelerated rate with a dependence on $\sqrt{\mu_1\mu_2}$ 
---when one function is $\mu_1$-strongly convex and the conjugate of the other is $\mu_2$-strongly convex---has inspired a distinct line of primal--dual acceleration research. \citet{lan2018optimal, lan2018random} demonstrated that replacing the quadratic term with the Bregman divergence of the dual function in the proximal update can transform \algname{CP} into the Accelerated Gradient Descent or Accelerated Gradient Extrapolation method. Similar principles were applied to the Mirror Prox method by \citet{cohen2020relative}. Recently, \citet{driggs2024practical} proposed the Accelerated Condat--V\~{u} (\algname{ACV}) algorithm. Their approach directly incorporates Nesterov's momentum into \algname{CV}. In contrast, our proposed methodology derives acceleration inherently from the primal--dual perspective, utilizing an extrapolation parameter of $1$ and relying on a simplified Lyapunov-based analysis. Furthermore, while \citet{condat2026nesterov} recently developed the 
Accelerated Proximal Alternating Predictor--Corrector algorithm (\algname{APAPC}), their theoretical guarantees remain restricted to cases where 
${\mygreen g} = (\mu_{\mygreen g}/2)\| \cdot \|^2$. Our framework successfully removes these limiting assumptions to accommodate general composite optimization problems.

\subsection{Contributions}

Our main contributions can be summarized as follows (see also Table~\ref{tab:summary}):
\begin{itemize}
    \item \textbf{A Unified Recipe for Acceleration.} We introduce a straightforward framework for accelerating gradient-based methods by utilizing a primal--dual reformulation of the original optimization problem and applying \algname{CP}. As an initial demonstration, we use this recipe to systematically derive Accelerated Proximal Gradient Descent (\algname{APGD}) and Accelerated Proximal Gradient Extrapolation (\algname{APGE}) alongside a unified Lyapunov-based analysis.
    \item \textbf{Four New Accelerated Primal--Dual Methods.} By extending our recipe to three-operator composite problems, we derive the Accelerated Condat--V\~{u} algorithms (\algname{ACV-I} and \algname{ACV-II}). We additionally introduce the Accelerated Primal--Dual Twice Reflected algorithms (\algname{APDTR-I} and \algname{APDTR-II}). 
    \item \textbf{Accelerated Convergence Guarantees.} Through rigorous Lyapunov-based analysis, we establish accelerated convergence guarantees for the \algname{ACV-I}, \algname{ACV-II}, \algname{APDTR-I}, and \algname{APDTR-II} algorithms in two different settings: when ${\myred h}$ is smooth and when it is nonsmooth. Notably, our derived complexity bounds achieve optimal rates under the smoothness assumption on the function ${\myred h}$ and accelerated rates otherwise, entirely removing the restrictive structural assumptions (e.g., ${\mygreen g} = 0$ or ${\myblue f}= 0$) required by existing literature.
\end{itemize}

\begin{table}[t]
    \centering
    \caption{Summary of theoretical iteration complexities of primal--dual methods for solving problem \eqref{eq:main_opt_prob} to achieve $\Psi_k \leq \varepsilon \Psi_0$, where $\Psi_k$ is a Lyapunov function. Parameters $\mu$ and $L$ denote the strong convexity and smoothness constants of the respective functions, while $\mathsf{K}$ is the linear operator with operator norm $\|\mathsf{K}\|$. The smallest eigenvalue and smallest positive eigenvalue of $\mathsf{K}\mathsf{K}^*$ are denoted by $\lambda_{\min}(\mathsf{K}\mathsf{K}^*)$ and $\lambda^+_{\min}(\mathsf{K}\mathsf{K}^*)$, respectively. Our proposed algorithms (bottom row) establish generalized convergence guarantees for both smooth and nonsmooth ${\myred h}$ without requiring ${\mygreen g} = 0$ or ${\myblue f} = 0$, under the additional nonsmooth-regime assumptions stated in Theorem~\ref{th:convergence_ACV_and_APDTR_non_smooth}.\vspace{2mm} }
    \label{tab:summary}
 \begin{threeparttable}
\resizebox{\textwidth}{!}{%
\begin{tabular}{c|c|c|c}
\bf Algorithm &\bf Restrictions & \bf ${\myred h}$ is smooth & \bf ${\myred h}$ is nonsmooth \\ 
\hline \hline 
\begin{tabular}{c}  \algname{PDDY} \\  \citep{salim2022dualize}\\\citep{condat2023randprox} \end{tabular}
&  
& $ \frac{L_{\myblue f}}{\mu_{\mygreen g}}+ \frac{\|\mathsf{K}\|}{\sqrt{\mu_{\mygreen g}\mu_{\hstar}}}$
& $ {\frac{L_{\myblue f}}{\mu_{\myblue f}} + \frac{\|\mathsf{K}\|^2}{\lambda_{\min}(\mathsf{K}\mathsf{K}^*)}}~^{\color{blue}(1)}$ \\
\begin{tabular}{c}  \algname{APAPC} \\  \citep{kovalev2020optimal}\\\citep{salim2022optimal} \end{tabular}
& \begin{tabular}{c} ${\mygreen g} = 0$ \\ ${\myred h} = \iota_{b} ~^{\color{blue}(2)}$ \end{tabular}
& \xmark 
& $ \sqrt{\frac{L_{\myblue f}}{\mu_{\myblue f}}}  \frac{\|\mathsf{K}\|}{\sqrt{\lambda^+_{\min}(\mathsf{K}\mathsf{K}^*)}} + \frac{\|\mathsf{K}\|^2}{\lambda^+_{\min}(\mathsf{K}\mathsf{K}^*)}$ \\
\begin{tabular}{c}  \algname{OPAPC} \\  \citep{kovalev2020optimal}\\\citep{salim2022optimal} \end{tabular}
& \begin{tabular}{c} ${\mygreen g} = 0$ \\ ${\myred h} = \iota_{b}$ \end{tabular}
& \xmark 
& $ \sqrt{\frac{L_{\myblue f}}{\mu_{\myblue f}}}  \frac{\|\mathsf{K}\|}{\sqrt{\lambda^+_{\min}(\mathsf{K}\mathsf{K}^*)}}$\\
\begin{tabular}{c}  \algname{APDA} \\  \citep{sadiev2022communication} \end{tabular}
& ${\myblue f} = 0$ 
& $ {\frac{\|\mathsf{K}\|}{\sqrt{\mu_{\mygreen g}\mu_{\hstar}}}}^{\color{blue}(3)}$ 
& $ \sqrt{\frac{L_{\mygreen g}}{\mu_{\mygreen g}}} \frac{\|\mathsf{K}\|}{\sqrt{\lambda_{\min}(\mathsf{K}\mathsf{K}^*)}} + \frac{\|\mathsf{K}\|^2}{\lambda_{\min}(\mathsf{K}\mathsf{K}^*)}$ \\
\begin{tabular}{c}  \algname{ACV} \\  \citep{driggs2024practical} \end{tabular}
&  
& $ \sqrt{\frac{L_{\myblue f}}{\mu_{\mygreen g}}} + \frac{\|\mathsf{K}\|}{\sqrt{\mu_{\mygreen g}\mu_{\hstar}}}$
& \xmark \\
\begin{tabular}{c}  \algname{APAPC} \\  \citep{condat2026nesterov} \end{tabular}
& ${\mygreen g} = \frac{\mu_{\mygreen g}}{2}\|x\|^2$ 
& $ \sqrt{\frac{L_{\myblue f}}{\mu_{\mygreen g}}} + \frac{\|\mathsf{K}\|}{\sqrt{\mu_{\mygreen g}\mu_{\hstar}}}$
& $ \sqrt{\frac{L_{\myblue f}}{\mu_{\mygreen g}}}  \frac{\|\mathsf{K}\|}{\sqrt{\lambda_{\min}(\mathsf{K}\mathsf{K}^*)}} + \frac{\|\mathsf{K}\|^2}{\lambda_{\min}(\mathsf{K}\mathsf{K}^*)}$$^{\color{blue}(4)}$ \\
\hline
\cellcolor{linen} \begin{tabular}{c}  \algname{ACV-I} \& \algname{ACV-II} \\  \algname{APDTR-I} \& \algname{APDTR-II} \\  Algorithm \ref{alg:ACV} \& \ref{alg:APDTR} $^{{\color{blue}(5)}}$\end{tabular}
& \cellcolor{linen}  &\cellcolor{linen} $ \sqrt{\frac{L_{\myblue f}}{\mu_{\mygreen g}}} + \frac{\|\mathsf{K}\|}{\sqrt{\mu_{\mygreen g}\mu_{\hstar}}}$ &\cellcolor{linen}  $ \sqrt{\frac{L_{\myblue f} +L_{\mygreen g}}{\mu_{\mygreen g}}}  \frac{\|\mathsf{K}\|}{\sqrt{\lambda_{\min}(\mathsf{K}\mathsf{K}^*)}} + \frac{\|\mathsf{K}\|^2}{\lambda_{\min}(\mathsf{K}\mathsf{K}^*)}$ \\
\end{tabular}
}
\begin{tablenotes}
        {\footnotesize 
        \item [{\color{blue}(1)}] When ${\myred h}$ is nonsmooth, \citet{condat2023randprox} show linear convergence rate for \algname{PDDY} algorithm under the assumption ${\mygreen g} = 0$.
        \item [{\color{blue}(2)}] \citet{salim2022optimal} study the linearly constrained case where ${\myred h} = \iota_{b}$, with $\iota_{b}(x) = (0$ if $x = b;\ +\infty$ otherwise$)$.
        \item [{\color{blue}(3)}] When ${\myred h}$ is smooth, \algname{APDA} reduces to \algname{CP}.
        \item [{\color{blue}(4)}] \citet{condat2026nesterov} show results for nonsmooth and linearly constrained cases. The complexity rate for the linearly constrained case  is the same as presented in the table with $\lambda_{\min}(\mathsf{K}\mathsf{K}^*)$ replaced by $\lambda_{\min}^+ (\mathsf{K}\mathsf{K}^*)$.
        \item [{\color{blue}(5)}] The resulting iteration complexity can be found in Corollary \ref{cor:acv_and_apdtr_complexity}, Corollary \ref{cor:acv_and_apdtr_complexity_non_smooth} and Corollary \ref{cor:acv_and_apdtr_complexity_non_smooth_lin_const} for the smooth, nonsmooth and linearly constrained cases, respectively.
        }
    \end{tablenotes}
    \end{threeparttable}
\end{table}

\section{Problem Setup}\label{secsetup}

Before formally stating our optimization problem, we recall the standard definitions of strong convexity, smoothness, and convex duality. Let $d\ge 1$.

\begin{definition}
    A proper, closed function $\phi: \R^d \rightarrow \R \cup \{+\infty\}$ is called $\mu$-strongly convex (for $\mu \ge 0$) if, for all $x, y \in \mathrm{dom}(\phi)$ and any subgradient $p \in \partial \phi(y)$, it holds that
    \begin{equation}
        \label{eq:mu_strongly_convex_def}
        \phi(x) \geq \phi(y) + \la p, x-y\ra + \tfrac{\mu}{2}\|x-y\|^2.
    \end{equation}
    If $\mu = 0$, the function $\phi$ is simply called convex.
\end{definition}

\begin{definition}
    A differentiable function $\phi: \R^d \rightarrow \R$ is called $L$-smooth (for $L > 0$) if its gradient $\nabla\phi$ is $L$-Lipschitz continuous, i.e.,
    \begin{equation}
        \label{eq:L_smooth_def}
        \|\nabla\phi(x) - \nabla\phi(y)\| \leq L\|x-y\|,~~ \text{ for all } ~~x,y \in \R^d.
    \end{equation}
\end{definition}

\begin{definition}
    Let $\phi: \R^d \rightarrow \R \cup \{+\infty\}$ be a proper, closed, convex function. Its convex conjugate $\phi^*$ is defined as
    \begin{equation*}
        \phi^*(v) = \sup_{u \in \R^d} \big\{ \la v,u\ra - \phi(u) \big\}.
    \end{equation*}
\end{definition}

A fundamental property connecting these definitions is given by the following standard lemma:
\begin{lemma}[
{\citep[Theorem 18.15]{bau17}}]
    \label{lem:strongly_convex_conjugate_smoothness}
    Let $\phi$ be a proper, closed, convex and differentiable function. If $\phi$ is $L$-smooth, its convex conjugate $\phi^*$ is $\nicefrac{1}{L}$-strongly convex.
\end{lemma}

With these concepts established, we consider the composite optimization problem formulated as the minimization of three convex functions: 
\begin{equation}
    \label{eq:main_opt_prob}
    \min_{x\in \R^{d_x}} {\myblue f}(x) + {\mygreen g}(x) + {\myred h}(\mathsf{K} x),
\end{equation}
where ${\myblue f}:\R^{d_x}\rightarrow \R$, ${\mygreen g}:\R^{d_x}\rightarrow \R \cup \{+\infty\}$, and ${\myred h}:\R^{d_y} \rightarrow \R \cup \{+\infty\}$ are proper, closed, convex functions for some $d_x\ge 1$ and $d_y\ge 1$, 
and $\mathsf{K}: \R^{d_x}\rightarrow\R^{d_y}$ is a linear operator with operator norm $\|\mathsf{K}\| = \sup_{x~:~\|x\|=1}\|\mathsf{K} x\|$. Additionally, we assume that ${\myblue f}$ is $L_{{\myblue f}}$-smooth and $\mu_{{\myblue f}}$-strongly convex, and ${\mygreen g}$ is $\mu_{{\mygreen g}}$-strongly convex, for some $L_{{\myblue f}}>0$, $\mu_{{\myblue f}}\ge 0$, $\mu_{{\mygreen g}}\ge 0$.

To derive accelerated methods, we can rewrite problem \eqref{eq:main_opt_prob} as follows:
\begin{equation}
\label{eq:main_opt_prob_2}
    \min_{x\in \R^{d_x}}  {\mygreen g}(x) + \underbrace{{\myblue f}(\mathsf{I} x) +{\myred h}(\mathsf{K} x)}_{\eqdef {\mypuprle \ell}(\mathsf{L} x)},
\end{equation}
where $\mathsf{I} \in \R^{d_x\times d_x}$ is the identity operator, and we define the block matrix $\mathsf{L}$ as
\begin{equation*}
    \mathsf{L} = \begin{pmatrix}
        \mathsf{I} \\
        \mathsf{K}
    \end{pmatrix},~~ \text{ and accordingly } ~~ \mathsf{L}^* = \begin{pmatrix}
        \mathsf{I} & \mathsf{K}^*
    \end{pmatrix}.
\end{equation*}

By the definition of the conjugate function, Problem \eqref{eq:main_opt_prob_2} is equivalent to the following saddle-point reformulation:
\begin{equation*}
    \min_{x\in \R^{d_x}} \max_{v \in \R^{d_x}\times\R^{d_y}}  {\mygreen g}(x) + \la v,\mathsf{L} x\ra - {\mypuprle \ell}^*(v), \quad \text{where }\quad v \eqdef \begin{pmatrix}
        u\\ y
    \end{pmatrix}.
\end{equation*}
Equivalently, by the definition of the function ${\mypuprle \ell}$, the saddle-point problem is:
\begin{equation}
    \label{eq:main_saddle_point_reformulation}
    \min_{x\in \R^{d_x}} \max_{y \in \R^{d_y}}\max_{u \in \R^{d_x}}  {\mygreen g}(x) + \la y, \mathsf{K} x\ra + \la u, x\ra - \hstar(y) - {\myblue f^*}(u).
\end{equation}

A saddle-point solution $(\xs, u^{\star}, \ys) \in \R^{d_x} \times \R^{d_x} \times \R^{d_y}$  of problem \eqref{eq:main_saddle_point_reformulation} is characterized by the  optimality conditions
\begin{equation}
    \label{eq:main_optimality_condition}
    \begin{cases}
        0 \in \partial {\mygreen g}(\xs) + \mathsf{K}^*\ys + u^{\star};\\
        0 \in \partial \hstar(\ys) - \mathsf{K}\xs;\\
        0 \in \partial {\myblue f^*} (u^{\star}) - \xs;
    \end{cases}
    \quad\Leftrightarrow \quad
    \begin{cases}
        0 \in \partial {\mygreen g}(\xs) + \mathsf{K}^*\ys + \nabla {\myblue f} (\zs);\\
        0 \in \partial \hstar(\ys) - \mathsf{K}\xs;\\
        0 = \xs-\zs.
    \end{cases}
\end{equation}
Indeed, from the formulation of \eqref{eq:main_saddle_point_reformulation}, we have $u^{\star} = \argmax_{u\in \R^{d_x}} \left\{ \langle u, \xs \rangle - {\myblue f^*} (u) \right\}$. It follows that $u^{\star} \in \partial {\myblue f} (\xs)$ \citep{bau17}; 
equivalently, by differentiability of ${\myblue f}$, we have $u^{\star} = \nabla {\myblue f} (\xs)$.

For the optimization problems to be well posed, we assume throughout that a solution to \eqref{eq:main_optimality_condition} exists.

Building on the reformulation in \eqref{eq:main_saddle_point_reformulation}, we will systematically derive several accelerated algorithms via a primal--dual perspective. First, we explain how to derive Accelerated Proximal Gradient Descent (\algname{APGD}) and Accelerated Proximal Gradient Extrapolation (\algname{APGE}) via the application of \algname{CP} on problem \eqref{eq:main_opt_prob} in the simplified case where ${\myred h} = 0$ \citep{chambolle2011first,condat2023proximal}. This will serve as our core recipe for acceleration. Next, we apply this recipe to the original problem~\eqref{eq:main_opt_prob}. Through this methodology, we derive four novel accelerated primal--dual algorithms: the Accelerated Condat--V\~u algorithms (Forms I and II) and the Accelerated Primal--Dual Twice Reflected algorithms (Forms I and II).

\subsection{Proximal operator}

For any proper, closed, convex function $F:\R^d\to \R\cup\{+\infty\}$, its proximal operator is defined as a mapping from $\R^d$ to $\R^d$:
\begin{equation}
    \label{eq:prox_def}
    v \mapsto \prox_{F}(v) \eqdef \argmin_{w\in \R^d}\left\{F(w)+\tfrac{1}{2}\|w-v\|^2\right\}.
\end{equation} 
In particular, for a proximal gradient update with stepsize $\gamma > 0$, involving $F$ and a convex differentiable function $G$,  we can write:
\begin{align*}
    \prox_{\gamma F}\big(v - \gamma \nabla G(v)\big) &= \argmin_{w\in \R^d}\left\{ F(w) + \la  \nabla G(v), w-v\ra +\tfrac{1}{2\gamma}\|w- v\|^2\right\}.
\end{align*}

\begin{definition}
    \label{def:bregman_divergence}
    Let $\phi:\R^d \rightarrow \R$ be a convex and differentiable function. Then the Bregman divergence associated with $\phi$ is defined as 
    \begin{equation}
        \label{eq:Bregman_def}
        D_{\phi}(x;y) \eqdef \phi (x) -\phi(y) - \la \nabla\phi (y), x-y\ra.
    \end{equation}
\end{definition}

The Bregman proximal operator of $F$ with respect to $\phi$ is defined as a mapping from $\R^d$ to $\R^d$:
\begin{equation*}
    v \mapsto \mathrm{prox}^{\phi}(v, F) \eqdef \argmin_{w \in \R^d}\left\{F(w) + D_{\phi}(w;v)\right\}.
\end{equation*}
Consequently, in our notation, the Bregman proximal gradient update is written as:
\begin{align*}
    \mathrm{prox}^{\phi}(v, \gamma (F + \la\nabla G(v), \cdot\ra)) 
    &= \argmin_{w\in \R^d}\left\{\gamma F(w) + \gamma\la\nabla G(v), w\ra  + D_{\phi}(w; v)\right\} \\
    &= \argmin_{w\in \R^d}\left\{ F(w) + \la\nabla G(v), w\ra  + \tfrac{1}{\gamma}D_{\phi}(w; v)\right\}.
\end{align*}

\section{Accelerated Proximal Gradient Descent and Extrapolation }
\label{sec:APGD_and_APGE}

In this section, we derive the Accelerated Proximal Gradient Descent (\algname{APGD}) and Accelerated Proximal Gradient Extrapolation (\algname{APGE}) methods based on \algname{CP} (Forms I and II). The primary distinction between the algorithms proposed by \citet{lan2018optimal, lan2018random} and our \algname{APGD}/\algname{APGE} lies in the extrapolation step. While we do not claim novelty for the \algname{APGD} and \algname{APGE} algorithms themselves, deriving them through this primal--dual lens allows us to provide a unified, Lyapunov-based convergence analysis that we find significantly simpler than previous constructions.

By setting ${\myred h} = 0$, problem~\eqref{eq:main_opt_prob} simplifies to:
\begin{equation}
\label{eq:main_prob_apgd}
    \min_{x\in \R^{d_x}} {\myblue f}(x) + {\mygreen g}(x) \quad\Leftrightarrow\quad  \min_{x\in \R^{d_x}} \max_{u\in \R^{d_x}} {\mygreen g}(x) + \la u, x\ra - {\myblue f^*}(u).
\end{equation}
Then, the optimality conditions for \eqref{eq:main_prob_apgd} (the simplified version of \eqref{eq:main_saddle_point_reformulation}) are:
\begin{equation}
    \label{eq:main_prob_apgd_optimality_condition}
    \begin{cases}
        0 \in \partial {\mygreen g}(\xs) + u^{\star},\\
        0 \in \partial {\myblue f^*} (u^{\star}) - \xs;
    \end{cases}
    \quad \Leftrightarrow\quad  
    \begin{cases}
        0 \in \partial {\mygreen g}(\xs) + \nabla {\myblue f} (\zs),\\
        0 = \zs - \xs.
    \end{cases}
\end{equation}

Applying \algname{CP} (Forms I and II) to the saddle-point reformulation in \eqref{eq:main_prob_apgd}, we obtain the following two methods:

\vspace{-0.5em}
\noindent
\begin{minipage}{0.48\textwidth}
\begin{align}
    x^{k+1} &= \prox_{\eta_x {\mygreen g}}\left(x^k- \eta_xu^k\right),\notag\\
    u^{k+1} &= \prox_{\eta_u {\myblue f^*}}\left(u^k+\eta_u (2x^{k+1}- x^k)\right); \label{eq:dual_update_apgd_1}
\end{align}
\end{minipage}%
\hfill\vrule\hfill
\begin{minipage}{0.48\textwidth}
\begin{align}
    u^{k+1} &= \prox_{\eta_u {\myblue f^*}}\left(u^k+\eta_u x^k\right),\label{eq:dual_update_apgd_2}\\
    x^{k+1} &= \prox_{\eta_x {\mygreen g}}\left(x^k- \eta_x(2u^{k+1}- u^k)\right).\notag
\end{align}
\end{minipage}

Next, we modify the dual updates \eqref{eq:dual_update_apgd_1} and \eqref{eq:dual_update_apgd_2} by replacing the standard Euclidean penalty $\tfrac{1}{2\eta_u}\|u -u^k\|^2$ with the Bregman divergence $\tfrac{1}{\eta_u}D_{{\myblue f^*} }(u; u^k)$. Writing this Bregman proximal step explicitly yields:
\vspace{-0.5em}
\noindent
\begin{minipage}{0.48\textwidth}
\begin{align}
    x^{k+1} &= \prox_{\eta_x {\mygreen g}}\left(x^k- \eta_xu^k\right),\notag\\
    u^{k+1} &= \prox^{\myblue f^*}\!\left(u^k, \eta_u\left( {\myblue f^*}(\cdot) -\la\cdot,2x^{k+1}-x^k\ra\right)\right); \notag
\end{align}
\end{minipage}%
\hfill\vrule\hfill
\begin{minipage}{0.48\textwidth}
\begin{align}
    u^{k+1} &= \prox^{\myblue f^*}\!\left(u^k, \eta_u\left( {\myblue f^*}(\cdot) -\la\cdot,x^k\ra\right)\right),\notag\\
    x^{k+1} &= \prox_{\eta_x {\mygreen g}}\left(x^k- \eta_x(2u^{k+1}- u^k)\right).\notag
\end{align}
\end{minipage}

To facilitate the algorithmic derivation, we assume that ${\myblue f}$ and ${\myblue f^*}$ are both differentiable. By evaluating the optimality condition of the $u$-update (setting the gradient with respect to $u$ to zero), we obtain:

\vspace{-0.5em}
\noindent
\begin{minipage}{0.53\textwidth}
\begin{align}
    \nabla {\myblue f^*} (u^{k+1}) &= \nabla {\myblue f^*} (u^k) - \eta_u(\nabla {\myblue f^*} (u^{k+1}) - (2x^{k+1} - x^k) ); \notag
\end{align}
\end{minipage}%
\hfill\vrule\hfill
\begin{minipage}{0.43\textwidth}
\begin{align}
    \nabla {\myblue f^*} (u^{k+1}) &= \nabla {\myblue f^*} (u^k) - \eta_u(\nabla {\myblue f^*} (u^{k+1}) - x^k ). \notag
\end{align}
\end{minipage}\smallskip
\begin{lemma}[
\citep{bau17}]
    \label{lem:inverse_gradient_of_conjugate_function}
        If a function $\phi: \R^{d_x} \to \R \cup \{+\infty\}$ is proper, closed, convex, then $(\partial \phi)^{-1} = \partial \phi^*$. If, additionally, both $\phi$ and 
        $\phi^*$ are differentiable, then 
        $(\nabla \phi)^{-1} = \nabla \phi^*$.
\end{lemma}

Denoting $z^k \eqdef \nabla {\myblue f^*} (u^k)$ (which by Lemma~\ref{lem:inverse_gradient_of_conjugate_function} implies $u^k = \nabla {\myblue f}(z^k)$) and substituting $\eta_z \eqdef \eta_u$, we arrive at

\vspace{-0.5em}
\noindent
\begin{minipage}{0.48\textwidth}
\begin{align}
    x^{k+1} &= \prox_{\eta_x {\mygreen g}}\left(x^k- \eta_x\nabla {\myblue f}(z^k)\right),\notag\\
    z^{k+1} &= z^k - \eta_z (z^{k+1} - (2x^{k+1} - x^k) ), \notag
\end{align}
\end{minipage}%
\hfill\vrule\hfill
\begin{minipage}{0.48\textwidth}
\begin{align}
    z^{k+1} &= z^k -\eta_z (z^{k+1} -  x^k ),\notag\\
    x^{k+1} &= \prox_{\eta_x {\mygreen g}}\left(x^k- \eta_x(2\nabla {\myblue f}(z^{k+1})- \nabla {\myblue f}(z^{k}))\right),\notag
\end{align}
\end{minipage}

which correspond to the Accelerated Proximal Gradient Descent (\algname{APGD}) and Accelerated Proximal Gradient Extrapolation (\algname{APGE}) methods, detailed in Algorithm~\ref{alg:APGD_APGE}.

\begin{remark}
    The differentiability of $\myblue f^*$ is used only for this formal derivation. The algorithms and convergence results below are stated and proved directly in the $z$-variables under the weaker assumptions that ${\myblue f}$ is $L_{{\myblue f}}$-smooth and convex.
\end{remark}

\begin{algorithm}[ht!]
\caption{Accelerated Proximal Gradient Descent \& Extrapolation (\algname{APGD}\& \algname{APGE})}
\label{alg:APGD_APGE}
    \begin{tabular}{@{}p{0.45\linewidth} | p{0.52\linewidth}@{}}
        \vspace{-0.7em} 
        \begin{algorithmic}[1]
            \State \textbf{Input:} Stepsizes $\eta_x, \eta_z >0$
            \For{ $k =  0, 1, 2, \ldots$}
                \State $x^{k+1} =\prox_{\eta_x {\mygreen g}}\left(x^k- \eta_x\nabla {\myblue f} (z^k)\right)$
                \State $z^{k+1} = \frac{1}{1+\eta_z}z^k +\frac{\eta_z}{1+\eta_z}  \left(2x^{k+1}-x^k\right)$
            \EndFor
        \end{algorithmic}
        &
        \vspace{-0.7em}
        \begin{algorithmic}[1]
            \State \textbf{Input:} Stepsizes $\eta_x, \eta_z > 0$
            \For{ $k =  0, 1, 2, \ldots$}
                \State $z^{k+1} = \frac{1}{1+\eta_z}z^k +\frac{\eta_z}{1+\eta_z}  \left(x^k\right)$
                \State $x^{k+1} =\prox_{\eta_x {\mygreen g}}\left(x^k- \eta_x(2\nabla {\myblue f}(z^{k+1})- \nabla {\myblue f}(z^{k}))\right)$
            \EndFor
        \end{algorithmic} 
    \end{tabular}\vspace{-0.2em}
\end{algorithm}

\begin{remark}
It is worth noting that there is an alternative way to derive the exact iterates of Algorithm~\ref{alg:APGD_APGE}. Specifically, we can define two monotone operators and a metric matrix to apply the matrix-scaled Douglas--Rachford splitting method:
\begin{equation*}
    \begin{cases}
         (\mathsf{P} + \cB)(z^k) = \mathsf{P} v^k;\\
         (\mathsf{P} + \cA)(w^{k+1}) = \mathsf{P}(2z^k -v^k);\\
         v^{k+1} = v^k +w^{k+1} -z^k.
    \end{cases}
\end{equation*}
The corresponding optimality condition and metric matrix are defined as follows:
\begin{equation*}
    \begin{pmatrix}
        0\\ 0
    \end{pmatrix}
    \in 
    \underbrace{\begin{pmatrix}
        \partial {\mygreen g} (x)  + u\\  -x 
    \end{pmatrix}}_{\eqdef \cB (x,u)}
    +
    \underbrace{\begin{pmatrix}
        0\\ \nabla  {\myblue f^*} (u)
    \end{pmatrix}}_{\eqdef \cA(x,u)}, \quad 
    \mathsf{P} = \begin{pmatrix}
        \left(\tfrac{1}{\eta_x} - \eta_u\right) \mathsf{I}_{d_x} & 0\\
        0 & \tfrac{1}{\eta_u}\mathsf{I}_{d_x}
    \end{pmatrix}.
\end{equation*}
Consequently, \algname{APGD} can be derived from the matrix-scaled Douglas--Rachford splitting method using operators $\cA$ and $\cB$. Similarly, \algname{APGE} can be derived simply by swapping the roles of operators $\cA$ and $\cB$.
\end{remark}

It is worth mentioning that \algname{APGD} (Algorithm~\ref{alg:APGD_APGE}) is indeed an accelerated version of Proximal Gradient Descent (\algname{PGD}). To see this clearly, suppose we set $z^0 = x^0$ and $\eta_z = 1$. The $z$-update then yields:
$$
    z^{k+1} = z^k - \big(z^{k+1} - (2x^{k+1}-x^k)\big) \quad \Rightarrow \quad 2z^{k+1}-z^k = 2x^{k+1} - x^k.
$$
By induction, starting from $z^0 = x^0$, it follows that $z^k = x^k$ for all $k \geq 0$. Substituting this into the $x$-update gives:
$$
    x^{k+1} = \prox_{\eta_x {\mygreen g}}\left(x^k- \eta_x\nabla {\myblue f} (x^k)\right),
$$
which is exactly the standard update rule for \algname{PGD}.

A similar observation holds for \algname{APGE} (Algorithm~\ref{alg:APGD_APGE}); it is indeed an accelerated version of the Forward-Reflected-Backward  splitting method (\algname{FRB}) proposed by \citet{malitsky2020forward}. To demonstrate this, suppose we set  $\eta_z = \frac{1}{\alpha_z -1}$ for some constant $\alpha_z > 1$. The $z$-update then yields: 
$$
    z^{k+1} = z^k - \eta_z(z^{k+1} - x^k) \quad \Rightarrow \quad z^{k+1} = \frac{1}{1+\eta_z}z^k + \frac{\eta_z}{1+\eta_z} x^k = \left(1-\frac{1}{\alpha_z}\right)z^k + \frac{1}{\alpha_z}x^k.
$$
As $\alpha_z \to 1$ (and consequently $\eta_z \to \infty$) and setting $z^0 = x^{-1}$, we see that $z^{k+1} \to x^k$. This is a formal limiting relation showing algorithmic ancestry; it is not covered by the stepsize regime in the convergence theorems below. In this limit, the $x$-update converges to:
$$
    x^{k+1} = \prox_{\eta_x {\mygreen g}}\left(x^k- \eta_x\big(2\nabla {\myblue f} (x^k) - \nabla {\myblue f} (x^{k-1})\big)\right),
$$
which is exactly the update rule for \algname{FRB}.

As mentioned previously, \algname{APGD} and \algname{APGE} were originally proposed by \citet{lan2018optimal} and \citet{lan2018random}, respectively. The primary algorithmic difference between our formulation and theirs lies in the selection of the extrapolation parameters $\beta_x, \beta_z \in (0,1]$: 
\begin{align*}
    \text{\citet{lan2018optimal} use:} \quad z^{k+1} &= z^k -\eta_z\left(z^{k+1}- \colorbox{yellow!30}{$\displaystyle\left[x^{k+1} +\beta_x(x^{k+1}- x^k)\right]$}\right);\\
    \text{\citet{lan2018random} use:} \quad x^{k+1} &= \prox_{\eta_x {\mygreen g}} \left(x^k -\eta_x \colorbox{yellow!30}{$\displaystyle\left[\nabla {\myblue f}(z^{k+1}) +\beta_z(\nabla {\myblue f}(z^{k+1}) - \nabla {\myblue f}(z^{k}))\right]$}\right).
\end{align*}
In our framework, we inherently fix $\beta_x = \beta_z = 1$ due to the primal--dual derivation. Beyond this structural simplification, the most significant difference is our proof technique: rather than relying on complex estimating sequences, we provide a unified, simplified Lyapunov-based analysis and explicitly construct the Lyapunov function for both algorithms.

\subsection{Convergence Guarantees}

We now establish the convergence guarantees for \algname{APGD} and \algname{APGE}. To this end, we first define the following Lyapunov function:
\begin{equation*}
    \Phi_{k} \eqdef \frac{1}{2\eta_x}\|x^{k}-\xs\|^2 + \frac{1}{\eta_z} D_{{\myblue f}}(z^{k}; \zs) \pm \la \nabla {\myblue f}(z^{k}) - \nabla{\myblue f}(\zs), x^{k} - \xs\ra, \quad \forall~ k \geq 0,
\end{equation*}
where the $\pm$ sign is taken as a minus for \algname{APGD} and a plus for \algname{APGE}.

\begin{theorem}
\label{th:apgd_apge_conv}
    Let $(\xs,\zs)$ be a solution of \eqref{eq:main_prob_apgd_optimality_condition}, and suppose that
    ${\mygreen g}$ is $\mu_{{\mygreen g}}$-strongly convex for some $\mu_{\mygreen g} \ge 0$ .
     If the stepsizes satisfy $L_{{\myblue f}}\eta_x\eta_z \leq 1$, then for every $k \geq 0$, the iterates generated by \algname{APGD} and \algname{APGE} (Algorithm~\ref{alg:APGD_APGE}) satisfy:
    \begin{equation}
        \Phi_k \leq \theta^k \Phi_0,\qquad \text{where}\quad \theta = \max\left\{\frac{1}{1+\mu_{{\mygreen g}}\eta_x}, \frac{2}{2+\eta_z}\right\}.
    \end{equation}
\end{theorem}

The proof of Theorem~\ref{th:apgd_apge_conv} is provided in Appendix~\ref{sec:proof_APGD_APGE}. Next, we derive the optimal parameter selection for \algname{APGD} and \algname{APGE} and state the corresponding convergence rate.

\begin{corollary}
    \label{cor:apgd_apge_complexity}
    Under the assumptions of Theorem~\ref{th:apgd_apge_conv}, suppose additionally that $\mu_{\mygreen g} > 0$. Then, the solution $(\xs,\zs)$ of \eqref{eq:main_prob_apgd_optimality_condition} is unique. Moreover, for any $\varepsilon >0$, the number of iterations required to achieve $\Phi_k \leq \varepsilon \Phi_0$ for \algname{APGD} and \algname{APGE} (Algorithm~\ref{alg:APGD_APGE}) with the parameter selection 
    \begin{equation*}
        \eta_x = \frac{1}{\sqrt{L_{\myblue f}\mu_{\mygreen g}}},\quad \eta_z  = \sqrt{\frac{\mu_{\mygreen g}}{L_{\myblue f}}},
    \end{equation*}
    is 
    \begin{equation*}
        \cO\left(\left(1+\sqrt{\frac{L_{\myblue f}}{\mu_{\mygreen g}}}\right)\log\frac{1}{\varepsilon} \right).
    \end{equation*}
\end{corollary}

The proof of Corollary~\ref{cor:apgd_apge_complexity} can be found in Appendix~\ref{sec:proof_of_cor_apgd_apde_complexity}.

As shown in Corollary~\ref{cor:apgd_apge_complexity}, Algorithm~\ref{alg:APGD_APGE} achieves the optimal accelerated rate in the strongly convex case, i.e., when $\mu_{\mygreen g} > 0$. In the case where ${\myblue f}$ is $\mu_{{\myblue f}}$-strongly convex, we can transfer the strong convexity by adding $\frac{\mu_{\myblue f}}{2}\|\cdot\|^2$ to ${\mygreen g}$, defining ${\mygreen \tilde g}(\cdot) \eqdef {\mygreen g}(\cdot) +  \frac{\mu_{\myblue f}}{2}\|\cdot\|^2$, and correspondingly subtracting it from ${\myblue f}$, defining ${\myblue \tilde{f}}(x) \eqdef {\myblue f}(x) - \frac{\mu_{\myblue f}}{2}\|x\|^2$. Consequently, line $3$ of \algname{APGD} and line $4$ of \algname{APGE} in Algorithm~\ref{alg:APGD_APGE} can be rewritten as
\begin{equation*}
    x^{k+1} =\prox_{\eta_x {\mygreen \tilde g}}\left(x^k- \eta_x\nabla {\myblue \tilde{f}} (z^k)  \right) \quad \text{and} \quad x^{k+1} =\prox_{\eta_x {\mygreen \tilde g}}\left(x^k- \eta_x\big(2\nabla {\myblue \tilde{f}} (z^{k+1}) - \nabla {\myblue \tilde{f}} (z^k)\big)  \right),
\end{equation*}
respectively. Applying Corollary~\ref{cor:apgd_apge_complexity} to problem~\eqref{eq:main_prob_apgd} using the modified functions ${\myblue \tilde f}$ and ${\mygreen \tilde g}$, we obtain the following convergence rate for Algorithm~\ref{alg:APGD_APGE}:
\begin{equation*}
    \cO\left(\sqrt{\frac{L_{\myblue f}}{\mu_{\mygreen g}+ \mu_{\myblue f}}}\log\frac{1}{\varepsilon} \right).
\end{equation*}
If $\mu_{\mygreen g} = 0$, this modification of Algorithm~\ref{alg:APGD_APGE} successfully achieves the optimal convergence rate of $\cO\Big(\sqrt{\frac{L_{\myblue f}}{\mu_{\myblue f}}}\log\frac{1}{\varepsilon}\Big)$.

\section{Four New Accelerated Primal--Dual Methods}
\label{sec:ACV_and_APDTR}

In Section~\ref{sec:APGD_and_APGE} we presented \emph{``the recipe of acceleration''}; that is, how to derive accelerated algorithms based on a  saddle-point reformulation and applying the Chambolle--Pock Algorithm to this problem. In this section, we apply the same recipe to problem \eqref{eq:main_saddle_point_reformulation} and derive four new accelerated algorithms.

Applying Chambolle--Pock Algorithm Form I and Form II to problem \eqref{eq:main_saddle_point_reformulation}, we obtain the following two methods:

\vspace{-0.5em}
\noindent
\begin{minipage}{0.41\textwidth}
\begin{align}
    x^{k+1} &= \prox_{\eta_x {\mygreen g}}\left(x^k- \eta_x \mathsf{K}^*y^k -\eta_xu^k \right);\notag\\
    y^{k+1} &= \prox_{\eta_y {\hstar}}\left(y^k+\eta_y \mathsf{K} (2x^{k+1}- x^k)\right); \notag\\
    u^{k+1} &= \prox_{\eta_u {\myblue f^*}}\left(u^k+\eta_u (2x^{k+1}- x^k)\right); \notag
\end{align}
\end{minipage}%
\hfill\vrule\hfill
\begin{minipage}{0.57\textwidth}
\begin{align}
    y^{k+1} &= \prox_{\eta_y {\hstar}}\left(y^k+\eta_y \mathsf{K} x^k\right); \notag\\
    u^{k+1} &= \prox_{\eta_u {\myblue f^*}}\left(u^k+\eta_u x^k\right);\notag\\
    x^{k+1} &= \prox_{\eta_x {\mygreen g}}\left(x^k- \eta_x\mathsf{K}^*(2y^{k+1}- y^k) - \eta_x(2u^{k+1}- u^k)\right).\notag
\end{align}
\end{minipage}

Using the same transformation for $u$-update as we did for \algname{APGD} and \algname{APGE} (see Section~\ref{sec:APGD_and_APGE}), we derive two new accelerated methods: 

\vspace{-0.5em}
\noindent
\begin{minipage}{0.48\textwidth}
\begin{subequations}
\label{eq:acv_main}
\begin{align}
    x^{k+1} &= \prox_{\eta_x {\mygreen g}}\left(x^k- \eta_x \mathsf{K}^*y^k -\eta_x \nabla {\myblue f}(z^k) \right);\label{eq:acv_main_x}\\
    y^{k+1} &= \prox_{\eta_y {\hstar}}\left(y^k+\eta_y \mathsf{K} (2x^{k+1}- x^k)\right); \label{eq:acv_main_y}\\
    z^{k+1} &= z^k - \eta_z\left( z^{k+1} - (2x^{k+1}- x^k)\right); \label{eq:acv_main_z}
\end{align}
\end{subequations}
\end{minipage}%
\hfill\vrule\hfill
\begin{minipage}{0.48\textwidth}
\begin{subequations}
\label{eq:apdtr_main} 
\begin{align}
    y^{k+1} &= \prox_{\eta_y {\hstar}}\left(y^k+\eta_y \mathsf{K} x^k\right); \label{eq:apdtr_y}\\
    z^{k+1} &= z^k - \eta_z\left( z^{k+1} - x^k\right); \label{eq:apdtr_z}\\
    x^{k+1} &= \prox_{\eta_x {\mygreen g}}\big(x^k- \eta_x\mathsf{K}^*\left(2y^{k+1}- y^k\right) \notag\\
    &\qquad- \eta_x\left(2\nabla {\myblue f}(z^{k+1}) - \nabla {\myblue f}(z^{k})\right)\big). \label{eq:apdtr_x}
\end{align}
\end{subequations}
\end{minipage}

First, consider the method defined in \eqref{eq:acv_main}. It is worth noting that if we swap the order of $x$- and $y$-updates (defined in \eqref{eq:acv_main_x} and \eqref{eq:acv_main_y}, respectively), we get another method, but we need to add extrapolation step in $y$-update. Such manipulations allow us to derive Accelerated Condat--V\~u Algorithm: Form I \& II (\algname{ACV-I} \& \algname{ACV-II}), which are given in Algorithm~\ref{alg:ACV}.

\begin{algorithm}[ht!]
\caption{Accelerated Condat--V\~u Algorithm: Form I \& II (\algname{ACV-I} \& \algname{ACV-II}) }
\label{alg:ACV}
    \begin{tabular}{@{}p{0.47\linewidth} | p{0.49\linewidth}@{}}
        \vspace{-0.75em} 
        \begin{algorithmic}[1]
        \State \textbf{Input:} Stepsizes $\eta_x, \eta_y, \eta_z > 0$.
        \For{$k = 0, 1, 2, \dots$}
            \State $x^{k+1} = \prox_{\eta_x {\mygreen g}} \left(x^k - \eta_x \mathsf{K}^* y^k - \eta_x \nabla {\myblue f}(z^k)\right)$
            \State $y^{k+1} = \prox_{\eta_y {\hstar}} \left(y^k + \eta_y \mathsf{K}(2x^{k+1} - x^k)\right)$
            \State $z^{k+1} = \frac{1}{1+\eta_z} z^k + \frac{\eta_z}{1+\eta_z}  \left(2x^{k+1} - x^k\right)$
        \EndFor
        \end{algorithmic}
        &
        \vspace{-0.75em}
        \begin{algorithmic}[1]
        \State \textbf{Input:} Stepsizes $\eta_x, \eta_y, \eta_z > 0$.
        \For{$k = 0, 1, 2, \dots$}
            \State $y^{k+1} = \prox_{\eta_y \hstar } \left(y^k + \eta_y \mathsf{K} x^k\right)$
            \State $\begin{aligned}[t]
                x^{k+1} = \prox_{\eta_x \mygreen g} \big(x^k - \eta_x \mathsf{K}^*(2&y^{k+1} - y^k) \\
                &- \eta_x \nabla {\myblue f}(z^k)\big)
            \end{aligned}$
            \State $z^{k+1} = \frac{1}{1+\eta_z} z^k + \frac{\eta_z}{1+\eta_z}  \left(2x^{k+1} - x^k\right)$
        \EndFor
        \end{algorithmic}\end{tabular}%
\end{algorithm}

Algorithm~\ref{alg:ACV} is indeed an accelerated version of Condat--V\~u Algorithm Forms I and II (\algname{CV-I} and \algname{CV-II}) 
\citep{condat2023proximal}. To demonstrate this exactly, we set $z^0 = x^0$ and $\eta_z = 1$ as we did for \algname{APGD} Algorithm~\ref{alg:APGD_APGE}. Then we have 
$$
    z^{k+1} = z^k - (z^{k+1} - (2x^{k+1}-x^k)) \quad \Rightarrow \quad 2z^{k+1}-z^k = 2x^{k+1} - x^k.
$$
By induction starting from $z^0 = x^0$, we have $z^k = x^k$ for all $k \geq 0$. Thus, the update rule for $x^{k+1}$ becomes
\begin{eqnarray*}
    \text{\algname{CV-I}}: && x^{k+1} = \prox_{\eta_x {\mygreen g}}\left(x^k- \eta_x \mathsf{K}^*y^k -\eta_x\nabla {\myblue f} (x^k)\right);\\
    \text{\algname{CV-II}}: && x^{k+1} = \prox_{\eta_x {\mygreen g}}\left(x^k- \eta_x \mathsf{K}^*(2y^{k+1}-y^k) -\eta_x\nabla {\myblue f} (x^k)\right),
\end{eqnarray*}
which are exactly the update rules for \algname{CV-I} and \algname{CV-II}.

Recently, \citet{driggs2024practical} proposed the Accelerated Condat--V\~u algorithm (\algname{ACV}), which is based on Nesterov's acceleration. The main differences between our methods and theirs are in the direct application of Nesterov's momentum and extrapolation step. In our methods the acceleration is derived from a primal--dual perspective, while \citet{driggs2024practical} incorporate Nesterov's momentum to \algname{CV} directly. Also, in our case we set the extrapolation parameter to $1$.

We now consider the method represented as  \eqref{eq:apdtr_main}. Again, similarly to the previous case, if we swap the $x$-sequence \eqref{eq:apdtr_x} and $y$-sequence \eqref{eq:apdtr_y}, we derive another form of the method. Thus, we have derived two new methods, which we call Accelerated Primal--Dual Twice Reflected Algorithms, Forms I and II (\algname{APDTR-I} and \algname{APDTR-II}).

\begin{algorithm}[ht!]
\caption{Accelerated Primal--Dual Twice Reflected Algorithm: Form I \& II (\algname{APDTR-I} \& \algname{APDTR-II}) }
\label{alg:APDTR}
\begin{tabular}{@{}p{0.48\linewidth} | p{0.48\linewidth}@{}}
        \vspace{-0.75em} 
        \begin{algorithmic}[1]
        \State \textbf{Input:} Stepsizes $\eta_x, \eta_y, \eta_z > 0$.
        \For{$k = 0, 1, 2, \dots$}
            \State $z^{k+1} = \frac{1}{1+\eta_z}z^k +\frac{\eta_z}{1+\eta_z} x^k$
            \State 
            $\begin{aligned}[t]
                x^{k+1} = \prox_{\eta_x {\mygreen g}} &\big(x^k - \eta_x \mathsf{K}^* y^k \\
                &-\eta_x(2\nabla {\myblue f}(z^{k+1}) - \nabla {\myblue f}(z^{k}))\big)
            \end{aligned}$
            \State $y^{k+1} = \prox_{\eta_y {\hstar}} \left(y^k + \eta_y \mathsf{K}(2x^{k+1} - x^k)\right)$
        \EndFor
        \end{algorithmic}
        &
        \vspace{-0.75em}
        \begin{algorithmic}[1]
		\State \textbf{Input:} Stepsizes $\eta_x, \eta_y, \eta_z > 0$.
		\For{ $k =  0, 1, 2, \ldots$}
		    \State $y^{k+1} = \prox_{\eta_y\hstar}\left(y^k +\eta_y \mathsf{K} x^k\right)$
            \State $z^{k+1} = \frac{1}{1+\eta_z}z^k +\frac{\eta_z}{1+\eta_z} x^k$
            \State 
            $\begin{aligned}[t]
                x^{k+1} = \prox_{\eta_x \mygreen g} \big(&x^k - \eta_x \mathsf{K}^*(2y^{k+1} - y^k) \\
                &- \eta_x(2\nabla {\myblue f}(z^{k+1})- \nabla {\myblue f}(z^{k}))\big)
            \end{aligned}$
		\EndFor
	\end{algorithmic}
    \end{tabular}	
\end{algorithm}

\citet{malitsky2026first} introduced the Primal--Dual Twice Reflected method (\algname{PDTR}) for monotone inclusion problems with three operators, which is a generalization of problem \eqref{eq:main_opt_prob}. Precisely, \algname{APDTR-I} is an accelerated version of \algname{PDTR} (in our notation, it is the Primal--Dual Twice Reflected Algorithm Form I (\algname{PDTR-I})).
To show  Algorithm~\ref{alg:APDTR} is indeed an accelerated version of Primal--Dual Twice Reflected Splitting Algorithm Forms I and II (\algname{PDTR-I} and \algname{PDTR-II}), we first set $\eta_z = \frac{1}{\alpha_z -1}$, where $\alpha_z > 1$ is a constant, and, if $\alpha_z \to 1$, then $\eta_z \to \infty$. This is a formal limiting relation showing algorithmic ancestry; it is not covered by the stepsize regime in the convergence theorems below. Then, we have the following update rule for $z^{k+1}$:
$$z^{k+1} = z^k - \eta_z(z^{k+1} - x^k) \quad \Rightarrow \quad z^{k+1} = \frac{1}{1+\eta_z}z^k + \frac{\eta_z}{1+\eta_z} x^k = \left(1-\frac{1}{\alpha_z}\right)z^k + \frac{1}{\alpha_z}x^k.$$
Therefore, if $\alpha_z \to 1$ and $z^0 = x^{-1}$, then $z^{k+1} \to x^k$ and 
\begin{eqnarray*}
    \text{\algname{PDTR-I}}: &&  x^{k+1} =\prox_{\eta_x {\mygreen g}}\left(x^k - \eta_x\mathsf{K}^*y^k - \eta_x\left(2\nabla {\myblue f} (x^k) - \nabla {\myblue f} (x^{k-1})\right)\right);\\
    \text{\algname{PDTR-II}}: && x^{k+1} =\prox_{\eta_x {\mygreen g}}\left(x^k - \eta_x\mathsf{K}^*(2y^{k+1}-y^k) - \eta_x\left(2\nabla {\myblue f} (x^k) - \nabla {\myblue f} (x^{k-1})\right)\right),
\end{eqnarray*}
which are exactly the update rules for \algname{PDTR-I} and \algname{PDTR-II}.

We now study convergence of four new accelerated methods: \algname{ACV-I} and  \algname{ACV-II}, \algname{APDTR-I} and  \algname{APDTR-II}. First, we introduce a Lyapunov function for those methods:
\begin{eqnarray*}
        \Psi_{k} &=& \frac{1}{2\eta_x}\|x^{k}-\xs\|^2 + \frac{1}{2\eta_y}\|y^{k}-\ys\|^2 + \frac{1}{\eta_z} D_{{\myblue f}}(z^{k};\zs)\\
        && \pm \la y^{k}-\ys, \mathsf{K}(x^{k}-\xs)\ra  \pm \la \nabla {\myblue f}(z^{k}) - \nabla{\myblue f}(\zs), x^{k} - \xs\ra,
    \end{eqnarray*}
    where: 
    \begin{itemize}
        \item For \algname{ACV-I}, the first and second $\pm$ signs are $-$, $-$, respectively;
        \item For \algname{ACV-II}, the first and second $\pm$ signs are $+$, $-$, respectively;
        \item For \algname{APDTR-I}, the first and second $\pm$ signs are $-$, $+$, respectively;
        \item For \algname{APDTR-II}, the first and second $\pm$ signs are $+$, $+$, respectively.
    \end{itemize}
The structure of the Lyapunov function is the same for all four methods, the only difference being these two signs. 

\subsection{Smooth Regime} 
We start with the case where ${\myred h}$ is $\nicefrac{1}{\mu_{\hstar}}$-smooth for some $\mu_{\hstar}>0$ (by Lemma~\ref{lem:strongly_convex_conjugate_smoothness},~$\hstar$ is $\mu_{\hstar}$-strongly convex). Under strong convexity of ${\mygreen g}$, 
we provide linear convergence guarantees for \algname{ACV-I} and  \algname{ACV-II}, \algname{APDTR-I} and  \algname{APDTR-II}.

\begin{theorem}
\label{th:convergence_ACV_and_APDTR}
    Let $(\xs,\ys,\zs)$ be a solution of \eqref{eq:main_optimality_condition}, and suppose that ${\mygreen g}$ is $\mu_{{\mygreen g}}$-strongly convex for some $\mu_{{\mygreen g}}\ge 0$ and $\hstar$ is $\mu_{\hstar}$-strongly convex for some $\mu_{\hstar}>0$.
    Then if $\|\mathsf{K}\|^2\eta_x\eta_y +L_{{\myblue f}}\eta_x\eta_z \leq 1$, the iterates generated by  Algorithm~\ref{alg:ACV} or  Algorithm~\ref{alg:APDTR} satisfy
    \begin{equation}
        \Psi_k \leq \theta^k \Psi_0,\quad \text{where } \theta = \max\left\{\frac{1}{1+\mu_{{\mygreen g}}\eta_x}, \frac{1}{1+\mu_{\hstar}\eta_y}, \frac{2}{2+\eta_z}\right\}.
    \end{equation}
\end{theorem}
The proof of Theorem~\ref{th:convergence_ACV_and_APDTR} can be found in Appendix~\ref{sec:proof_of_th_convergence_acv_and_apdtr}. 

We now  state the iteration complexity of Algorithm~\ref{alg:ACV} and  Algorithm~\ref{alg:APDTR}.
\begin{corollary}
    \label{cor:acv_and_apdtr_complexity}
    Under the assumptions of Theorem~\ref{th:convergence_ACV_and_APDTR}, suppose additionally that $\mu_{\mygreen g} > 0$.
    Then the solution $(\xs,\ys,\zs)$ of \eqref{eq:main_optimality_condition} is unique, and for any $\varepsilon >0$ the number of iterations to achieve $\Psi_k \leq \varepsilon \Psi_0$ for  Algorithm~\ref{alg:ACV} and  Algorithm~\ref{alg:APDTR} with the parameter selection 
    \begin{equation*}
        \eta_x = \min\left\{\frac{1}{\sqrt{L_{\myblue f}\mu_{\mygreen g}}}, \sqrt{\frac{\mu_{\hstar}}{\|\mathsf{K}\|^2\mu_{\mygreen g}}}\right\},
        \quad \eta_y = \frac{1}{2}\sqrt{\frac{\mu_{\mygreen g}}{\|\mathsf{K}\|^2\mu_{\hstar}}},
        \quad \eta_z  = \frac{1}{2}\sqrt{\frac{\mu_{\mygreen g}}{L_{\myblue f}}},
    \end{equation*}
    is equal to 
    \begin{equation*}
        \cO\left(\left(\sqrt{\frac{L_{\myblue f}}{\mu_{\mygreen g}}} + \sqrt{\frac{\|\mathsf{K}\|^2}{\mu_{\mygreen g}\mu_{\hstar}}}\right)\log\frac{1}{\varepsilon} \right).
    \end{equation*}
\end{corollary}
The proof of Corollary~\ref{cor:acv_and_apdtr_complexity} can be found in Appendix~\ref{sec:proof_of_cor_acv_and_apdtr_complexity}.

As we can see in Corollary~\ref{cor:acv_and_apdtr_complexity}, the iteration complexity of all four algorithms is the same and has two accelerated terms. The first term corresponds to the accelerated term with respect to ${\myblue f}$, the second term is the primal--dual coupling. 
Under the standard first-order oracle model for smooth strongly convex-concave bilinear saddle-point problems, this dependence matches known lower bounds up to constants and logarithmic factors, i.e.,
\begin{equation}
    \Omega\left(\left(\sqrt{\frac{L_{\myblue f}}{\mu_{\mygreen g}}} + \sqrt{\frac{\|\mathsf{K}\|^2}{\mu_{\mygreen g}\mu_{\hstar}}}\right)\log\frac{1}{\varepsilon} \right) \notag
\end{equation}
\citep{borodich2025linear}. Each of the terms was proved by \cite{nesterov2013introductory} and \cite{ibrahim2020linear}, respectively. Also, the derived complexity in Corollary~\ref{cor:acv_and_apdtr_complexity}  matches the iteration complexity of \algname{ACV} \citep{driggs2024practical}.

\subsection{Nonsmooth Regime} 
Now, we consider the case when ${\myred h}$ is nonsmooth. In this case, we cannot derive the same iteration complexity as in the smooth case. To see this,
we look at the dual problem to the original problem \eqref{eq:main_opt_prob}:
\begin{equation}
    \label{eq:dual_problem}
    \max_{y \in \R^d_{y}}\left\{ - ({\myblue f} + {\mygreen g})^*(-\mathsf{K}^* y) - \hstar(y)\right\}\quad \Leftrightarrow\quad \min_{y \in \R^d_{y}}\left\{ ({\myblue f} + {\mygreen g})^*(-\mathsf{K}^* y) + \hstar(y)\right\}
\end{equation}

When $\mu_{\mygreen g}>0$ and $\mu_{\hstar}>0$, both the primal and dual problems are strongly convex, which makes it possible for algorithms to converge linearly, as we have established for Algorithm~\ref{alg:ACV} and Algorithm~\ref{alg:APDTR} in Theorem~\ref{th:convergence_ACV_and_APDTR}. When ${\myred h}$ is nonsmooth, 
we need additional conditions on  ${\myblue f}, {\mygreen g}$ and $\mathsf{K}$ to make linear convergence possible. 
Specifically, we want $({\myblue f} + {\mygreen g})^*( - \mathsf{K}^* \cdot)$ to be smooth and strongly convex. For example, if ${\mygreen g} = 0$ (as assumed in several studies \citep{kovalev2020optimal,salim2022optimal,condat2023randprox}), then the assumption $\lambda_{\min}(\mathsf{K}\mathsf{K}^*)>0$ guarantees strong convexity of problem \eqref{eq:dual_problem}. However, when ${\mygreen g} \neq 0$, we have to additionally require smoothness of ${\mygreen g}$. For instance, the case when 
${\mygreen g} = \frac{\mu_{\mygreen g}}{2}\|\cdot\|_2^2$, considered recently in \citet{condat2026nesterov}, is a particular case of our framework.

\begin{theorem}
\label{th:convergence_ACV_and_APDTR_non_smooth}
    Let $(\xs,\ys,\zs)$ be a solution of \eqref{eq:main_optimality_condition}. Suppose that $\lambda_{\min}(\mathsf{K}\mathsf{K}^*) > 0$, ${\mygreen g}$ is  $L_{\mygreen g}$-smooth for some $L_{\mygreen g}>0$ and $\mu_{{\mygreen g}}$-strongly convex for some $\mu_{{\mygreen g}}\ge 0$. Then if $8\|\mathsf{K}\|^2\eta_x\eta_y \leq 1$ and $8L_{{\myblue f}}\eta_x\eta_z \leq 1$, the iterates of Algorithm~\ref{alg:ACV} or  Algorithm~\ref{alg:APDTR} satisfy
    \begin{equation}
        \Psi_k \leq \theta^k \Psi_0,\quad \text{where } \theta = \max\left\{\frac{2}{2+\mu_{{\mygreen g}}\eta_x}, \frac{40}{40+\lambda_{\min}(\mathsf{K}\mathsf{K}^*)\eta_x\eta_y}, \frac{20}{20+\frac{\lambda_{\min}(\mathsf{K}\mathsf{K}^*)}{(L_{\mygreen g}+ L_{\myblue f})}\eta_y}, \frac{2}{2+\eta_z}\right\}.
    \end{equation}
\end{theorem}
The proof of Theorem~\ref{th:convergence_ACV_and_APDTR_non_smooth} can be found in Appendix~\ref{sec:proof_convergence_acv_apdtr_non_smooth}.

\begin{corollary}
    \label{cor:acv_and_apdtr_complexity_non_smooth}
    Under assumptions of Theorem~\ref{th:convergence_ACV_and_APDTR_non_smooth}, suppose additionally that  $ L_{\mygreen g} \geq\mu_{\mygreen g} > 0$. Then the solution $(\xs,\ys,\zs)$ of \eqref{eq:main_optimality_condition} is unique, and for any $\varepsilon >0$ the number of iterations to achieve $\Psi_k \leq \varepsilon \Psi_0$ for Algorithm~\ref{alg:ACV} and Algorithm~\ref{alg:APDTR} solving problem \eqref{eq:main_opt_prob} with a parameter selection 
    \begin{equation*}
        \eta_x = \min\left\{\frac{1}{\sqrt{L_{\myblue f}\mu_{\mygreen g}}}, \sqrt{\frac{\lambda_{\min}(\mathsf{K}\mathsf{K}^*)}{\|\mathsf{K}\|^2}}\cdot \frac{1}{\sqrt{(L_{\mygreen g}+L_{\myblue f})\mu_{\mygreen g}}}\right\},
        \quad \eta_y = \frac{1}{8 \|\mathsf{K}\|^2 \eta_x},
        \quad \eta_z  = \frac{1}{8L_{\myblue f}\eta_x},
    \end{equation*}
    is equal to 
    \begin{equation*}
        \cO\left(\left(\sqrt{\frac{L_{\myblue f} + L_{\mygreen g}}{\mu_{\mygreen g}} } \cdot \frac{\|\mathsf{K}\|}{\sqrt{\lambda_{\min}(\mathsf{K}\mathsf{K}^*)}} + \frac{\|\mathsf{K}\|^2}{\lambda_{\min}(\mathsf{K}\mathsf{K}^*)}\right)\log\frac{1}{\varepsilon}\right).
    \end{equation*}
\end{corollary}
The proof of Corollary~\ref{cor:acv_and_apdtr_complexity_non_smooth} can be found in Appendix~\ref{sec:proof_of_cor_acv_and_apdtr_complexity_non_smooth}.

\subsection{Linearly Constrained Case} 
Let $b \in \mathrm{ran}(\mathsf{K})$, the range of $\mathsf{K}$. We focus in this section on linearly constrained minimization problems of the form 
\begin{equation}
    \label{eq:main_opt_prob_lc}
    \min_{x \in \R^{d_x}} {\myblue f}(x) + {\mygreen g}(x), \quad \text{s.t.}\quad \mathsf{K}x = b.
\end{equation}
This problem is a special case of \eqref{eq:main_opt_prob} with  ${\myred h}= \iota_{\{b\}}$, 
where $\iota_{\mathcal{C}}$ denotes the indicator function of a set $\mathcal{C}$, that is, $\iota_{\mathcal{C}}(z) =( 0$ if $z \in \mathcal{C};  +\infty$ otherwise$)$. In that case,  the conjugate of ${\myred h}$ is  $\hstar: y \in \R^{d_y} \mapsto \la y, b\ra$.

Thus, we can solve this problem with Algorithm~\ref{alg:ACV} and Algorithm~\ref{alg:APDTR}, whose update rules become:
\begin{align*}
    y^{k+1} &= \arg\min_{y \in \R^{d_y}} \left\{ \hstar(y) - \la \mathsf{K}x^k, y\ra + \frac{1}{2\eta_y}\|y- y^k\|_2^2\right\}\\
    &= \arg\min_{y \in \R^{d_y}} \left\{  - \la y,\mathsf{K}x^k -  b\ra + \frac{1}{2\eta_y}\|y- y^k \|_2^2\right\}\\
    &= \arg\min_{y \in \R^{d_y}} \left\{ \frac{1}{2\eta_y}\|y- (y^k + \eta_y(\mathsf{K}x^k -  b))\|_2^2\right\}.
\end{align*}
Therefore, if we set $y^0 \in \mathrm{ran}(\mathsf{K})$, then $y^k \in \mathrm{ran}(\mathsf{K})$ for all $k \geq 0$, and the update rule for $y$ in Algorithm~\ref{alg:ACV} and Algorithm~\ref{alg:APDTR} simplifies to:
\begin{align*}
    \text{\algname{ACV-I} and \algname{APDTR-I}:} &\quad  y^{k+1} = y^k + \eta_y(\mathsf{K}(2x^{k+1} - x^k) -  b);\\
    \text{\algname{ACV-II} and \algname{APDTR-II}:} &\quad  y^{k+1} = y^k + \eta_y(\mathsf{K}x^k -  b).
\end{align*}
To provide convergence guarantees for these algorithms, 
we introduce the following assumption:
\begin{equation}
    \label{ass:minimal_eigenvalue}
    \lambda_{\min}^+(\mathsf{K}\mathsf{K}^*) \eqdef \inf_{u \in \mathrm{ran}(\mathsf{K})\setminus\{0\}} \frac{\|\mathsf{K}^*u\|_2^2}{\|u\|_2^2} > 0.
\end{equation}

Besides, the optimality conditions of problem~\eqref{eq:main_opt_prob_lc} are:
\begin{equation}
    \label{eq:main_optimality_condition_lc}
    \begin{cases}
        0 \in \nabla {\mygreen g}(\xs) + \mathsf{K}^*\ys + u^{\star},\\
        \ys \in \mathrm{ran}(\mathsf{K}),\\
        \mathsf{K}\xs = b,\\
        0 \in \partial {\myblue f^*}(u^{\star}) - \xs;
    \end{cases}
    \quad \Leftrightarrow \quad 
    \begin{cases}
        0 \in \nabla {\mygreen g}(\xs) + \mathsf{K}^*\ys + \nabla {\myblue f} (\zs),\\
        \ys \in \mathrm{ran}(\mathsf{K}),\\
        \mathsf{K}\xs = b,\\
        0 = \xs - \zs.
    \end{cases}
\end{equation}
We restrict the saddle-point solutions to satisfy the condition
$\ys \in \mathrm{ran}(\mathsf{K})$.  Indeed, projecting any dual solution onto $\mathrm{ran}(\mathsf{K})$ leaves $\mathsf{K}^*\ys$ unchanged, and since $b \in \mathrm{ran}(\mathsf{K})$, it also leaves $\la \ys, b\ra$ unchanged.
The proof of Corollary~\ref{cor:acv_and_apdtr_complexity_non_smooth} then applies on $\mathrm{ran}(\mathsf{K})$ with $\lambda_{\min}(\mathsf{K}\mathsf{K}^*)$ replaced by $\lambda_{\min}^+(\mathsf{K}\mathsf{K}^*)$.

\begin{corollary}
    \label{cor:acv_and_apdtr_complexity_non_smooth_lin_const}
    Let $(\xs,\ys,\zs)$ be a solution of \eqref{eq:main_optimality_condition_lc}. Suppose that $\lambda^+_{\min}(\mathsf{K}\mathsf{K}^*) > 0$, ${\mygreen g}$ is $L_{\mygreen g}$-smooth and $\mu_{{\mygreen g}}$-strongly convex for some $L_{\mygreen g} \geq \mu_{\mygreen g} > 0$. Then, the solution $(\xs,\ys,\zs)$ of \eqref{eq:main_optimality_condition_lc} is unique. Moreover, for any $\varepsilon >0$, the number of iterations to achieve $\Psi_k \leq \varepsilon \Psi_0$ for Algorithm~\ref{alg:ACV} and Algorithm~\ref{alg:APDTR} starting from $ (x^0,y^0,z^0)$ such that $y^0 \in \mathrm{ran}(\mathsf{K}) $ with parameter selection 
    \begin{equation*}
        \eta_x = \min\left\{\frac{1}{\sqrt{L_{\myblue f}\mu_{\mygreen g}}}, \sqrt{\frac{\lambda^+_{\min}(\mathsf{K}\mathsf{K}^*)}{\|\mathsf{K}\|^2}}\cdot \frac{1}{\sqrt{(L_{\mygreen g}+L_{\myblue f})\mu_{\mygreen g}}}\right\},
        \quad \eta_y = \frac{1}{8 \|\mathsf{K}\|^2 \eta_x},
        \quad \eta_z  = \frac{1}{8 L_{\myblue f}\eta_x},
    \end{equation*}
    is equal to 
    \begin{equation*}
        \cO\left(\left(\sqrt{\frac{L_{\myblue f} + L_{\mygreen g}}{\mu_{\mygreen g}}} \cdot \frac{\|\mathsf{K}\|}{\sqrt{\lambda^+_{\min}(\mathsf{K}\mathsf{K}^*)}} + \frac{\|\mathsf{K}\|^2}{\lambda^+_{\min}(\mathsf{K}\mathsf{K}^*)}\right)\log\frac{1}{\varepsilon}\right).
    \end{equation*}
\end{corollary}

As we can observe, the derived complexities in Corollary~\ref{cor:acv_and_apdtr_complexity_non_smooth} and Corollary~\ref{cor:acv_and_apdtr_complexity_non_smooth_lin_const} have two terms. 
The second term could be removed by incorporating Chebyshev acceleration \citep{kovalev2020optimal,salim2022optimal}. 

\section{Conclusion and Future Work}

In this paper, we introduced a unified, primal--dual recipe for accelerating gradient-based splitting methods to solve composite optimization problems. By formulating the three-operator splitting problem from a saddle-point perspective and applying the Chambolle--Pock algorithm, we systematically derived four novel accelerated algorithms: \algname{ACV-I}, \algname{ACV-II}, \algname{APDTR-I}, and \algname{APDTR-II}. Through a simplified, unified Lyapunov-based analysis, we established optimal rates when ${\myred h}$ is smooth and accelerated rates in the nonsmooth and linearly constrained cases. Crucially, our framework achieves these accelerated rates without relying on restrictive structural assumptions, such as requiring one of the nonsmooth components to be zero, that have limited existing 
algorithms.

Future research could extend this unified recipe to stochastic settings, particularly analyzing its robustness and acceleration capabilities in the presence of heavy-tailed noise, or to decentralized optimization problems where the objective function is distributed across multiple agents \citep{kovalev2020optimal}. Moreover, we believe that combining the presented framework with the concept of randomized proximal operators \citep{condat2023randprox} could yield a deeper understanding of the fundamental nature of accelerated and variance-reduced optimization methods, ultimately paving the way for a unified framework. 

\section*{Acknowledgements}

This work was supported by funding from King Abdullah University of Science and Technology (KAUST): i) KAUST Baseline Research Scheme, ii) CRG Grant ORFS-CRG12-2024-6460, and iii) Center of Excellence for Generative AI, under award number 5940.

\bibliography{literature}
\clearpage
{\renewcommand\baselinestretch{0}
\tableofcontents
\renewcommand\baselinestretch{1}
}

\clearpage
\part*{Appendix}
\appendix

\section{Auxiliary Lemmas}

The first lemma states the three-point identity for the Bregman divergence, also known as the ``Magic Identity'' in the book of \citet{bental_nemirovski_2022_lmco}.
\begin{lemma}[Three-point Identity]
    \label{lem:magic_identity}
    Let $\mathcal{X} \subseteq \mathbb{R}^n$ be a convex, closed set, $\phi: \mathcal{X}\to \mathbb{R}\cup\{\infty\}$ be a differentiable convex function and $\nabla\phi$ its gradient on $\mathcal{X}$. Then for any $z^+, z, \zs \in \mathcal{X}$ 
    \begin{equation*}
        D_\phi(z^+;\zs) = D_\phi(z;\zs) + \la \nabla\phi(z^+)-\nabla\phi(\zs), z^+-z\ra - D_\phi(z;z^+),
    \end{equation*}
    where $D_\phi(x;y) = \phi(x) - \phi(y) - \la \nabla\phi(y), x-y\ra$ is the Bregman divergence (see Definition \ref{def:bregman_divergence}).
\end{lemma}

\begin{lemma}[\citep{nesterov2018lectures}]
    \label{lem:smoothness_bregman}
    Let $\phi$ be $L$-smooth and convex. Then, for any $x,y \in \mathbb{R}^n$ it holds 
    \begin{equation*}
        \frac{1}{2L}\|\nabla\phi(x) - \nabla\phi(y)\|^2\leq D_\phi(x;y) \leq \frac{L}{2}\|x-y\|^2.
    \end{equation*}
\end{lemma}

\begin{lemma}[\citep{orabona2019modern}]
    \label{lem:strong_convexity_bregman_non_diff}
    Let $\phi$ be $\mu$-strongly convex. Then, for any $x,y \in \mathbb{R}^n$ and any $p \in \partial \phi(x), q\in \partial \phi(y)$ it holds 
    \begin{equation*}
       \langle p-q, x-y\rangle \geq \mu \|x-y\|^2.
    \end{equation*}
\end{lemma}

\begin{lemma}[\citep{nesterov2018lectures}]
    \label{lem:strong_convexity_bregman}
    Let $\phi$ be differentiable and $\mu$-strongly convex with $\mu > 0$. Then, for any $x,y \in \mathbb{R}^n$ it holds 
    \begin{equation*}
        \frac{\mu}{2}\|x-y\|^2\leq D_\phi(x;y) \leq \frac{1}{2\mu}\|\nabla\phi(x) - \nabla\phi(y)\|^2.
    \end{equation*}
\end{lemma}

\section{Additional Proofs of Section \ref{sec:APGD_and_APGE} }
\label{sec:proof_APGD_and_APGE}

\begin{lemma}
    \label{lem:lyaponov_func_apgd_apge} 
    Define the Lyapunov function $\Phi_{k}$ for any integer $k \geq 0$ as 
    \begin{equation}
        \label{eq:lyaponov_func_apgd_apge}
        \Phi_{k} \eqdef \frac{1}{2\eta_x}\|x^{k}-\xs\|^2 +  \frac{1}{\eta_z} D_{{\myblue f}}(z^{k};\zs) \pm \la \nabla {\myblue f}(z^{k}) - \nabla{\myblue f}(\zs), x^{k} - \xs\ra,
    \end{equation}
    where the sign $\pm$ is '$-$' for \algname{APGD} and '$+$' for \algname{APGE}.
    Then, if $L_{\myblue f} \eta_x\eta_z \leq 1$ then it holds that
    \begin{equation}
        \label{eq:lyaponov_func_apgd_apge_1}
        0 \leq \Phi_{k} \leq \frac{1}{\eta_x}\|x^{k}-\xs\|^2 +  \frac{2}{\eta_z} D_{{\myblue f}}(z^{k};\zs).
    \end{equation}
\end{lemma}
\begin{proof}
    By Fenchel-Young's inequality and Lemma~\ref{lem:smoothness_bregman} one has
    \begin{eqnarray*}
    \Phi_{k} &\geq& \frac{1}{2\eta_x}\|x^{k}-\xs\|^2 +  \frac{1}{\eta_z} D_{{\myblue f}}(z^{k};\zs) - \frac{\eta_x}{2}\|\nabla {\myblue f}(z^{k}) - \nabla{\myblue f}(\zs)\|^2 -\frac{1}{2\eta_x}\|x^{k} - \xs\|^2\\
    &\geq& \frac{1}{2\eta_x}\|x^{k}-\xs\|^2 +  \frac{1}{\eta_z} D_{{\myblue f}}(z^{k};\zs) -L_{\myblue f}\eta_x D_{\myblue f}(z^{k}; \zs)  - \frac{1}{2\eta_x}\|x^{k} - \xs\|^2.
\end{eqnarray*}
Then, because $L_{\myblue f} \eta_x\eta_z \leq 1$ , $\Phi_{k}$ is nonnegative for any integer $k \geq 0$.
By the same argument with opposite sign for bounding the inner product term one can prove that for any integer $k \geq 0$, it holds that
\begin{equation*}
    \Phi_k \leq \frac{1}{\eta_x}\|x^k-\xs\|^2 + \frac{2}{\eta_z} D_{\myblue f}(z^k; \zs).
\end{equation*}
This completes the proof.

\end{proof}

\subsection{Descent Lemmas}
First, we derive the descent lemma for the primal update of \algname{APGD} (see Line 3 (left) of Algorithm~\ref{alg:APGD_APGE}). 
\begin{lemma}
    \label{lem:descent_lemma_apgd_x}
    Let $(\xs,\zs)$ be a solution of \eqref{eq:main_prob_apgd_optimality_condition}, and assume  ${\mygreen g}$ is $\mu_{\mygreen g}$-strongly convex for some $\mu_{\mygreen g}\ge 0$.
    Then, for every $k \geq 0$, the iterates of Algorithm~\ref{alg:APGD_APGE} satisfy 
    \begin{eqnarray}
        \frac{1}{2\eta_x}\|x^{k+1}-\xs\|^2 &\leq& \frac{1}{2\eta_x}\|x^{k}-\xs\|^2 - \mu_{\mygreen g}\|x^{k+1}-\xs\|^2 - \frac{1}{2\eta_x}\|x^{k+1}-x^k\|^2 \notag\\
        && -  \la \nabla {\myblue f} (z^k) - \nabla {\myblue f} (\zs), x^{k}-\xs\ra  -  \la \nabla {\myblue f} (z^k) - \nabla {\myblue f} (\zs), x^{k+1}-x^k\ra. \label{eq:descent_lemma_apgd_x}
    \end{eqnarray}
\end{lemma}
\begin{proof}
    According to Line $3$ (left) of Algorithm~\ref{alg:APGD_APGE},
    \begin{equation*}
        x^{k+1} = \prox_{\eta_x {\mygreen g}}(x^k - \eta_x \nabla {\myblue f}(z^k)) \quad \Rightarrow \quad x^{k+1} = x^k -\eta_x\left(\partial {\mygreen g}(x^{k+1}) + \nabla {\myblue f}(z^k) \right).\footnote{Here and throughout, we slightly abuse notation by treating the subgradient $\partial {\mygreen g}(x^{k+1})$ as a single vector rather than a set of vectors. When we write this equality, we mean that there exists a valid subderivative in this set such that the equality holds.}
    \end{equation*}
    Using this, we have
    \begin{eqnarray}
        \frac{1}{2\eta_x}\|x^{k+1}-\xs\|^2 &=& \frac{1}{2\eta_x}\|x^{k}-\xs\|^2 + \frac{1}{\eta_x}\la x^{k+1}-x^k, x^{k+1}-\xs\ra  - \frac{1}{2\eta_x}\|x^{k+1}-x^k\|^2 \notag\\
        &=& \frac{1}{2\eta_x}\|x^{k}-\xs\|^2 - \la \partial{\mygreen g}(x^{k+1}) + \nabla {\myblue f}(z^k), x^{k+1}-\xs\ra  - \frac{1}{2\eta_x}\|x^{k+1}-x^k\|^2.\label{eq:descent_lemma_apgd_x_proof_1}
    \end{eqnarray}
    Plugging  the optimality condition $\partial{\mygreen g}(\xs) + \nabla {\myblue f}(\zs) = 0$ into \eqref{eq:descent_lemma_apgd_x_proof_1}, we obtain
    \begin{eqnarray*}
        \frac{1}{2\eta_x}\|x^{k+1}-\xs\|^2 &=& \frac{1}{2\eta_x}\|x^{k}-\xs\|^2 - \la \partial{\mygreen g}(x^{k+1}) - \partial{\mygreen g}(\xs), x^{k+1}-\xs\ra  - \frac{1}{2\eta_x}\|x^{k+1}-x^k\|^2\\
        && - \la \nabla {\myblue f} (z^k) - \nabla {\myblue f} (\zs), x^{k+1}-\xs\ra.
    \end{eqnarray*}
    Using Lemma~\ref{lem:strong_convexity_bregman_non_diff}, we get 
    \begin{eqnarray}
        \frac{1}{2\eta_x}\|x^{k+1}-\xs\|^2 &\leq& \frac{1}{2\eta_x}\|x^{k}-\xs\|^2 -  \mu_{{\mygreen g}}\|x^{k+1}-\xs\|^2  - \frac{1}{2\eta_x}\|x^{k+1}-x^k\|^2 \notag\\
        && - \la \nabla {\myblue f} (z^k) - \nabla {\myblue f} (\zs), x^{k+1}-\xs\ra.\label{eq:proof_thm_conv_apgd_2}
    \end{eqnarray}
    Rewriting \eqref{eq:proof_thm_conv_apgd_2}, we obtain \eqref{eq:descent_lemma_apgd_x}.
\end{proof}
We now show the descent lemma for the $z$-update.
\begin{lemma}
    \label{lem:descent_lemma_apgd_z}
    Let $(\xs,\ys,\zs)$ be a solution of \eqref{eq:main_optimality_condition}. 
    Assume the sequence $\{z^k\}_{k\geq 0}$ is generated by the  update rule
    \begin{equation*}
        z^{k+1}= z^k -\eta_z\left(z^{k+1}- (2x^{k+1}-x^k)\right).\footnote{For example, see Line $4$ (left) of Algorithm~\ref{alg:APGD_APGE}.}
    \end{equation*}
    Then, for any constant $a_z > 0$ and any integer  $k \geq 0$, it holds 
    \begin{eqnarray}
        \frac{1}{\eta_z} D_{\myblue f}(z^{k+1};\zs) &\leq&\frac{1}{\eta_z} D_{\myblue f}(z^k;\zs) - D_{\myblue f}(z^{k+1};\zs) + \la \nabla {\myblue f}(z^{k+1}) - \nabla{\myblue f}(\zs), x^{k+1} - \xs\ra - \frac{1-a_z}{\eta_z} D_{\myblue f}(z^k;z^{k+1})\notag\\
        && - D_{\myblue f}(\zs;z^{k+1})+ \la \nabla {\myblue f}(z^{k}) - \nabla{\myblue f}(\zs), x^{k+1} - x^k\ra + L_{\myblue f}\eta_x\eta_z  \frac{1}{2a_z\eta_x}\|x^{k+1}-x^k\|^2. \label{eq:descent_lemma_apgd_z}
    \end{eqnarray}
\end{lemma}
\begin{proof}
    By Lemma~\ref{lem:magic_identity}, we have
    \begin{eqnarray}
        \frac{1}{\eta_z} D_{{\myblue f}}(z^{k+1};\zs) &=& \frac{1}{\eta_z} D_{{\myblue f}}(z^{k};\zs) + \frac{1}{\eta_z}\la \nabla {\myblue f}(z^{k+1}) - \nabla{\myblue f}(\zs), z^{k+1} - z^{k}\ra - \frac{1}{\eta_z} D_{{\myblue f}}(z^{k};z^{k+1})\notag\\
    &=& \frac{1}{\eta_z} D_{{\myblue f}}(z^{k};\zs) - \frac{1}{\eta_z}\la \nabla {\myblue f}(z^{k+1}) - \nabla{\myblue f}(\zs), \eta_z\left(z^{k+1}-(2x^{k+1}-x^k)\right)\ra\notag\\
    && - \frac{1}{\eta_z} D_{{\myblue f}}(z^{k};z^{k+1}), \label{eq:descent_lemma_apgd_z_proof_1}
    \end{eqnarray}
    where in the last equation we used the update rule for $z^{k+1}= z^k -\eta_z\left(z^{k+1}- (2x^{k+1}-x^k)\right)$.
    Plugging in the optimality condition $\xs - \zs = 0$ in \eqref{eq:descent_lemma_apgd_z_proof_1}, we have 
\begin{eqnarray}
    \frac{1}{\eta_z} D_{{\myblue f}}(z^{k+1};\zs) &=& \frac{1}{\eta_z} D_{{\myblue f}}(z^{k};\zs) - \la \nabla {\myblue f}(z^{k+1}) - \nabla{\myblue f}(\zs), z^{k+1} -\zs \ra - \frac{1}{\eta_z} D_{{\myblue f}}(z^{k};z^{k+1}) \notag\\
    && + \la \nabla {\myblue f}(z^{k+1}) - \nabla{\myblue f}(\zs), x^{k+1}-x^k \ra + \la \nabla {\myblue f}(z^{k+1}) - \nabla{\myblue f}(\zs), x^{k+1} - \xs\ra \notag\\
    &=& \frac{1}{\eta_z} D_{{\myblue f}}(z^{k};\zs) - D_{{\myblue f}}(z^{k+1};\zs) -  D_{{\myblue f}}(\zs; z^{k+1})\notag\\
    && + \la \nabla {\myblue f}(z^{k+1}) - \nabla{\myblue f}(\zs), x^{k+1} - \xs\ra - \frac{1}{\eta_z} D_{{\myblue f}}(z^{k};z^{k+1})\notag\\
    && + \underbrace{\la \nabla {\myblue f}(z^{k+1}) - \nabla{\myblue f}(z^k), x^{k+1}-x^k \ra}_{\eqdef \cT_1}  + \la \nabla {\myblue f}(z^{k}) - \nabla{\myblue f}(\zs), x^{k+1}-x^k \ra. \label{eq:descent_lemma_apgd_z_proof_2}
\end{eqnarray}
To continue our proof we need to bound the term $\cT_1$ from \eqref{eq:descent_lemma_apgd_z_proof_2}. By Fenchel-Young's inequality, 
\begin{eqnarray*}
    \cT_1 &\leq& \frac{C}{2}\|\nabla {\myblue f}(z^{k+1}) - \nabla{\myblue f}(z^k)\|^2 + \frac{1}{2C}\|x^{k+1}-x^k \|^2.
\end{eqnarray*}
Since $\myblue f$ is $L_{\myblue f}$-smooth and convex, we apply Lemma~\ref{lem:smoothness_bregman} to derive 
\begin{eqnarray}
   \cT_1
    &\leq& CL_{{\myblue f}}D_{{\myblue f}}(z^k; z^{k+1}) + \frac{1}{2C}\|x^{k+1}-x^k \|^2 \notag\\
    &\overset{C = \nicefrac{a_z}{L_{\myblue f}\eta_z}}{=}& \frac{a_z}{\eta_z}D_{{\myblue f}}(z^k; z^{k+1}) + L_{{\myblue f}}\eta_x\eta_z\frac{1}{2a_z\eta_x}\|x^{k+1}-x^k \|^2. \label{eq:descent_lemma_apgd_z_proof_3}
\end{eqnarray}
Plugging in \eqref{eq:descent_lemma_apgd_z_proof_3} into \eqref{eq:descent_lemma_apgd_z_proof_2}, we obtain \eqref{eq:descent_lemma_apgd_z}, which concludes proof.
\end{proof}

Next, we state descent lemmas for the updates of \algname{APGE} (see Algorithm~\ref{alg:APGD_APGE}). We start with the primal variables.
\begin{lemma}
    \label{lem:descent_lemma_apge_x}
    Let $(\xs,\zs)$ be a solution of \eqref{eq:main_prob_apgd_optimality_condition}, and assume ${\mygreen g}$ is $\mu_{\mygreen g}$-strongly convex for some $\mu_{\mygreen g}\ge 0$.
    Then, for every $k\geq 0$, the iterates generated by \algname{APGE} (Algorithm~\ref{alg:APGD_APGE}) satisfy 
    \begin{eqnarray}
        \frac{1}{2\eta_x}\|x^{k+1}-\xs\|^2 &\leq& \frac{1}{2\eta_x}\|x^{k}-\xs\|^2 -  \mu_{{\mygreen g}}\|x^{k+1}-\xs\|^2  - \frac{1}{2\eta_x}\|x^{k+1}-x^k\|^2 \notag\\
        && - \la \nabla {\myblue f} (z^{k+1}) - \nabla {\myblue f} (\zs), x^{k+1}-\xs\ra  - \la \nabla {\myblue f}(z^{k+1}) - \nabla {\myblue f}(z^k), x^{k+1}-\xs\ra. \label{eq:descent_lemma_apge_x}
    \end{eqnarray}
\end{lemma}
\begin{proof}
    According to Line $4$ (right) of Algorithm~\ref{alg:APGD_APGE}, we have 
    \begin{equation*}
        x^{k+1} = \prox_{\eta_x {\mygreen g}}\left(x^k - \eta_x(2\nabla {\myblue f}(z^{k+1}) - \nabla {\myblue f}(z^k))\right) = x^k -\eta_x\left(\partial{\mygreen g}(x^{k+1}) + (2\nabla {\myblue f}(z^{k+1}) - \nabla {\myblue f}(z^k))\right).
    \end{equation*}
    Using this, we obtain
    \begin{eqnarray}
        \frac{1}{2\eta_x}\|x^{k+1}-\xs\|^2 &=& \frac{1}{2\eta_x}\|x^{k}-\xs\|^2 + \frac{1}{\eta_x}\la x^{k+1}-x^k, x^{k+1}-\xs\ra  - \frac{1}{2\eta_x}\|x^{k+1}-x^k\|^2 \notag\\
    &=& \frac{1}{2\eta_x}\|x^{k}-\xs\|^2 - \la \partial{\mygreen g}(x^{k+1}) + \nabla {\myblue f}(z^{k+1}), x^{k+1}-\xs\ra  - \frac{1}{2\eta_x}\|x^{k+1}-x^k\|^2 \notag\\
    && - \la \nabla {\myblue f}(z^{k+1}) - \nabla {\myblue f}(z^k), x^{k+1}-\xs\ra.\label{eq:descent_lemma_apge_x_proof_1}
\end{eqnarray}
Plugging in optimality conditions $\partial{\mygreen g}(\xs) + \nabla {\myblue f}(\zs) = 0$, we derive 
\begin{eqnarray*}
    \frac{1}{2\eta_x}\|x^{k+1}-\xs\|^2 &=& \frac{1}{2\eta_x}\|x^{k}-\xs\|^2 - \la \partial{\mygreen g}(x^{k+1}) - \partial{\mygreen g}(\xs), x^{k+1}-\xs\ra  - \frac{1}{2\eta_x}\|x^{k+1}-x^k\|^2\\
    && - \la \nabla {\myblue f} (z^{k+1}) - \nabla {\myblue f} (\zs), x^{k+1}-\xs\ra - \la \nabla {\myblue f}(z^{k+1}) - \nabla {\myblue f}(z^k), x^{k+1}-\xs\ra.
\end{eqnarray*}
Applying Lemma~\ref{lem:strong_convexity_bregman_non_diff}, we derive \eqref{eq:descent_lemma_apge_x}.
\end{proof}

Finally, we present the descent lemma for $z$-update.
\begin{lemma}
    \label{lem:descent_lemma_apge_z}
    Let $(\xs,\ys,\zs)$ be a solution of \eqref{eq:main_prob_apgd_optimality_condition}. 
   Assume the sequence $\{z^k\}_{k\geq 0}$ is generated by the  update rule
    \begin{equation*}
        z^{k+1}= z^k -\eta_z\left(z^{k+1}- x^k\right). \footnote{For example, see Line $3$ (right) of Algorithm~\ref{alg:APGD_APGE}.}
    \end{equation*}
    Then, for any constant $a_z > 0$ and any integer $k \geq 0$, it holds
    \begin{eqnarray}
        \frac{1}{\eta_z} D_{\myblue f}(z^{k+1};\zs) &\leq&\frac{1}{\eta_z} D_{\myblue f}(z^k;\zs) - D_{\myblue f}(z^{k+1};\zs) - \frac{1-a_z}{\eta_z}D_{\myblue f}(z^k;z^{k+1})\notag\\
        &&  + \la \nabla {\myblue f}(z^{k}) - \nabla{\myblue f}(\zs), x^{k} - \xs\ra + \la \nabla {\myblue f}(z^{k+1}) - \nabla{\myblue f}(z^k), x^{k+1} - \xs\ra \notag\\
        &&- D_{\myblue f}(\zs; z^{k+1})+ L_{\myblue f}\eta_x\eta_z  \frac{1}{2a_z\eta_x}\|x^{k+1}-x^k\|^2. \label{eq:descent_lemma_apge_z}
    \end{eqnarray}
\end{lemma}
\begin{proof}
By Lemma~\ref{lem:magic_identity}, we have 
\begin{eqnarray}
    \frac{1}{\eta_z} D_{{\myblue f}}(z^{k+1};\zs) &=& \frac{1}{\eta_z} D_{{\myblue f}}(z^{k};\zs) + \frac{1}{\eta_z}\la \nabla {\myblue f}(z^{k+1}) - \nabla{\myblue f}(\zs), z^{k+1} - z^{k}\ra - \frac{1}{\eta_z} D_{{\myblue f}}(z^{k};z^{k+1})\notag\\
    &=& \frac{1}{\eta_z} D_{{\myblue f}}(z^{k};\zs) - \frac{1}{\eta_z}\la \nabla {\myblue f}(z^{k+1}) - \nabla{\myblue f}(\zs), \eta_z(z^{k+1}- x^{k})\ra - \frac{1}{\eta_z} D_{{\myblue f}}(z^{k};z^{k+1}), \label{eq:descent_lemma_apge_z_proof_1}
\end{eqnarray}
where in the last equation we used the update rule for $z^{k+1}$.
Plugging in the optimality condition $\xs - \zs = 0$ in \eqref{eq:descent_lemma_apge_z_proof_1}, we have 
\begin{eqnarray}
    \frac{1}{\eta_z} D_{{\myblue f}}(z^{k+1};\zs) &=& \frac{1}{\eta_z} D_{{\myblue f}}(z^{k};\zs) - \la \nabla {\myblue f}(z^{k+1}) - \nabla{\myblue f}(\zs), z^{k+1} -\zs \ra - \frac{1}{\eta_z} D_{{\myblue f}}(z^{k};z^{k+1}) \notag\\
    && + \la \nabla {\myblue f}(z^{k+1}) - \nabla{\myblue f}(\zs), x^{k} - \xs\ra \notag\\
    &=& \frac{1}{\eta_z} D_{{\myblue f}}(z^{k};\zs) - D_{{\myblue f}}(z^{k+1};\zs) -D_{{\myblue f}}(\zs;z^{k+1}) - \frac{1}{\eta_z} D_{{\myblue f}}(z^{k};z^{k+1})\notag\\
    && + \la \nabla {\myblue f}(z^{k}) - \nabla{\myblue f}(\zs), x^{k} - \xs\ra +  \la \nabla {\myblue f}(z^{k+1}) - \nabla {\myblue f}(z^{k}), x^{k+1} - \xs\ra\notag\\
    && - \underbrace{\la \nabla {\myblue f}(z^{k+1}) - \nabla{\myblue f}(z^k), x^{k+1}-x^k\ra}_{= \cT_1}. \label{eq:descent_lemma_apge_z_proof_2}
\end{eqnarray}
Plugging in~\eqref{eq:descent_lemma_apgd_z_proof_3} into~\eqref{eq:descent_lemma_apge_z_proof_2}, we obtain \eqref{eq:descent_lemma_apge_z}.
\end{proof}

\subsection{Proof of Theorem~\ref{th:apgd_apge_conv}}
\label{sec:proof_APGD_APGE}

\paragraph{Convergence rate of \algname{APGD}.} Recall the Lyapunov function for \algname{APGD} is defined as (see Lemma~\ref{lem:lyaponov_func_apgd_apge})
\begin{equation*}
    \Phi_{k} \eqdef \frac{1}{2\eta_x}\|x^{k}-\xs\|^2 +  \frac{1}{\eta_z} D_{{\myblue f}}(z^{k};\zs) - \la \nabla {\myblue f}(z^{k}) - \nabla{\myblue f}(\zs), x^{k} - \xs\ra.
\end{equation*}
Applying Lemmas~\ref{lem:descent_lemma_apgd_x} and \ref{lem:descent_lemma_apgd_z} with  $a_z = 1$, we obtain 
\begin{eqnarray*}
    \Phi_{k+1}  &\leq&  \frac{1}{2\eta_x}\|x^{k}-\xs\|^2   - \frac{1}{2\eta_x}\|x^{k+1}-x^k\|^2 - \mu_{\mygreen g}\|x^{k+1} - \xs\|^2\\
    && -  \la \nabla {\myblue f} (z^k) - \nabla {\myblue f} (\zs), x^{k}-\xs\ra -  \la \nabla {\myblue f} (z^k) - \nabla {\myblue f} (\zs), x^{k+1}-x^k\ra\\
    && + \frac{1}{\eta_z} D_{{\myblue f}}(z^{k};\zs) -D_{\myblue f}(z^{k+1};\zs) + L_{\myblue f}\eta_x\eta_z \frac{1}{2\eta_x}\|x^{k+1}-x^k \|^2 \\
    && + \la \nabla {\myblue f}(z^{k}) - \nabla{\myblue f}(\zs), x^{k+1}-x^k \ra \\
    &\leq& \Phi_k - \mu_{\mygreen g}\|x^{k+1} - \xs\|^2 - D_{\myblue f}(z^{k+1};\zs),
\end{eqnarray*}
where in the last inequality we used $L_{\myblue f}\eta_x\eta_z \leq 1$.

\paragraph{Convergence rate of \algname{APGE}.} 
Recall the Lyapunov function for \algname{APGE} is defined as (see Lemma~\ref{lem:lyaponov_func_apgd_apge})
\begin{equation*}
    \Phi_{k+1} \eqdef \frac{1}{2\eta_x}\|x^{k+1}-\xs\|^2 +  \frac{1}{\eta_z} D_{{\myblue f}}(z^{k+1};\zs) + \la \nabla {\myblue f}(z^{k+1}) - \nabla{\myblue f}(\zs), x^{k+1} - \xs\ra.
\end{equation*}
Applying Lemmas~\ref{lem:descent_lemma_apge_x} and \ref{lem:descent_lemma_apge_z} with $a_z = 1$, we obtain
\begin{eqnarray*}
    \Phi_{k+1}  &\leq& \frac{1}{2\eta_x}\|x^{k}-\xs\|^2 - \mu_{\mygreen g}\|x^{k+1} - \xs\|^2  - \frac{1}{2\eta_x}\|x^{k+1}-x^k\|^2 - \la \nabla {\myblue f} (z^{k+1}) - \nabla {\myblue f} (z^k), x^{k+1}-\xs\ra  \\
    && + \frac{1}{\eta_z} D_{{\myblue f}}(z^{k};\zs) -D_{\myblue f}(z^{k+1};\zs) + L_{\myblue f}\eta_x\eta_z \frac{1}{2\eta_x}\|x^{k+1}-x^k \|^2\\
    &&  + \la \nabla {\myblue f}(z^{k}) - \nabla{\myblue f}(\zs), x^{k}-\xs \ra  + \la \nabla {\myblue f}(z^{k+1}) - \nabla{\myblue f}(z^k), x^{k+1}-\xs \ra \\
    &\leq& \Phi_k - \mu_{\mygreen g}\|x^{k+1} - \xs\|^2 - D_{\myblue f}(z^{k+1};\zs),
\end{eqnarray*}
where in the last inequality we used $L_{\myblue f}\eta_x\eta_z \leq 1$.

Finally, applying Lemma~\ref{lem:lyaponov_func_apgd_apge}, we get for both algorithms (\algname{APGD} and \algname{APGE}) the same relation
\begin{eqnarray*}
    \Phi_{k+1}  \leq  \Phi_k - \min\left\{\mu_{\mygreen g}\eta_x, \frac{\eta_z}{2}\right\}\Phi_{k+1} &\Rightarrow& \min\left\{1+\mu_{\mygreen g}\eta_x, \frac{2+ \eta_z}{2}\right\}\Phi_{k+1} \leq \Phi_k.
\end{eqnarray*}

Thus, denoting $\theta = \max\left\{\frac{1}{1+\mu_{{\mygreen g}}\eta_x}, \frac{2}{2+\eta_z}\right\}$, we have 
\begin{eqnarray}
    \label{eq:proof_thm_conv_apge_8}
    \Phi_{k+1} &\leq& \theta \Phi_{k} ~\leq~ \theta^{k+1} \Phi_0.
\end{eqnarray}

\subsection{Proof of Corollary~\ref{cor:apgd_apge_complexity}}
\label{sec:proof_of_cor_apgd_apde_complexity}
We aim to derive the iteration complexity of Algorithm~\ref{alg:APGD_APGE} as Corollary~\ref{cor:apgd_apge_complexity}.
In view of \eqref{eq:proof_thm_conv_apge_8}, we have 
\begin{equation}
    k \geq \cO\left(\frac{1}{1-\theta}\log \frac{1}{\varepsilon}\right) \quad\Rightarrow\quad \Phi_k \leq \varepsilon \Phi_0.
\end{equation}
By the definition of $\theta = \max\left\{\frac{1}{1+\mu_{{\mygreen g}}\eta_x}, \frac{2}{2+\eta_z}\right\}$, we have 
\begin{eqnarray*}
    \cO\left(\frac{1}{1-\theta}\log\frac{1}{\varepsilon}\right) &=& \cO\left(\max\left\{\frac{1+\mu_{{\mygreen g}}\eta_x}{\mu_{{\mygreen g}}\eta_x}, \frac{2+\eta_z}{\eta_z}\right\}\log\frac{1}{\varepsilon}\right)\\
    &=& \cO\left(\max\left\{1+\frac{1}{\mu_{{\mygreen g}}\eta_x}, 1+\frac{2}{\eta_z}\right\}\log\frac{1}{\varepsilon}\right)\\
    &=& \cO\left( \left(1 +\sqrt{\frac{L_{\myblue f}}{\mu_{{\mygreen g}}}}\right)\log\frac{1}{\varepsilon}\right).
\end{eqnarray*}
This proves the statement of Corollary~\ref{cor:apgd_apge_complexity}.

\section{Additional Proofs of Section~\ref{sec:ACV_and_APDTR}}
\label{sec:proof_ACV_and_APDTR}

\begin{lemma}
    \label{lem:lyaponov_func_acv_apdtr}
    Define the Lyapunov function $\Psi_{k}$ for every $k \geq 0$ as 
    \begin{eqnarray}
        \Psi_{k} &\eqdef& \frac{1}{2\eta_x}\|x^{k}-\xs\|^2 + \frac{1}{2\eta_y}\|y^{k}-\ys\|^2+  \frac{1}{\eta_z} D_{{\myblue f}}(z^{k};\zs) \notag\\
        && \pm \la y^{k} - \ys, \mathsf{K}(x^{k}-\xs)\ra \pm \la \nabla {\myblue f}(z^{k}) - \nabla{\myblue f}(\zs), x^{k} - \xs\ra. \label{eq:lyaponov_func_acv_apdtr}
    \end{eqnarray}
    Then, if $\|\mathsf{K}\|^2\eta_x\eta_y + L_{\myblue f} \eta_x\eta_z \leq 1$, it holds that
    \begin{equation}
        \label{eq:lyaponov_func_acv_apdtr_1}
        0 \leq \Psi_{k} \leq \frac{1}{\eta_x}\|x^{k}-\xs\|^2 + \frac{1}{\eta_y}\|y^{k}-\ys\|^2 +  \frac{2}{\eta_z} D_{{\myblue f}}(z^{k};\zs).
    \end{equation}
\end{lemma}
\begin{proof}
    By Fenchel-Young's inequality and Lemma~\ref{lem:smoothness_bregman}, we have 
    \begin{eqnarray*}
    \Psi_{k} &\geq& \frac{1}{2\eta_x}\|x^{k}-\xs\|^2 + \frac{1}{2\eta_y}\|y^{k}-\ys\|^2+  \frac{1}{\eta_z} D_{{\myblue f}}(z^{k};\zs) \notag\\
    && -\frac{1}{2\eta_y}\|y^k-\ys\|^2  - \frac{1}{2L_{\myblue f}\eta_z}\|\nabla {\myblue f}(z^{k}) - \nabla{\myblue f}(\zs)\|^2 -\left(\|\mathsf{K}\|^2\eta_x\eta_y + L_{\myblue f}\eta_x\eta_z\right)\frac{1}{2\eta_x}\|x^{k} - \xs\|^2\\
    &\geq& \frac{1}{2\eta_x}\|x^{k}-\xs\|^2 -\left(\|\mathsf{K}\|^2\eta_x\eta_y + L_{\myblue f}\eta_x\eta_z\right)\frac{1}{2\eta_x}\|x^{k} - \xs\|^2.
\end{eqnarray*}
Then, because $\|\mathsf{K}\|^2\eta_x\eta_y + L_{\myblue f} \eta_x\eta_z \leq 1$ , $\Psi_{k}$ is nonnegative for any $k \geq 0$.
By the same argument with opposite sign for bounding the inner product terms, we can prove that for any $k \geq 0$, we have
\begin{equation*}
    \Psi_k \leq \frac{1}{\eta_x}\|x^k-\xs\|^2 + \frac{1}{\eta_y}\|y^{k}-\ys\|^2 + \frac{2}{\eta_z} D_{\myblue f}(z^k; \zs).
\end{equation*}
This completes the proof.
\end{proof}

\subsection{Descent Lemmas for \algname{ACV-I} and \algname{ACV-II}}

First, we present descent lemmas for primal and dual variables of \algname{ACV-I} (Algorithm~\ref{alg:ACV}).

\begin{lemma}
    \label{lem:descent_lemma_acv_1_x}
    Let $(\xs,\ys,\zs)$ be a solution of \eqref{eq:main_optimality_condition}, and assume $\mygreen g$ is $\mu_{\mygreen g}$-strongly convex for some $\mu_{\mygreen g}\ge 0$. Then the iterates generated by \algname{ACV-I} (Algorithm~\ref{alg:ACV}) satisfy
    \begin{eqnarray}
        \frac{1}{2\eta_x}\|x^{k+1}-\xs\|^2&\leq& \frac{1}{2\eta_x}\|x^{k}-\xs\|^2 - \mu_{\mygreen g}\|x^{k+1}-\xs\|^2 -\frac{1}{2\eta_x}\|x^{k+1} -x^k\|^2 \notag\\
        && -\la \mathsf{K}^*(y^k - \ys), x^{k} - \xs\ra  -\la \mathsf{K}^*(y^k - \ys), x^{k+1} - x^{k}\ra \notag\\ 
        &&  -\la \nabla {\myblue f} (z^k)  - \nabla {\myblue f} (\zs) , x^{k} - \xs\ra -\la \nabla {\myblue f} (z^k)  - \nabla {\myblue f} (\zs) , x^{k+1} - x^k\ra. \label{eq:descent_lemma_acv_1_x_1}
    \end{eqnarray}
   Additionally, if ${\mygreen g}$ is $L_{\mygreen g}$-smooth,  then the iterates generated by \algname{ACV-I} (Algorithm~\ref{alg:ACV}) satisfy
    \begin{eqnarray}
        \frac{1}{2\eta_x}\|x^{k+1}-\xs\|^2&\leq& \frac{1}{2\eta_x}\|x^{k}-\xs\|^2 - \frac{\mu_{\mygreen g}}{2}\|x^{k+1}-\xs\|^2 - \frac{1}{2L_{\mygreen g}}\|\nabla {\mygreen g}(x^{k+1}) - \nabla {\mygreen g}(\xs)\|^2\notag\\ &&-\frac{1}{2\eta_x}\|x^{k+1} -x^k\|^2  -\la \mathsf{K}^*(y^k - \ys), x^{k} - \xs\ra  -\la \mathsf{K}^*(y^k - \ys), x^{k+1} - x^{k}\ra \notag\\
        &&  -\la \nabla {\myblue f} (z^k)  - \nabla {\myblue f} (\zs) , x^{k} - \xs\ra -\la \nabla {\myblue f} (z^k)  - \nabla {\myblue f} (\zs) , x^{k+1} - x^k\ra. 
        \label{eq:descent_lemma_acv_1_x_2}
    \end{eqnarray}
\end{lemma}
\begin{proof}
    According to Line $3$ (left) of Algorithm~\ref{alg:ACV}:
    \begin{equation*}
        x^{k+1} = \prox_{\eta_x {\mygreen g}}\left(x^k - \eta_x \mathsf{K}^* y^k - \eta_x \nabla {\myblue f}(z^k) \right) = x^k -\eta_x\left(\partial {\mygreen g}(x^{k+1}) + \mathsf{K}^* y^k + \nabla {\myblue f}(z^k)\right),
    \end{equation*}
    we have 
    \begin{eqnarray}
        \frac{1}{2\eta_x}\|x^{k+1}-\xs\|^2 &=& \frac{1}{2\eta_x}\|x^{k}-\xs\|^2 +\frac{1}{\eta_x}\la x^{k+1} -x^k, x^{k+1} - \xs \ra -\frac{1}{2\eta_x}\|x^{k+1} -x^k\|^2 \notag\\
        &=&\frac{1}{2\eta_x}\|x^{k}-\xs\|^2 -\la \partial {\mygreen g} (x^{k+1}) + \mathsf{K}^*y^k + \nabla {\myblue f}(z^k), x^{k+1} - \xs \ra -\frac{1}{2\eta_x}\|x^{k+1} -x^k\|^2. \label{eq:proof_descent_lemma_acv_1_x_1}
    \end{eqnarray}   
    Plugging in the optimality condition $\partial{\mygreen g}(\xs) +\mathsf{K}^* \ys + \nabla {\myblue f} (\zs) = 0$ into \eqref{eq:proof_descent_lemma_acv_1_x_1}, we obtain 
    \begin{eqnarray}
        \frac{1}{2\eta_x}\|x^{k+1}-\xs\|^2 &=& \frac{1}{2\eta_x}\|x^{k}-\xs\|^2 -\la \partial {\mygreen g} (x^{k+1}) - \partial {\mygreen g} (\xs), x^{k+1} - \xs \ra -\frac{1}{2\eta_x}\|x^{k+1} -x^k\|^2\notag\\
        && -\la \mathsf{K}^* (y^k - \ys), x^{k+1} - \xs\ra -\la  \nabla {\myblue f} (z^k) - \nabla {\myblue f} (\zs) , x^{k+1} - \xs\ra. \label{eq:proof_descent_lemma_acv_1_x_2}
    \end{eqnarray}
    Using Lemma~\ref{lem:strong_convexity_bregman_non_diff}, we obtain 
    \begin{eqnarray}
        \frac{1}{2\eta_x}\|x^{k+1}-\xs\|^2 &\leq& \frac{1}{2\eta_x}\|x^{k}-\xs\|^2 - \mu_{{\mygreen g}}\|x^{k+1} - \xs \|^2 -\frac{1}{2\eta_x}\|x^{k+1} -x^k\|^2 \notag\\
        && -\la \mathsf{K}^*(y^k - \ys), x^{k+1} - \xs\ra -\la \nabla {\myblue f} (z^k)  - \nabla {\myblue f} (\zs) , x^{k+1} - \xs\ra. \label{eq:proof_descent_lemma_acv_1_x_3}
    \end{eqnarray}
    Rewriting \eqref{eq:proof_descent_lemma_acv_1_x_3}, this concludes the first part of the proof.
    
    If we assume that ${\mygreen g}$ is $L_{\mygreen g}$-smooth, then 
    $\partial {\mygreen g}(x) = \{ \nabla {\mygreen g}(x) \}$. Therefore, applying  Lemma~\ref{lem:smoothness_bregman} and Lemma~\ref{lem:strong_convexity_bregman}, we derive from \eqref{eq:proof_descent_lemma_acv_1_x_2} that
    \begin{eqnarray}
        \frac{1}{2\eta_x}\|x^{k+1}-\xs\|^2 &=& \frac{1}{2\eta_x}\|x^{k}-\xs\|^2 -\la \nabla{\mygreen g} (x^{k+1}) - \nabla{\mygreen g} (\xs), x^{k+1} - \xs \ra -\frac{1}{2\eta_x}\|x^{k+1} -x^k\|^2\notag\\
        && -\la \mathsf{K}^* (y^k - \ys), x^{k+1} - \xs\ra -\la  \nabla {\myblue f} (z^k) - \nabla {\myblue f} (\zs) , x^{k+1} - \xs\ra\\
        &\leq& \frac{1}{2\eta_x}\|x^{k}-\xs\|^2 - \frac{\mu_{\mygreen g}}{2}\|x^{k+1} - \xs \|^2 -\frac{1}{2L_{\mygreen g}}\|\nabla{\mygreen g} (x^{k+1}) - \nabla{\mygreen g} (\xs)\|^2\notag\\
        && -\frac{1}{2\eta_x}\|x^{k+1} -x^k\|^2 -\la \mathsf{K}^* (y^k - \ys), x^{k+1} - \xs\ra \notag\\
        &&-\la  \nabla {\myblue f} (z^k) - \nabla {\myblue f} (\zs) , x^{k+1} - \xs\ra \label{eq:proof_descent_lemma_acv_1_x_4}
    \end{eqnarray}
    Rewriting \eqref{eq:proof_descent_lemma_acv_1_x_4}, this concludes the second part of the proof.
\end{proof}

\begin{lemma}
    \label{lem:descent_lemma_acv_1_y}
    Let $(\xs,\ys,\zs)$ be a solution of \eqref{eq:main_optimality_condition}. Assume $\hstar$ is $\mu_{\hstar}$-strongly convex and the sequence $\{y^k\}_{k\geq 0}$ is defined recursively by the update rule
    \begin{equation*}
        y^{k+1} = \prox_{\eta_y\hstar}\left(y^k + \eta_y \mathsf{K}(2x^{k+1}-x^k)\right). \footnote{For example, the update rule in Line $4$ (left) of Algorithm~\ref{alg:ACV} or Line $5$ (left) of Algorithm~\ref{alg:APDTR} is of this form.}
    \end{equation*}
    Then, for any constant $a_y >0$ and any integer $k \geq 0$, it holds:
    \begin{eqnarray}
        \frac{1}{2\eta_y}\|y^{k+1}-\ys\|^2 &\leq& \frac{1}{2\eta_y}\|y^{k}-\ys\|^2 - \mu_{\hstar}\|y^{k+1}-\ys\|^2 -\frac{1-a_{y}}{2\eta_y}\|y^{k+1}-y^k\|^2 +\la y^{k+1}-\ys, \mathsf{K}(x^{k+1}-\xs)\ra \notag\\
        && +\la y^{k}-\ys, \mathsf{K}(x^{k+1}-x^k)\ra  + \|\mathsf{K}\|^2\eta_x\eta_y \frac{1}{2a_y\eta_x}\|x^{k+1}-x^k\|^2. \label{eq:descent_lemma_acv_1_y}
    \end{eqnarray}
\end{lemma}
\begin{proof}
    Using the update rule for $y^{k+1}$:
    \begin{equation*}
        y^{k+1} = \prox_{\eta_y\hstar}\left(y^k + \eta_y \mathsf{K}(2x^{k+1}-x^k)\right) = y^k - \eta_y(\partial\hstar(y^{k+1}) - \mathsf{K}(2x^{k+1}-x^k)),
    \end{equation*}
    we derive 
    \begin{eqnarray}
        \frac{1}{2\eta_y}\|y^{k+1}-\ys\|^2 &=& \frac{1}{2\eta_y}\|y^{k}-\ys\|^2 +\frac{1}{\eta_y}\la y^{k+1}-y^k, y^{k+1}-\ys\ra -\frac{1}{2\eta_y}\|y^{k+1}-y^k\|^2 \notag\\
        &=&\frac{1}{2\eta_y}\|y^{k}-\ys\|^2 - \la \partial \hstar (y^{k+1}) - \mathsf{K} (2x^{k+1}-x^k), y^{k+1}-\ys\ra -\frac{1}{2\eta_y}\|y^{k+1}-y^k\|^2 .\label{eq:proof_descent_lemma_acv_1_y_1}
    \end{eqnarray}
    Plugging in the optimality condition $\partial \hstar(\ys) - \mathsf{K}\xs = 0$ into \eqref{eq:proof_descent_lemma_acv_1_y_1}, we obtain
    \begin{eqnarray*}
        \frac{1}{2\eta_y}\|y^{k+1}-\ys\|^2 &=&  \frac{1}{2\eta_y}\|y^{k}-\ys\|^2 - \la \partial \hstar (y^{k+1}) - \partial\hstar (\ys), y^{k+1}-\ys\ra -\frac{1}{2\eta_y}\|y^{k+1}-y^k\|^2 \\
        && +\la y^{k+1}-\ys, \mathsf{K}(x^{k+1}-\xs)\ra +\la y^{k+1}-\ys, \mathsf{K}(x^{k+1}-x^k)\ra.
    \end{eqnarray*}
    Applying Lemma~\ref{lem:strong_convexity_bregman_non_diff}, we have 
    \begin{eqnarray}
        \frac{1}{2\eta_y}\|y^{k+1}-\ys\|^2 &\leq& \frac{1}{2\eta_y}\|y^{k}-\ys\|^2 - \mu_{\hstar}\|y^{k+1}-\ys\|^2 -\frac{1}{2\eta_y}\|y^{k+1}-y^k\|^2 \notag\\
        && +\la y^{k+1}-\ys, \mathsf{K}(x^{k+1}-\xs)\ra +\la y^{k}-\ys, \mathsf{K}(x^{k+1}-x^k)\ra \notag\\
        &&  + \underbrace{\la y^{k+1}- y^k, \mathsf{K}(x^{k+1}-x^k)\ra}_{\eqdef \cT_1}. \label{eq:proof_descent_lemma_acv_1_y_2}
    \end{eqnarray}
    To continue our proof, we need to bound $\cT_1$. By Fenchel-Young's inequality, we get 
    \begin{eqnarray*}
        \cT_1 \leq \frac{C}{2} \|y^{k+1}- y^k\|^2 + \frac{1}{2C}\|\mathsf{K}(x^{k+1}-x^k)\|^2.
    \end{eqnarray*}
    By the definition of the operator norm $\|\mathsf{K}\|$, we have 
    \begin{eqnarray}
        \cT_1 &\leq& \frac{C}{2} \|y^{k+1}- y^k\|^2 + \frac{\|\mathsf{K}\|^2}{2C}\|x^{k+1}-x^k\|^2 \notag\\
        &\overset{C = \nicefrac{a_y}{\eta_y}}{=}& \frac{a_y}{2\eta_y} \|y^{k+1}- y^k\|^2 + \|\mathsf{K}\|^2\eta_x\eta_y \frac{1}{2a_y\eta_x}\|x^{k+1}-x^k\|^2 \label{eq:proof_descent_lemma_acv_1_y_3}
    \end{eqnarray}
    Plugging in \eqref{eq:proof_descent_lemma_acv_1_y_3} into \eqref{eq:proof_descent_lemma_acv_1_y_2}, we obtain 
    \begin{eqnarray*}
        \frac{1}{2\eta_y}\|y^{k+1}-\ys\|^2 &\leq& \frac{1}{2\eta_y}\|y^{k}-\ys\|^2 - \mu_{\hstar}\|y^{k+1}-\ys\|^2 -\frac{1}{2\eta_y}\|y^{k+1}-y^k\|^2 \notag\\
        && +\la y^{k+1}-\ys, \mathsf{K}(x^{k+1}-\xs)\ra +\la y^{k}-\ys, \mathsf{K}(x^{k+1}-x^k)\ra \notag\\
        &&  +  \frac{a_y}{2\eta_y} \|y^{k+1}- y^k\|^2 + \|\mathsf{K}\|^2\eta_x\eta_y \frac{1}{2a_y\eta_x}\|x^{k+1}-x^k\|^2,
    \end{eqnarray*}
    which concludes the proof.
\end{proof}

Now we present descent lemmas for \algname{ACV-II} (Algorithm~\ref{alg:ACV}).

\begin{lemma}
    \label{lem:descent_lemma_acv_2_x}
    Let $(\xs,\ys,\zs)$ be a solution of \eqref{eq:main_optimality_condition}, and assume $\mygreen g$ is $\mu_{\mygreen g}$-strongly convex for some $\mu_{\mygreen g}\ge 0$. Then the iterates generated by \algname{ACV-II} (Algorithm~\ref{alg:ACV}) satisfy
    \begin{eqnarray}
        \frac{1}{2\eta_x}\|x^{k+1}-\xs\|^2&\leq& \frac{1}{2\eta_x}\|x^{k}-\xs\|^2 - \mu_{\mygreen g}\|x^{k+1}-\xs\|^2 -\frac{1}{2\eta_x}\|x^{k+1} -x^k\|^2 \notag\\
        && -\la \mathsf{K}^*(y^{k+1} - \ys), x^{k+1} - \xs\ra - \la \mathsf{K}^*(y^{k+1} - y^k), x^{k+1} - \xs\ra \notag\\
        &&  -\la \nabla {\myblue f} (z^k)  - \nabla {\myblue f} (\zs) , x^{k} - \xs\ra -\la \nabla {\myblue f} (z^k)  - \nabla {\myblue f} (\zs) , x^{k+1} - x^k\ra. \label{eq:descent_lemma_acv_2_x_1}
    \end{eqnarray}
    Additionally, if $\mygreen g$ is $L_{\mygreen g}$-smooth, the iterates generated by \algname{ACV-II} (Algorithm~\ref{alg:ACV}) satisfy
    \begin{eqnarray}
        \frac{1}{2\eta_x}\|x^{k+1}-\xs\|^2&\leq& \frac{1}{2\eta_x}\|x^{k}-\xs\|^2  -\frac{1}{2\eta_x}\|x^{k+1} -x^k\|^2 \notag\\
        && - \frac{\mu_{\mygreen g}}{2}\|x^{k+1}-\xs\|^2 - \frac{1}{2 L_{\mygreen g}}\|\nabla {\mygreen g}(x^{k+1}) -\nabla {\mygreen g}(\xs)\|^2 \notag\\
        && -\la \mathsf{K}^*(y^{k+1} - \ys), x^{k+1} - \xs\ra - \la \mathsf{K}^*(y^{k+1} - y^k), x^{k+1} - \xs\ra \notag\\
        &&  -\la \nabla {\myblue f} (z^k)  - \nabla {\myblue f} (\zs) , x^{k} - \xs\ra -\la \nabla {\myblue f} (z^k)  - \nabla {\myblue f} (\zs) , x^{k+1} - x^k\ra. \label{eq:descent_lemma_acv_2_x_2}
    \end{eqnarray}
\end{lemma}
\begin{proof}
    According to Line $4$ (left) of Algorithm~\ref{alg:ACV}:
    \begin{equation*}
        x^{k+1} = \text{prox}_{\eta_x {\mygreen g}}\left(x^k - \eta_x \mathsf{K}^* (2y^{k+1} -y^k) - \eta_x \nabla {\myblue f}(z^k) \right) = x^k -\eta_x\left(\partial {\mygreen g}(x^{k+1}) + \mathsf{K}^* (2y^{k+1} -y^k) + \nabla {\myblue f}(z^k)\right),
    \end{equation*}
    we have 
    \begin{eqnarray}
        \frac{1}{2\eta_x}\|x^{k+1}-\xs\|^2 &=& \frac{1}{2\eta_x}\|x^{k}-\xs\|^2 +\frac{1}{\eta_x}\la x^{k+1} -x^k, x^{k+1} - \xs \ra -\frac{1}{2\eta_x}\|x^{k+1} -x^k\|^2 \notag\\
        &=&\frac{1}{2\eta_x}\|x^{k}-\xs\|^2 -\la \partial {\mygreen g} (x^{k+1}) + \mathsf{K}^*(2y^{k+1} -y^k) + \nabla {\myblue f}(z^k), x^{k+1} - \xs \ra \notag\\
        && -\frac{1}{2\eta_x}\|x^{k+1} -x^k\|^2. \label{eq:proof_descent_lemma_acv_2_x_1}
    \end{eqnarray}
    Plugging in the optimality condition $\partial{\mygreen g}(\xs) +\mathsf{K}^* \ys + \nabla {\myblue f} (\zs) = 0$ into \eqref{eq:proof_descent_lemma_acv_2_x_1}, we obtain 
    \begin{eqnarray}
        \frac{1}{2\eta_x}\|x^{k+1}-\xs\|^2 &=& \frac{1}{2\eta_x}\|x^{k}-\xs\|^2 -\la \partial {\mygreen g} (x^{k+1}) - \partial {\mygreen g} (\xs), x^{k+1} - \xs \ra -\frac{1}{2\eta_x}\|x^{k+1} -x^k\|^2\notag\\
        && -\la \mathsf{K}^* (y^{k+1} - \ys), x^{k+1} - \xs\ra - \la \mathsf{K}^* (y^{k+1} - y^k), x^{k+1} - \xs\ra\notag\\
        && -\la  \nabla {\myblue f} (z^k) - \nabla {\myblue f} (\zs) , x^{k+1} - \xs\ra. \label{eq:proof_descent_lemma_acv_2_x_2}
    \end{eqnarray}
    Applying Lemma~\ref{lem:strong_convexity_bregman_non_diff}, we obtain 
    \begin{eqnarray}
        \frac{1}{2\eta_x}\|x^{k+1}-\xs\|^2 &\leq& \frac{1}{2\eta_x}\|x^{k}-\xs\|^2 - \mu_{{\mygreen g}}\|x^{k+1} - \xs \|^2 -\frac{1}{2\eta_x}\|x^{k+1} -x^k\|^2 \notag\\
        && -\la \mathsf{K}^* (y^{k+1} - \ys), x^{k+1} - \xs\ra - \la \mathsf{K}^* (y^{k+1} - y^k), x^{k+1} - \xs\ra\notag\\
        && -\la  \nabla {\myblue f} (z^k) - \nabla {\myblue f} (\zs) , x^{k+1} - \xs\ra. \label{eq:proof_descent_lemma_acv_2_x_3}
    \end{eqnarray}
    Rewriting \eqref{eq:proof_descent_lemma_acv_2_x_3}, this concludes the first part of the proof.

    If we assume that ${\mygreen g}$ is $L_{\mygreen g}$-smooth, we have that $\nabla {\mygreen g}(x) \in \partial {\mygreen g}(x)$. Then, applying  Lemma~\ref{lem:smoothness_bregman} and Lemma~\ref{lem:strong_convexity_bregman}, we derive from \eqref{eq:proof_descent_lemma_acv_2_x_2}:
    \begin{eqnarray}
        \frac{1}{2\eta_x}\|x^{k+1}-\xs\|^2 &=& \frac{1}{2\eta_x}\|x^{k}-\xs\|^2 -\la \nabla{\mygreen g} (x^{k+1}) - \nabla{\mygreen g} (\xs), x^{k+1} - \xs \ra -\frac{1}{2\eta_x}\|x^{k+1} -x^k\|^2\notag\\
        && -\la \mathsf{K}^* (y^{k+1} - \ys), x^{k+1} - \xs\ra - \la \mathsf{K}^* (y^{k+1} - y^k), x^{k+1} - \xs\ra\notag\\
        && -\la  \nabla {\myblue f} (z^k) - \nabla {\myblue f} (\zs) , x^{k+1} - \xs\ra\notag\\
        &\leq& \frac{1}{2\eta_x}\|x^{k}-\xs\|^2 - \frac{\mu_{\mygreen g}}{2}\|x^{k+1} - \xs \|^2 -\frac{1}{2L_{\mygreen g}}\|\nabla{\mygreen g} (x^{k+1}) - \nabla{\mygreen g} (\xs)\|^2\notag\\
        && -\frac{1}{2\eta_x}\|x^{k+1} -x^k\|^2 -\la \mathsf{K}^* (y^{k+1} - \ys), x^{k+1} - \xs\ra \notag\\
        && - \la \mathsf{K}^* (y^{k+1} - y^k), x^{k+1} - \xs\ra -\la  \nabla {\myblue f} (z^k) - \nabla {\myblue f} (\zs) , x^{k+1} - \xs\ra. \label{eq:proof_descent_lemma_acv_2_x_4}
    \end{eqnarray}
    Rewriting \eqref{eq:proof_descent_lemma_acv_2_x_4}, this concludes the second part of the proof.
\end{proof}

\begin{lemma}
    \label{lem:descent_lemma_acv_2_y}
    Let $(\xs,\ys,\zs)$ be a solution of \eqref{eq:main_optimality_condition}. Assume $\hstar$ is $\mu_{\hstar}$-strongly convex and the sequence $\{y^k\}_{k \geq 0}$ is defined by the update rule
    \begin{equation*}
        y^{k+1} = \prox_{\eta_y\hstar}\left(y^k + \eta_y \mathsf{K}x^k\right).\footnote{For example, the update rule in Line $3$ (right) of Algorithm~\ref{alg:ACV} or Line $3$ (right) of Algorithm~\ref{alg:APDTR} is of this form.}
    \end{equation*}
    Then, for any constant $a_y >0$ and any integer $k \geq 0$, it holds:
    \begin{eqnarray}
        \frac{1}{2\eta_y}\|y^{k+1}-\ys\|^2 &\leq& \frac{1}{2\eta_y}\|y^{k}-\ys\|^2 - \mu_{\hstar}\|y^{k+1}-\ys\|^2 -\frac{1-a_{y}}{2\eta_y}\|y^{k+1}-y^k\|^2 +\la y^{k}-\ys, \mathsf{K}(x^{k}-\xs)\ra \notag\\
        && +\la y^{k+1}-y^k, \mathsf{K}(x^{k+1}-\xs)\ra  + \|\mathsf{K}\|^2\eta_x\eta_y \frac{1}{2a_y\eta_x}\|x^{k+1}-x^k\|^2. \label{eq:descent_lemma_acv_2_y}
    \end{eqnarray}
\end{lemma}
\begin{proof}
    Using the update rule for the sequence $y^k$,
    \begin{equation*}
        y^{k+1} = \prox_{\eta_y\hstar}\left(y^k + \eta_y \mathsf{K}x^k\right) = y^k - \eta_y(\partial\hstar(y^{k+1}) - \mathsf{K}x^k),
    \end{equation*}
    we derive 
    \begin{eqnarray}
        \frac{1}{2\eta_y}\|y^{k+1}-\ys\|^2 &=& \frac{1}{2\eta_y}\|y^{k}-\ys\|^2 +\frac{1}{\eta_y}\la y^{k+1}-y^k, y^{k+1}-\ys\ra -\frac{1}{2\eta_y}\|y^{k+1}-y^k\|^2 \notag\\
        &=&\frac{1}{2\eta_y}\|y^{k}-\ys\|^2 - \la \partial \hstar (y^{k+1}) - \mathsf{K} x^k, y^{k+1}-\ys\ra -\frac{1}{2\eta_y}\|y^{k+1}-y^k\|^2 .\label{eq:proof_descent_lemma_acv_2_y_1}
    \end{eqnarray}
    Plugging in the optimality condition $\partial \hstar(\ys) - \mathsf{K}\xs = 0$ into \eqref{eq:proof_descent_lemma_acv_2_y_1}, we obtain
    \begin{eqnarray*}
        \frac{1}{2\eta_y}\|y^{k+1}-\ys\|^2 &=&  \frac{1}{2\eta_y}\|y^{k}-\ys\|^2 - \la \partial \hstar (y^{k+1}) - \partial\hstar (\ys), y^{k+1}-\ys\ra -\frac{1}{2\eta_y}\|y^{k+1}-y^k\|^2 \\
        && +\la y^{k+1}-\ys, \mathsf{K}(x^{k}-\xs)\ra.
    \end{eqnarray*}
    Applying Lemma~\ref{lem:strong_convexity_bregman_non_diff}, we have 
    \begin{eqnarray}
        \frac{1}{2\eta_y}\|y^{k+1}-\ys\|^2 &\leq& \frac{1}{2\eta_y}\|y^{k}-\ys\|^2 - \mu_{\hstar}\|y^{k+1}-\ys\|^2 -\frac{1}{2\eta_y}\|y^{k+1}-y^k\|^2 \notag\\
        && +\la y^{k}-\ys, \mathsf{K}(x^{k}-\xs)\ra + \la y^{k+1}-y^k, \mathsf{K}(x^{k+1}-\xs)\ra\notag\\
        &&  - \underbrace{\la y^{k+1}- y^k, \mathsf{K}(x^{k+1}-x^k)\ra}_{\eqdef \cT_1}. \label{eq:proof_descent_lemma_acv_2_y_2}
    \end{eqnarray}
    Plugging in \eqref{eq:proof_descent_lemma_acv_1_y_3} into \eqref{eq:proof_descent_lemma_acv_2_y_2}, we obtain 
    \begin{eqnarray*}
        \frac{1}{2\eta_y}\|y^{k+1}-\ys\|^2 &\leq& \frac{1}{2\eta_y}\|y^{k}-\ys\|^2 - \mu_{\hstar}\|y^{k+1}-\ys\|^2 -\frac{1}{2\eta_y}\|y^{k+1}-y^k\|^2 \notag\\
        && +\la y^{k}-\ys, \mathsf{K}(x^{k}-\xs)\ra + \la y^{k+1}-y^k, \mathsf{K}(x^{k+1}-\xs)\ra \notag\\
        &&  +  \frac{a_y}{2\eta_y} \|y^{k+1}- y^k\|^2 + \|\mathsf{K}\|^2\eta_x\eta_y \frac{1}{2a_y\eta_x}\|x^{k+1}-x^k\|^2.
    \end{eqnarray*}
    
    Rearranging the terms from the previous inequality, we conclude the proof.
\end{proof}

\subsection{Descent Lemmas for \algname{APDTR-I} and \algname{APDTR-II}}
\label{sec:descent_lemmas_for_apdtr}

The descent lemmas for \algname{APDTR-I} and \algname{APDTR-II} are stated in the following lemmas. Note that the descent lemma for dual updates are the same as for \algname{ACV-I} and \algname{ACV-II}, we do not repeat it here. 

\begin{lemma}
    \label{lem:descent_lemma_apdtr_i_x}
    Let $(\xs,\ys,\zs)$ be a solution of \eqref{eq:main_optimality_condition}, and assume ${\mygreen g}$ is $\mu_{\mygreen g}$-strongly convex for some $\mu_{\mygreen g}\ge 0$. Then the iterates generated by \algname{APDTR-I} (Algorithm~\ref{alg:APDTR}) satisfy
    \begin{eqnarray}
        \frac{1}{2\eta_x}\|x^{k+1}-\xs\|^2 &\leq& \frac{1}{2\eta_x}\|x^{k}-\xs\|^2 - \mu_{{\mygreen g}}\|x^{k+1} - \xs \|^2 -\frac{1}{2\eta_x}\|x^{k+1} -x^k\|^2 \notag\\
        && -\la\mathsf{K}^*( y^{k} - \ys), x^{k} - \xs\ra -\la  \nabla {\myblue f} (z^{k+1}) - \nabla {\myblue f} (\zs) , x^{k+1} - \xs\ra \notag\\
        && -\la \mathsf{K}^*(y^{k}-\ys), x^{k+1} - x^k \ra  -\la \nabla {\myblue f}(z^{k+1}) -  \nabla {\myblue f}(z^k), x^{k+1} - \xs \ra. \label{eq:descent_lemma_apdtr_i_x_1}
    \end{eqnarray}
  Additionally,  if ${\mygreen g}$ is $L_{\mygreen g}$-smooth, then the iterates generated by \algname{APDTR-I} (Algorithm~\ref{alg:APDTR}) satisfy
    \begin{eqnarray}
        \frac{1}{2\eta_x}\|x^{k+1}-\xs\|^2 &\leq& \frac{1}{2\eta_x}\|x^{k}-\xs\|^2 - \frac{\mu_{{\mygreen g}}}{2}\|x^{k+1} - \xs \|^2 - \frac{1}{2L_{{\mygreen g}}}\|\nabla {\mygreen g}(x^{k+1}) - \nabla {\mygreen g}(\xs) \|^2 -\frac{1}{2\eta_x}\|x^{k+1} -x^k\|^2 \notag\\
        && -\la \mathsf{K}^*( y^{k} - \ys), x^{k} - \xs\ra -\la  \nabla {\myblue f} (z^{k+1}) - \nabla {\myblue f} (\zs) , x^{k+1} - \xs\ra \notag\\
        && -\la \mathsf{K}^*(y^{k}-\ys), x^{k+1} - x^k \ra  -\la \nabla {\myblue f}(z^{k+1}) -  \nabla {\myblue f}(z^k), x^{k+1} - \xs \ra. \label{eq:descent_lemma_apdtr_i_x_2}
    \end{eqnarray}
\end{lemma}
\begin{proof}
    According to Line $4$ (left) of Algorithm~\ref{alg:APDTR}, we have 
\begin{eqnarray*}
    x^{k+1} &=& \prox_{\eta_x {\mygreen g}}\left(x^k - \eta_x \mathsf{K}^* y^k - \eta_x(2\nabla {\myblue f}(z^{k+1}) -  \nabla {\myblue f}(z^k)) \right)\\
    &=& x^k -\eta_x\left(\partial {\mygreen g}(x^{k+1}) +  \mathsf{K}^* y^k + (2\nabla {\myblue f}(z^{k+1}) -  \nabla {\myblue f}(z^k))\right),
\end{eqnarray*}
Using this, we derive 
\begin{eqnarray}
    \frac{1}{2\eta_x}\|x^{k+1}-\xs\|^2 &=& \frac{1}{2\eta_x}\|x^{k}-\xs\|^2 +\frac{1}{\eta_x}\la x^{k+1} -x^k, x^{k+1} - \xs \ra -\frac{1}{2\eta_x}\|x^{k+1} -x^k\|^2 \notag\\
    &=&\frac{1}{2\eta_x}\|x^{k}-\xs\|^2 -\frac{1}{2\eta_x}\|x^{k+1} -x^k\|^2\notag\\
    && -\la  \partial {\mygreen g}(x^{k+1}) +  \mathsf{K}^* y^k + (2\nabla {\myblue f}(z^{k+1}) -  \nabla {\myblue f}(z^k)), x^{k+1} - \xs \ra  \notag.  \label{eq:proof_descent_lemma_apdtr_i_x_1}
\end{eqnarray}
Plugging in the optimality condition $\partial{\mygreen g}(\xs) +\mathsf{K}^* \ys + \nabla {\myblue f} (\zs) = 0$ into \eqref{eq:proof_descent_lemma_apdtr_i_x_1}, we obtain 
\begin{eqnarray}
    \frac{1}{2\eta_x}\|x^{k+1}-\xs\|^2 &=& \frac{1}{2\eta_x}\|x^{k}-\xs\|^2 -\la \partial {\mygreen g} (x^{k+1}) - \partial {\mygreen g} (\xs), x^{k+1} - \xs \ra -\frac{1}{2\eta_x}\|x^{k+1} -x^k\|^2 \notag\\
    && -\la \mathsf{K}^* (y^{k} - \ys), x^{k+1} - \xs\ra  -\la \nabla {\myblue f}(z^{k+1}) -  \nabla {\myblue f}(\zs), x^{k+1} - \xs \ra \notag\\
    && -\la \nabla {\myblue f}(z^{k+1}) -  \nabla {\myblue f}(z^k), x^{k+1} - \xs \ra. \label{eq:proof_descent_lemma_apdtr_i_x_2}
\end{eqnarray}
Applying Lemma~\ref{lem:strong_convexity_bregman_non_diff}, we obtain 
\begin{eqnarray}
    \frac{1}{2\eta_x}\|x^{k+1}-\xs\|^2 &\leq& \frac{1}{2\eta_x}\|x^{k}-\xs\|^2 - \mu_{{\mygreen g}}\|x^{k+1} - \xs \|^2 -\frac{1}{2\eta_x}\|x^{k+1} -x^k\|^2 \notag\\
    && -\la \mathsf{K}^* (y^{k} - \ys), x^{k} - \xs\ra -\la  \nabla {\myblue f} (z^{k+1}) - \nabla {\myblue f} (\zs) , x^{k+1} - \xs\ra \notag\\
    && -\la\mathsf{K}^* (y^{k}-\ys), x^{k+1} - x^k\ra  -\la \nabla {\myblue f}(z^{k+1}) -  \nabla {\myblue f}(z^k), x^{k+1} - \xs \ra, \notag
\end{eqnarray}
which concludes the first part of the proof. 

If we assume that ${\mygreen g}$ is $L_{\mygreen g}$-smooth, we have that $\nabla {\mygreen g}(x) \in \partial {\mygreen g}(x)$. Then, applying  Lemma~\ref{lem:smoothness_bregman} and Lemma~\ref{lem:strong_convexity_bregman}, we derive from \eqref{eq:proof_descent_lemma_apdtr_i_x_2} that
\begin{eqnarray*}
    \frac{1}{2\eta_x}\|x^{k+1}-\xs\|^2 &\leq& \frac{1}{2\eta_x}\|x^{k}-\xs\|^2 - \frac{\mu_{{\mygreen g}}}{2}\|x^{k+1} - \xs \|^2 - \frac{1}{2L_{{\mygreen g}}}\|\nabla {\mygreen g}(x^{k+1}) - \nabla {\mygreen g}(\xs) \|^2 -\frac{1}{2\eta_x}\|x^{k+1} -x^k\|^2 \notag\\
    && -\la \mathsf{K}^* (y^{k} - \ys), x^{k} - \xs\ra -\la  \nabla {\myblue f} (z^{k+1}) - \nabla {\myblue f} (\zs) , x^{k+1} - \xs\ra \notag\\
    && -\la\mathsf{K}^* (y^{k}-\ys), x^{k+1} - x^k \ra  -\la \nabla {\myblue f}(z^{k+1}) -  \nabla {\myblue f}(z^k), x^{k+1} - \xs \ra, \notag
\end{eqnarray*}
which completes the proof.
\end{proof}

\begin{lemma}
    \label{lem:descent_lemma_apdtr_ii_x}
    Let $(\xs,\ys,\zs)$ be a solution of \eqref{eq:main_optimality_condition}, and assume ${\mygreen g}$ is $\mu_{\mygreen g}$-strongly convex for some $\mu_{\mygreen g}\ge 0$. Then the iterates generated by \algname{APDTR-II} (Algorithm~\ref{alg:APDTR}) satisfy
    \begin{eqnarray}
        \frac{1}{2\eta_x}\|x^{k+1}-\xs\|^2 &\leq& \frac{1}{2\eta_x}\|x^{k}-\xs\|^2 - \mu_{{\mygreen g}}\|x^{k+1} - \xs \|^2 -\frac{1}{2\eta_x}\|x^{k+1} -x^k\|^2 \notag\\
        && -\la \mathsf{K}^*(y^{k+1} - \ys), x^{k+1} - \xs\ra -\la  \nabla {\myblue f} (z^{k+1}) - \nabla {\myblue f} (\zs) , x^{k+1} - \xs\ra \notag\\
        && -\la \mathsf{K}^*(y^{k+1} - y^{k}), x^{k+1} - \xs \ra  -\la \nabla {\myblue f}(z^{k+1}) -  \nabla {\myblue f}(z^k), x^{k+1} - \xs \ra. \label{eq:descent_lemma_apdtr_ii_x_1}
    \end{eqnarray}
    Additionally, if ${\mygreen g}$ is $L_{\mygreen g}$-smooth, then the iterates generated by \algname{APDTR-II} (Algorithm~\ref{alg:APDTR}) satisfy
    \begin{eqnarray}
        \frac{1}{2\eta_x}\|x^{k+1}-\xs\|^2 &\leq& \frac{1}{2\eta_x}\|x^{k}-\xs\|^2 - \frac{\mu_{{\mygreen g}}}{2}\|x^{k+1} - \xs \|^2 - \frac{1}{2L_{{\mygreen g}}}\|\nabla {\mygreen g}(x^{k+1}) - \nabla {\mygreen g}(\xs) \|^2 -\frac{1}{2\eta_x}\|x^{k+1} -x^k\|^2 \notag\\
        && -\la \mathsf{K}^*(y^{k+1} - \ys), x^{k+1} - \xs\ra -\la  \nabla {\myblue f} (z^{k+1}) - \nabla {\myblue f} (\zs) , x^{k+1} - \xs\ra \notag\\
        && -\la \mathsf{K}^*(y^{k+1}-y^k), x^{k+1} - \xs\ra  -\la \nabla {\myblue f}(z^{k+1}) -  \nabla {\myblue f}(z^k), x^{k+1} - \xs \ra. \label{eq:descent_lemma_apdtr_ii_x_2}
    \end{eqnarray}
\end{lemma}
\begin{proof}
    According to Line $5$ (right) of Algorithm~\ref{alg:APDTR}, we have 
\begin{eqnarray*}
    x^{k+1} &=& \prox_{\eta_x {\mygreen g}}\left(x^k - \eta_x \mathsf{K}^* (2y^{k+1}-y^k) - \eta_x(2\nabla {\myblue f}(z^{k+1}) -  \nabla {\myblue f}(z^k)) \right)\\
    &=& x^k -\eta_x\left(\partial {\mygreen g}(x^{k+1}) +  \mathsf{K}^* (2y^{k+1}-y^k) + (2\nabla {\myblue f}(z^{k+1}) -  \nabla {\myblue f}(z^k))\right).
\end{eqnarray*}
Using this, we derive 
\begin{eqnarray}
    \frac{1}{2\eta_x}\|x^{k+1}-\xs\|^2 &=& \frac{1}{2\eta_x}\|x^{k}-\xs\|^2 +\frac{1}{\eta_x}\la x^{k+1} -x^k, x^{k+1} - \xs \ra -\frac{1}{2\eta_x}\|x^{k+1} -x^k\|^2 \notag\\
    &=&\frac{1}{2\eta_x}\|x^{k}-\xs\|^2 -\frac{1}{2\eta_x}\|x^{k+1} -x^k\|^2\notag\\
    && -\la  \partial {\mygreen g}(x^{k+1}) +  \mathsf{K}^* (2y^{k+1}-y^k) + (2\nabla {\myblue f}(z^{k+1}) -  \nabla {\myblue f}(z^k)), x^{k+1} - \xs \ra  \notag.  \label{eq:proof_descent_lemma_apdtr_ii_x_1}
\end{eqnarray}

Plugging in the optimality condition $\partial{\mygreen g}(\xs) +\mathsf{K}^* \ys + \nabla {\myblue f} (\zs) = 0$ into \eqref{eq:proof_descent_lemma_apdtr_ii_x_1}, we obtain 
\begin{eqnarray}
    \frac{1}{2\eta_x}\|x^{k+1}-\xs\|^2 &=& \frac{1}{2\eta_x}\|x^{k}-\xs\|^2 -\la \partial {\mygreen g} (x^{k+1}) - \partial {\mygreen g} (\xs), x^{k+1} - \xs \ra -\frac{1}{2\eta_x}\|x^{k+1} -x^k\|^2 \notag\\
    && -\la \mathsf{K}^* (y^{k+1} - \ys), x^{k+1} - \xs\ra -\la  \nabla {\myblue f} (z^{k+1}) - \nabla {\myblue f} (\zs) , x^{k+1} - \xs\ra \notag\\
    && -\la\mathsf{K}^* (y^{k+1}-y^k), x^{k+1} - \xs \ra  -\la \nabla {\myblue f}(z^{k+1}) -  \nabla {\myblue f}(z^k), x^{k+1} - \xs \ra. \label{eq:proof_descent_lemma_apdtr_ii_x_2}
\end{eqnarray}
Applying Lemma~\ref{lem:strong_convexity_bregman_non_diff}, we obtain 
\begin{eqnarray}
    \frac{1}{2\eta_x}\|x^{k+1}-\xs\|^2 &\leq& \frac{1}{2\eta_x}\|x^{k}-\xs\|^2 - \mu_{{\mygreen g}}\|x^{k+1} - \xs \|^2 -\frac{1}{2\eta_x}\|x^{k+1} -x^k\|^2 \notag\\
    && -\la \mathsf{K}^* (y^{k+1} - \ys), x^{k+1} - \xs\ra -\la  \nabla {\myblue f} (z^{k+1}) - \nabla {\myblue f} (\zs) , x^{k+1} - \xs\ra \notag\\
    && -\la\mathsf{K}^* (y^{k+1}-y^k), x^{k+1} - \xs \ra  -\la \nabla {\myblue f}(z^{k+1}) -  \nabla {\myblue f}(z^k), x^{k+1} - \xs \ra, \notag
\end{eqnarray}
which concludes the first part of the proof. 

If we assume that ${\mygreen g}$ is $L_{\mygreen g}$-smooth, we have that $\nabla {\mygreen g}(x) \in \partial {\mygreen g}(x)$, moreover, $\partial {\mygreen g}(x) = \{ \nabla {\mygreen g}(x) \}$. Then, applying  Lemma~\ref{lem:smoothness_bregman} and Lemma~\ref{lem:strong_convexity_bregman}, we derive from \eqref{eq:proof_descent_lemma_apdtr_ii_x_2}
\begin{eqnarray*}
    \frac{1}{2\eta_x}\|x^{k+1}-\xs\|^2 &\leq& \frac{1}{2\eta_x}\|x^{k}-\xs\|^2 - \frac{\mu_{{\mygreen g}}}{2}\|x^{k+1} - \xs \|^2 - \frac{1}{2L_{{\mygreen g}}}\|\nabla {\mygreen g}(x^{k+1}) - \nabla {\mygreen g}(\xs) \|^2 -\frac{1}{2\eta_x}\|x^{k+1} -x^k\|^2 \notag\\
    && -\la \mathsf{K}^* (y^{k+1} - \ys), x^{k+1} - \xs\ra -\la  \nabla {\myblue f} (z^{k+1}) - \nabla {\myblue f} (\zs) , x^{k+1} - \xs\ra \notag\\
    && -\la\mathsf{K}^* (y^{k+1}-y^k), x^{k+1} - \xs \ra  -\la \nabla {\myblue f}(z^{k+1}) -  \nabla {\myblue f}(z^k), x^{k+1} - \xs \ra, \notag
\end{eqnarray*}
which completes the last part of the proof. 
\end{proof}

Descent lemmas for the dual updates, i.e.\ on $y,z$, are the same as for \algname{APGE} and \algname{ACV}.

\subsection{Proof of Theorem~\ref{th:convergence_ACV_and_APDTR}}
\label{sec:proof_of_th_convergence_acv_and_apdtr}

\paragraph{Lyapunov analysis for \algname{ACV-I}.}
Recall the definition of Lyapunov function for \algname{ACV-I} 
\begin{eqnarray*}
    \Psi_{k+1} &=& \frac{1}{2\eta_x}\|x^{k+1}-\xs\|^2 + \frac{1}{2\eta_y}\|y^{k+1}-\ys\|^2 + \frac{1}{\eta_z} D_{{\myblue f}}(z^{k+1};\zs)\\
    && -\la y^{k+1}-\ys, \mathsf{K}(x^{k+1}-\xs)\ra  -  \la \nabla {\myblue f}(z^{k+1}) - \nabla{\myblue f}(\zs), x^{k+1} - \xs\ra.
\end{eqnarray*}
Applying Lemma~\ref{lem:descent_lemma_acv_1_x}, precisely, \eqref{eq:descent_lemma_acv_1_x_1}, Lemma~\ref{lem:descent_lemma_acv_1_y} with $a_y =1$ and Lemma~\ref{lem:descent_lemma_apgd_z} with $a_z =1$, we obtain
\begin{eqnarray}
    \Psi_{k+1}&\leq& \frac{1}{2\eta_x}\|x^{k}-\xs\|^2 -\la \mathsf{K}^*(y^k - \ys), x^{k} - \xs\ra -\la \nabla {\myblue f} (z^k)  - \nabla {\myblue f} (\zs) , x^{k} - \xs\ra  -\mu_{\mygreen g}\|x^{k+1} -\xs\|^2 \notag\\
    &&   -\la \mathsf{K}^*(y^k - \ys), x^{k+1} - x^{k}\ra -\la \nabla {\myblue f} (z^k)  - \nabla {\myblue f} (\zs) , x^{k+1} - x^k\ra -\frac{1}{2\eta_x}\|x^{k+1} -x^k\|^2\notag\\ 
    && +  \frac{1}{2\eta_y}\|y^{k}-\ys\|^2  - \mu_{\hstar}\|y^{k+1}-\ys\|^2 +\la y^{k}-\ys, \mathsf{K}(x^{k+1}-x^k)\ra + \|\mathsf{K}\|^2\eta_x\eta_y \frac{1}{2\eta_x}\|x^{k+1}-x^k\|^2\notag\\
    && + \frac{1}{\eta_z} D_{{\myblue f}}(z^{k};\zs) - D_{\myblue f}(z^{k+1};\zs)  + \la \nabla {\myblue f}(z^{k}) - \nabla{\myblue f}(\zs), x^{k+1}-x^k \ra + L_{{\myblue f}}\eta_x\eta_z\frac{1}{2\eta_x}\|x^{k+1}-x^k \|^2 \notag\\
    &=& \Psi_k  -\mu_{\mygreen g}\|x^{k+1} -\xs\|^2 - \mu_{\hstar}\|y^{k+1}-\ys\|^2 - D_{\myblue f}(z^{k+1};\zs) \notag\\
    && - \left(1 - \|\mathsf{K}\|^2\eta_x\eta_y  -  L_{{\myblue f}}\eta_x\eta_z\right)\frac{1}{2\eta_x}\|x^{k+1} - x^k\|^2\notag\\
    &\leq& \Psi_k  -\mu_{\mygreen g}\|x^{k+1} -\xs\|^2 - \mu_{\hstar}\|y^{k+1}-\ys\|^2 - D_{\myblue f}(z^{k+1};\zs), \label{eq:proof_acv_i_conv_14}
\end{eqnarray}
where in the last inequality we used $1 - \|\mathsf{K}\|^2\eta_x\eta_y  -  L_{{\myblue f}}\eta_x\eta_z \geq 0$.

\paragraph{Lyapunov analysis for \algname{ACV-II}.}
Recall the definition of the Lyapunov function for \algname{ACV-II}:
\begin{eqnarray*}
    \Psi_{k+1} &=& \frac{1}{2\eta_x}\|x^{k+1}-\xs\|^2 + \frac{1}{2\eta_y}\|y^{k+1}-\ys\|^2 + \frac{1}{\eta_z} D_{{\myblue f}}(z^{k+1};\zs)\\
    && +\la y^{k+1}-\ys, \mathsf{K}(x^{k+1}-\xs)\ra  -  \la \nabla {\myblue f}(z^{k+1}) - \nabla{\myblue f}(\zs), x^{k+1} - \xs\ra.
\end{eqnarray*}
Applying Lemma~\ref{lem:descent_lemma_acv_2_x}, precisely, \eqref{eq:descent_lemma_acv_2_x_1}, Lemma~\ref{lem:descent_lemma_acv_2_y} with $a_y =1$ and Lemma~\ref{lem:descent_lemma_apgd_z} with $a_z =1$, we obtain
\begin{eqnarray}
    \Psi_{k+1}&\leq& \frac{1}{2\eta_x}\|x^{k}-\xs\|^2 +\la \mathsf{K}^*(y^k - \ys), x^{k} - \xs\ra -\la \nabla {\myblue f} (z^k)  - \nabla {\myblue f} (\zs) , x^{k} - \xs\ra  -\mu_{\mygreen g}\|x^{k+1} -\xs\|^2 \notag\\
    &&   -\la \mathsf{K}^*(y^{k+1} - y^k), x^{k+1} - \xs\ra -\la \nabla {\myblue f} (z^k)  - \nabla {\myblue f} (\zs) , x^{k+1} - x^k\ra -\frac{1}{2\eta_x}\|x^{k+1} -x^k\|^2\notag\\ 
    && +  \frac{1}{2\eta_y}\|y^{k}-\ys\|^2  - \mu_{\hstar}\|y^{k+1}-\ys\|^2 + \la y^{k+1}-y^k, \mathsf{K}(x^{k+1}-\xs)\ra + \|\mathsf{K}\|^2\eta_x\eta_y \frac{1}{2\eta_x}\|x^{k+1}-x^k\|^2\notag\\
    && + \frac{1}{\eta_z} D_{{\myblue f}}(z^{k};\zs) - D_{\myblue f}(z^{k+1};\zs)  + \la \nabla {\myblue f}(z^{k}) - \nabla{\myblue f}(\zs), x^{k+1}-x^k \ra + L_{{\myblue f}}\eta_x\eta_z\frac{1}{2\eta_x}\|x^{k+1}-x^k \|^2 \notag\\
    &=& \Psi_k  -\mu_{\mygreen g}\|x^{k+1} -\xs\|^2 - \mu_{\hstar}\|y^{k+1}-\ys\|^2 - D_{\myblue f}(z^{k+1};\zs) \notag\\
    && - \left(1 - \|\mathsf{K}\|^2\eta_x\eta_y  -  L_{{\myblue f}}\eta_x\eta_z\right)\frac{1}{2\eta_x}\|x^{k+1} - x^k\|^2\notag\\
    &\leq& \Psi_k  -\mu_{\mygreen g}\|x^{k+1} -\xs\|^2 - \mu_{\hstar}\|y^{k+1}-\ys\|^2 - D_{\myblue f}(z^{k+1};\zs), \label{eq:proof_acv_ii_conv_14}
\end{eqnarray}
where in the last inequality we used $1 - \|\mathsf{K}\|^2\eta_x\eta_y  -  L_{{\myblue f}}\eta_x\eta_z \geq 0$.

\paragraph{Lyapunov analysis for \algname{APDTR-I}.}
Recall the definition of the Lyapunov function for \algname{APDTR-I}:
\begin{eqnarray*}
    \Psi_{k+1} &=& \frac{1}{2\eta_x}\|x^{k+1}-\xs\|^2 + \frac{1}{2\eta_y}\|y^{k+1}-\ys\|^2 + \frac{1}{\eta_z} D_{{\myblue f}}(z^{k+1};\zs)\\
    && -\la y^{k+1}-\ys, \mathsf{K}(x^{k+1}-\xs)\ra  + \la \nabla {\myblue f}(z^{k+1}) - \nabla{\myblue f}(\zs), x^{k+1} - \xs\ra.
\end{eqnarray*}
Applying Lemma~\ref{lem:descent_lemma_apdtr_i_x}, precisely, \eqref{eq:descent_lemma_apdtr_i_x_1}, Lemma~\ref{lem:descent_lemma_acv_1_y} with $a_y =1$ and Lemma~\ref{lem:descent_lemma_apge_z} with $a_z =1$, we obtain
\begin{eqnarray}
    \Psi_{k+1}&\leq& \frac{1}{2\eta_x}\|x^{k}-\xs\|^2 -\la \mathsf{K}^*(y^k - \ys), x^{k} - \xs\ra +\la \nabla {\myblue f} (z^k)  - \nabla {\myblue f} (\zs) , x^{k} - \xs\ra  -\mu_{\mygreen g}\|x^{k+1} -\xs\|^2 \notag\\
    &&   -\la \mathsf{K}^*(y^k - \ys), x^{k+1} - x^{k}\ra -\la \nabla {\myblue f} (z^{k+1})  - \nabla {\myblue f} (z^k) , x^{k+1} - \xs\ra -\frac{1}{2\eta_x}\|x^{k+1} -x^k\|^2\notag\\ 
    && +  \frac{1}{2\eta_y}\|y^{k}-\ys\|^2  - \mu_{\hstar}\|y^{k+1}-\ys\|^2 +\la y^{k}-\ys, \mathsf{K}(x^{k+1}-x^k)\ra + \|\mathsf{K}\|^2\eta_x\eta_y \frac{1}{2\eta_x}\|x^{k+1}-x^k\|^2\notag\\
    && + \frac{1}{\eta_z} D_{{\myblue f}}(z^{k};\zs) - D_{\myblue f}(z^{k+1};\zs) + \la \nabla {\myblue f}(z^{k+1}) - \nabla{\myblue f}(z^k), x^{k+1}-\xs \ra + L_{{\myblue f}}\eta_x\eta_z\frac{1}{2\eta_x}\|x^{k+1}-x^k \|^2 \notag\\
    &=& \Psi_k  -\mu_{\mygreen g}\|x^{k+1} -\xs\|^2 - \mu_{\hstar}\|y^{k+1}-\ys\|^2 - D_{\myblue f}(z^{k+1};\zs) \notag\\
    && - \left(1 - \|\mathsf{K}\|^2\eta_x\eta_y  -  L_{{\myblue f}}\eta_x\eta_z\right)\frac{1}{2\eta_x}\|x^{k+1} - x^k\|^2\notag\\
    &\leq& \Psi_k  -\mu_{\mygreen g}\|x^{k+1} -\xs\|^2 - \mu_{\hstar}\|y^{k+1}-\ys\|^2 - D_{\myblue f}(z^{k+1};\zs), \label{eq:proof_apdtr_i_conv_14}
\end{eqnarray}
where in the last inequality we used $1 - \|\mathsf{K}\|^2\eta_x\eta_y  -  L_{{\myblue f}}\eta_x\eta_z \geq 0$.

\paragraph{Lyapunov analysis for \algname{APDTR-II}.}
Recall the definition of the Lyapunov function for \algname{APDTR-II}:
\begin{eqnarray*}
    \Psi_{k+1} &=& \frac{1}{2\eta_x}\|x^{k+1}-\xs\|^2 + \frac{1}{2\eta_y}\|y^{k+1}-\ys\|^2 + \frac{1}{\eta_z} D_{{\myblue f}}(z^{k+1};\zs)\\
    && +\la y^{k+1}-\ys, \mathsf{K}(x^{k+1}-\xs)\ra  + \la \nabla {\myblue f}(z^{k+1}) - \nabla{\myblue f}(\zs), x^{k+1} - \xs\ra.
\end{eqnarray*}
Applying Lemma~\ref{lem:descent_lemma_apdtr_ii_x}, precisely, \eqref{eq:descent_lemma_apdtr_ii_x_1}, Lemma~\ref{lem:descent_lemma_acv_2_y} with $a_y =1$ and Lemma~\ref{lem:descent_lemma_apge_z} with $a_z =1$, we obtain
\begin{eqnarray}
    \Psi_{k+1}&\leq& \frac{1}{2\eta_x}\|x^{k}-\xs\|^2 +\la y^k - \ys, \mathsf{K}(x^{k} - \xs)\ra +\la \nabla {\myblue f} (z^k)  - \nabla {\myblue f} (\zs) , x^{k} - \xs\ra  -\mu_{\mygreen g}\|x^{k+1} -\xs\|^2 \notag\\
    &&   - \la \mathsf{K}^*(y^{k+1} - y^k),x^{k+1} - \xs\ra  - \la \nabla {\myblue f} (z^{k+1})  - \nabla {\myblue f} (z^k) , x^{k+1} - \xs\ra -\frac{1}{2\eta_x}\|x^{k+1} -x^k\|^2\notag\\ 
    && +  \frac{1}{2\eta_y}\|y^{k}-\ys\|^2  - \mu_{\hstar}\|y^{k+1}-\ys\|^2 + \la y^{k+1}-y^k, \mathsf{K}(x^{k+1}-\xs)\ra + \|\mathsf{K}\|^2\eta_x\eta_y \frac{1}{2\eta_x}\|x^{k+1}-x^k\|^2\notag\\
    && + \frac{1}{\eta_z} D_{{\myblue f}}(z^{k};\zs) - D_{\myblue f}(z^{k+1};\zs)  + \la \nabla {\myblue f}(z^{k+1}) - \nabla{\myblue f}(z^k), x^{k+1}-\xs \ra + L_{{\myblue f}}\eta_x\eta_z\frac{1}{2\eta_x}\|x^{k+1}-x^k \|^2 \notag\\
    &=& \Psi_k  -\mu_{\mygreen g}\|x^{k+1} -\xs\|^2 - \mu_{\hstar}\|y^{k+1}-\ys\|^2 - D_{\myblue f}(z^{k+1};\zs) \notag\\
    && - \left(1 - \|\mathsf{K}\|^2\eta_x\eta_y  -  L_{{\myblue f}}\eta_x\eta_z\right)\frac{1}{2\eta_x}\|x^{k+1} - x^k\|^2\notag\\
    &\leq& \Psi_k  -\mu_{\mygreen g}\|x^{k+1} -\xs\|^2 - \mu_{\hstar}\|y^{k+1}-\ys\|^2 - D_{\myblue f}(z^{k+1};\zs), \label{eq:proof_apdtr_ii_conv_14}
\end{eqnarray}
where in the last inequality we used $1 - \|\mathsf{K}\|^2\eta_x\eta_y  -  L_{{\myblue f}}\eta_x\eta_z \geq 0$.

Applying Lemma~\ref{lem:lyaponov_func_acv_apdtr} into \eqref{eq:proof_acv_i_conv_14} and \eqref{eq:proof_acv_ii_conv_14}, \eqref{eq:proof_apdtr_i_conv_14} and \eqref{eq:proof_apdtr_ii_conv_14}, we obtain
\begin{eqnarray*}
    \Psi_{k+1} \leq \Psi_k -\min\left\{\mu_{\mygreen g}\eta_x, \mu_{\hstar}\eta_y, \frac{\eta_z}{2}\right\}\Psi_{k+1}
    &\Rightarrow& 
    \min\left\{1+ \mu_{\mygreen g}\eta_x, 1+ \mu_{\hstar}\eta_y, \frac{2 +\eta_z}{2}\right\} \Psi_{k+1} \leq \Psi_k.
\end{eqnarray*}
Thus, denoting $\theta = \max\left\{\frac{1}{1+\mu_{{\mygreen g}}\eta_x}, \frac{1}{1+\mu_{{\hstar}}\eta_y}, \frac{2}{2+\eta_z}\right\}$, we have 
\begin{eqnarray}
    \Psi_{k+1} &\leq& \theta \Psi_{k} ~\leq~ \theta^{k+1} \Psi_0. \label{eq:ascdancjnadlcal}
\end{eqnarray}

\subsection{Proof of Corollary~\ref{cor:acv_and_apdtr_complexity}}
\label{sec:proof_of_cor_acv_and_apdtr_complexity}
We aim to derive iteration complexity of Algorithm~\ref{alg:ACV} and Algorithm~\ref{alg:APDTR} as Corollary~\ref{cor:acv_and_apdtr_complexity}.
In view of \eqref{eq:ascdancjnadlcal}, we have 
\begin{equation}
    k \geq \cO\left(\frac{1}{1-\theta}\log \frac{1}{\varepsilon}\right) \quad\Rightarrow\quad \Psi_k \leq \varepsilon \Psi_0.
\end{equation}
By the definition of $\theta = \max\left\{\frac{1}{1+\mu_{{\mygreen g}}\eta_x}, \frac{1}{1+\mu_{{\hstar}}\eta_y}, \frac{2}{2+\eta_z}\right\}$, we have 
\begin{eqnarray*}
    \cO\left(\frac{1}{1-\theta}\log\frac{1}{\varepsilon}\right) &=& \cO\left(\max\left\{\frac{1+\mu_{{\mygreen g}}\eta_x}{\mu_{{\mygreen g}}\eta_x}, \frac{1+\mu_{{\hstar}}\eta_y}{\mu_{{\hstar}}\eta_y}, \frac{2+\eta_z}{\eta_z}\right\}\log\frac{1}{\varepsilon}\right)\\
    &=& \cO\left(\max\left\{1+\frac{1}{\mu_{{\mygreen g}}\eta_x}, 1+\frac{1}{\mu_{{\hstar}}\eta_y}, 1+\frac{2}{\eta_z}\right\}\log\frac{1}{\varepsilon}\right).
\end{eqnarray*}
By the selection of parameters $\eta_x, \eta_y, \eta_z$ as follows
\begin{equation*}
        \eta_x = \min\left\{\frac{1}{\sqrt{L_{\myblue f}\mu_{\mygreen g}}}, \sqrt{\frac{\mu_{\hstar}}{\|\mathsf{K}\|^2\mu_{\mygreen g}}}\right\},
        \quad \eta_y = \frac{1}{2}\sqrt{\frac{\mu_{\mygreen g}}{\|\mathsf{K}\|^2\mu_{\hstar}}},
        \quad \eta_z  = \frac{1}{2}\sqrt{\frac{\mu_{\mygreen g}}{L_{\myblue f}}},
    \end{equation*}
we obtain 
\begin{eqnarray*}
    \cO\left(\frac{1}{1-\theta}\log\frac{1}{\varepsilon}\right) &=& \cO\left(\max\left\{1+ \sqrt{\frac{L_{\myblue f}}{\mu_{{\mygreen g}}}}, 1+ \sqrt{\frac{\|\mathsf{K}\|^2}{\mu_{\mygreen g}\mu_{{\hstar}}}}\right\}\log\frac{1}{\varepsilon}\right).
\end{eqnarray*}
This proves the statement of Corollary~\ref{cor:acv_and_apdtr_complexity}.

\subsection{Proof of Theorem~\ref{th:convergence_ACV_and_APDTR_non_smooth}}
\label{sec:proof_convergence_acv_apdtr_non_smooth}

\paragraph{Lyapunov analysis for \algname{ACV-I}.} Recall the Lyapunov function for \algname{ACV-I}
\begin{eqnarray*}
    \Psi_{k+1} &=& \frac{1}{2\eta_x}\|x^{k+1}-\xs\|^2 + \frac{1}{2\eta_y}\|y^{k+1}-\ys\|^2 + \frac{1}{\eta_z} D_{{\myblue f}}(z^{k+1};\zs)\\
    && -\la y^{k+1}-\ys, \mathsf{K}(x^{k+1}-\xs)\ra  - \la \nabla {\myblue f}(z^{k+1}) - \nabla{\myblue f}(\zs), x^{k+1} - \xs\ra.
\end{eqnarray*}

Combining Lemma~\ref{lem:descent_lemma_acv_1_x} (inequality \eqref{eq:descent_lemma_acv_1_x_2}), Lemma~\ref{lem:descent_lemma_acv_1_y} with $\mu_{\hstar}=0$ and $a_y = \nicefrac{1}{2}$ and Lemma~\ref{lem:descent_lemma_apgd_z} with $a_z = \nicefrac{1}{2}$ together, we obtain 
\begin{eqnarray}
    \Psi_{k+1}&\leq& \frac{1}{2\eta_x}\|x^{k}-\xs\|^2 - \frac{\mu_{{\mygreen g}}}{2}\|x^{k+1} - \xs \|^2 - \frac{1}{2L_{\mygreen g}}\|\nabla {\mygreen g}(x^{k+1}) - \nabla {\mygreen g}(\xs)\|^2  -\frac{1}{2\eta_x}\|x^{k+1} -x^k\|^2 \notag\\
    &&  -\la \mathsf{K}^*(y^k - \ys), x^{k} - \xs\ra  -\la \mathsf{K}^*(y^k - \ys), x^{k+1} - x^{k}\ra \notag\\ 
    &&  -\la \nabla {\myblue f} (z^k)  - \nabla {\myblue f} (\zs) , x^{k} - \xs\ra -\la \nabla {\myblue f} (z^k)  - \nabla {\myblue f} (\zs) , x^{k+1} - x^k\ra \notag\\
    &&+ \frac{1}{2\eta_y}\|y^{k}-\ys\|^2   -\frac{1}{4\eta_y} \|y^{k+1}- y^k\|^2  + \|\mathsf{K}\|^2\eta_x\eta_y \frac{1}{\eta_x}\|x^{k+1}-x^k\|^2 \notag\\
    && + \frac{1}{\eta_z} D_{{\myblue f}}(z^{k}; \zs) - D_{{\myblue f}}(z^{k+1}; \zs) - D_{{\myblue f}}(\zs;z^{k+1})  -\frac{1}{2\eta_z}D_{{\myblue f}}(z^k; z^{k+1}) + L_{{\myblue f}}\eta_x\eta_z\frac{1}{\eta_x}\|x^{k+1}-x^k \|^2\notag\\
    && +  \la y^{k}-\ys, \mathsf{K}(x^{k+1}-x^k)\ra  + \la \nabla {\myblue f}(z^{k}) - \nabla{\myblue f}(\zs), x^{k+1}-x^k \ra \notag \\
    &=& \Psi_k - \frac{\mu_{{\mygreen g}}}{2}\|x^{k+1} - \xs \|^2   - D_{\myblue f}(z^{k+1};\zs) - \frac{1}{2L_{\mygreen g}}\|\nabla {\mygreen g}(x^{k+1}) - \nabla {\mygreen g}(\xs)\|^2  \notag\\
    && -\left(1-\delta  - 2\|\mathsf{K}\|^2\eta_x\eta_y - 2L_{f}\eta_x\eta_z\right) \frac{1}{2\eta_x}\|x^{k+1}-x^k\|^2 \notag\\
    && - \frac{\delta}{2\eta_x}\|x^{k+1}-x^k\|^2 -\frac{1}{4\eta_y} \|y^{k+1}- y^k\|^2 - \frac{1}{2\eta_z}D_f(z^{k+1}; z^k) - D_{\myblue f}(\zs; z^{k+1}). \label{eq:proof_acv_i_conv_2_14}
\end{eqnarray}

We now bound the term $-\frac{\delta}{2\eta_x}\|x^{k+1}-x^k\|^2$ from (\ref{eq:proof_acv_i_conv_2_14}). By the update rule for $x^{k+1}$ in \algname{ACV-I} (Line $3$ (left) from Algorithm~\ref{alg:ACV}), we have 
\begin{eqnarray*}
    -\frac{\delta}{2\eta_x}\|x^{k+1}-x^k\|^2 &=& - \frac{\delta\eta_x}{2}\|\nabla {\mygreen g}(x^{k+1}) + \mathsf{K}^*y^{k} + \nabla {\myblue f}(z^{k})\|^2  \\
    &\leq& - \frac{\delta\eta_x}{10}\|\mathsf{K}^*(y^{k+1}-\ys)\|^2 +  \frac{\delta\eta_x}{2}\|\mathsf{K}^*(y^{k+1}-y^k)\|^2 + \frac{\delta\eta_x}{2}\|\nabla{\mygreen g}(x^{k+1}) - \nabla{\mygreen g}(\xs)\|^2 \\
    &&+ \frac{\delta\eta_x}{2}\|\nabla{\myblue f}(z^{k+1}) - \nabla{\myblue f}(\zs)\|^2 + \frac{\delta\eta_x}{2}\|\nabla{\myblue f}(z^{k+1}) - \nabla{\myblue f}(z^k)\|^2.
\end{eqnarray*}

 Using $L_{\myblue f}$-smoothness of ${\myblue f}$ and boundedness of ${\mathsf K}$, we have 
\begin{eqnarray}
    -\frac{\delta}{2\eta_x}\|x^{k+1}-x^k\|^2 &\leq& - \frac{\delta\eta_x\lambda_{\min}(\mathsf{K}\mathsf{K}^*) }{10}\|y^{k+1}-\ys\|^2 +  \frac{\delta\eta_x\|\mathsf{K}\|^2}{2}\|y^{k+1}-y^k\|^2  \notag\\
    &&+ \frac{\delta\eta_x}{2}\|\nabla{\mygreen g}(x^{k+1}) - \nabla{\mygreen g}(\xs)\|^2 + \delta L_{\myblue f}\eta_xD_{\myblue f}(\zs;z^{k+1}) + \delta L_{\myblue f}\eta_xD_{\myblue f}(z^{k};z^{k+1}). \label{eq:proof_acv_i_conv_2_15}
\end{eqnarray}
Plugging \eqref{eq:proof_acv_i_conv_2_15} into \eqref{eq:proof_acv_i_conv_2_14}, we obtain
\begin{eqnarray*}
    \Psi_{k+1} &\leq& \Psi_k  - \frac{\mu_{{\mygreen g}}}{2}\|x^{k+1} - \xs \|^2 - \frac{\delta\eta_x\lambda_{\min}(\mathsf{K}\mathsf{K}^*) }{10}\|y^{k+1}-\ys\|^2  - D_{\myblue f}(z^{k+1};\zs)\\
    && -\left(1-\delta  - 2\|\mathsf{K}\|^2\eta_x\eta_y - 2L_{f}\eta_x\eta_z\right) \frac{1}{2\eta_x}\|x^{k+1}-x^k\|^2 \notag\\
    && - \left(\frac{1}{2L_{\mygreen g}} - \frac{\delta\eta_x}{2}\right)\|\nabla {\mygreen g}(x^{k+1}) - \nabla {\mygreen g}(\xs)\|^2  - \left( 1 - 2\delta\|\mathsf{K}\|^2\eta_x\eta_y\right)\frac{1}{4\eta_y}\|y^{k+1}-y^k\|^2 \notag\\
    && - \left(1 - 2\delta L_{\myblue f}\eta_x\eta_z\right)\frac{1}{2\eta_z}D_f(z^{k};z^{k+1}) - \left(1 - \delta L_{\myblue f}\eta_x\right)D_{\myblue f}(\zs; z^{k+1}).
\end{eqnarray*}
Assume that $\delta$ and parameters $\eta_x, \eta_y, \eta_z$ satisfy the conditions 
\begin{eqnarray*}
    \delta \leq \frac{1}{2}, \quad \delta \leq \frac{1}{L_{\mygreen g}\eta_x}, \quad \delta \leq \frac{1}{L_{\myblue f}\eta_x},\quad 1-8\|\mathsf{K}\|^2\eta_x\eta_y \geq 0, \quad 1-8L_{\myblue f}\eta_x\eta_z \geq 0.
\end{eqnarray*}
Then, we derive the inequality
\begin{eqnarray}
    \Psi_{k+1} \leq \Psi_k - \min\left\{\frac{\mu_{\mygreen g}\eta_x}{2}, \frac{1}{20} \delta\lambda_{\min}(\mathsf{K}\mathsf{K}^*) \eta_x\eta_y, \frac{\eta_z}{2}\right\}\Psi_{k+1} &\Rightarrow& \Psi_{k+1} \leq \theta \Psi_k \leq \theta^{k+1}\Psi_0,\notag
\end{eqnarray}
where the coefficient $\theta$ is defined as 
\begin{equation}
    \theta = \max\left\{\frac{2}{2+\mu_{\mygreen g}\eta_x}, \frac{20}{20+\delta\lambda_{\min}(\mathsf{K}\mathsf{K}^*) \eta_x\eta_y}, \frac{2}{2+\eta_z}\right\}. \notag
\end{equation}
 Selecting $\delta = \min\left\{\frac{1}{2},\frac{1}{\eta_x(L_{\mygreen g}+L_{\myblue f})}\right\}$, we have 
\begin{equation}
    \label{eq:proof_acv_i_conv_2_16}
    \Psi_{k+1} \leq \theta^{k+1}\Psi_0, 
\end{equation}
where $\theta$ is defined as 
\begin{equation*}
    \theta = \max\left\{\frac{2}{2+\mu_{{\mygreen g}}\eta_x}, \frac{40}{40+\lambda_{\min}(\mathsf{K}\mathsf{K}^*)\eta_x\eta_y}, \frac{20}{20+\frac{\lambda_{\min}(\mathsf{K}\mathsf{K}^*)}{(L_{\mygreen g}+ L_{\myblue f})}\eta_y}, \frac{2}{2+\eta_z}\right\}.
\end{equation*}
This finishes the proof. 

\paragraph{Lyapunov analysis for \algname{ACV-II}.} We define the Lyapunov function 
\begin{eqnarray*}
    \Psi_{k+1} &=& \frac{1}{2\eta_x}\|x^{k+1}-\xs\|^2 + \frac{1}{2\eta_y}\|y^{k+1}-\ys\|^2 + \frac{1}{\eta_z} D_{{\myblue f}}(z^{k+1};\zs)\\
    && +\la y^{k+1}-\ys, \mathsf{K}(x^{k+1}-\xs)\ra  - \la \nabla {\myblue f}(z^{k+1}) - \nabla{\myblue f}(\zs), x^{k+1} - \xs\ra.
\end{eqnarray*}
The rest of the proof is the same up to the cross terms as in the proof for \algname{ACV-I}. The difference is in  the application of Lemma~\ref{lem:descent_lemma_acv_2_x} (inequality \eqref{eq:descent_lemma_acv_2_x_2}), Lemma~\ref{lem:descent_lemma_acv_2_y} with $\mu_{\hstar}=0$ and $a_y = \nicefrac{1}{2}$ and Lemma~\ref{lem:descent_lemma_apgd_z} with $a_z = \nicefrac{1}{2}$:
\begin{eqnarray}
    \Psi_{k+1} &\leq& 
    \frac{1}{2\eta_x}\|x^{k+1} - \xs\|^2  - \frac{\mu_{\mygreen g}}{2}\|x^{k+1} - \xs\|^2 - \frac{1}{2L_{\mygreen g}}\| \nabla {\mygreen g}(x^{k+1}) - \nabla {\mygreen g}(\xs)\|^2 - \frac{1}{2\eta_x}\|x^{k+1}-x^k\|^2 \notag\\
    && - \la \mathsf{K}^*(y^{k+1} - y^k), x^{k+1} - \xs \ra - \la \nabla {\myblue f}(z^{k}) - \nabla {\myblue f}(\zs), x^{k} - \xs \ra  - \la \nabla {\myblue f}(z^{k}) - \nabla {\myblue f}(\zs), x^{k+1} - x^k \ra\notag\\
    && + \frac{1}{2\eta_y}\|y^k -\ys\|^2 - \frac{1}{4\eta_y}\|y^{k+1}-y^k\|^2 + \|\mathsf{K}\|^2 \eta_x\eta_y \frac{1}{\eta_y}\|x^{k+1}-x^k\|^2 \notag\\
    && \frac{1}{\eta_z}D_{\myblue f}(z^{k};\zs) - D_{\myblue f}(z^{k+1};\zs) - D_{\myblue f}(\zs;z^{k+1}) - \frac{1}{2\eta_z}D_{\myblue f}(z^{k};z^{k+1}) + L_{\myblue f}\eta_x\eta_z\frac{1}{\eta_x}\|x^{k+1}-x^k\|^2\notag\\
    && + \la y^k - \ys, \mathsf{K}(x^k -\xs)\ra + \la y^{k+1}- y^k, \mathsf{K}(x^{k+1}-\xs)\ra  + \la \nabla {\myblue f}(z^{k}) - \nabla {\myblue f}(\zs), x^{k+1} - x^k \ra  \notag\\
    &=& \Psi_k - \frac{\mu_{{\mygreen g}}}{2}\|x^{k+1} - \xs \|^2   - D_{\myblue f}(z^{k+1};\zs) - \frac{1}{2L_{\mygreen g}}\|\nabla {\mygreen g}(x^{k+1}) - \nabla {\mygreen g}(\xs)\|^2  \notag\\
    && -\left(1-\delta  - 2\|\mathsf{K}\|^2\eta_x\eta_y - 2L_{f}\eta_x\eta_z\right) \frac{1}{2\eta_x}\|x^{k+1}-x^k\|^2 \notag\\
    && - \frac{\delta}{2\eta_x}\|x^{k+1}-x^k\|^2 -\frac{1}{4\eta_y} \|y^{k+1}- y^k\|^2 - \frac{1}{2\eta_z}D_f(z^{k+1}; z^k) - D_{\myblue f}(\zs; z^{k+1}). 
\end{eqnarray}

It remains to show the following inequality:
\begin{eqnarray*}
    -\frac{\delta}{2\eta_x}\|x^{k+1}-x^k\|^2 &=& - \frac{\delta\eta_x}{2}\|\nabla {\mygreen g}(x^{k+1}) + \mathsf{K}^* (2y^{k+1} - y^k) + \nabla {\myblue f}(z^{k})\|^2  \\
    &\leq& - \frac{\delta\eta_x}{10}\|\mathsf{K}^*(y^{k+1}-\ys)\|^2 +  \frac{\delta\eta_x}{2}\|\mathsf{K}^*(y^{k+1}-y^k)\|^2 + \frac{\delta\eta_x}{2}\|\nabla{\mygreen g}(x^{k+1}) - \nabla{\mygreen g}(\xs)\|^2 \\
    &&+ \frac{\delta\eta_x}{2}\|\nabla{\myblue f}(z^{k+1}) - \nabla{\myblue f}(\zs)\|^2 + \frac{\delta\eta_x}{2}\|\nabla{\myblue f}(z^{k+1}) - \nabla{\myblue f}(z^k)\|^2,
\end{eqnarray*}
which is the same as inequality~\eqref{eq:proof_acv_i_conv_2_15}. 
Thus, setting $\delta = \min\left\{\frac{1}{2},\frac{1}{\eta_x(L_{\mygreen g}+L_{\myblue f})}\right\}$, we get
\begin{eqnarray}
    \label{eq:proof_acv_ii_conv_2_16}
    \Psi_{k+1} \leq \Psi_k - \min\left\{\frac{\mu_{\mygreen g}\eta_x}{2}, \frac{1}{20} \delta\lambda_{\min}(\mathsf{K}\mathsf{K}^*) \eta_x\eta_y, \frac{\eta_z}{2}\right\}\Psi_{k+1} &\Rightarrow& \Psi_{k+1} \leq \theta \Psi_k \leq \theta^{k+1}\Psi_0,
\end{eqnarray}
where the coefficient $\theta$ is defined as 
\begin{equation*}
    \theta = \max\left\{\frac{2}{2+\mu_{{\mygreen g}}\eta_x}, \frac{40}{40+\lambda_{\min}(\mathsf{K}\mathsf{K}^*)\eta_x\eta_y}, \frac{20}{20+\frac{\lambda_{\min}(\mathsf{K}\mathsf{K}^*)}{(L_{\mygreen g}+ L_{\myblue f})}\eta_y}, \frac{2}{2+\eta_z}\right\}.
\end{equation*}

\paragraph{Lyapunov analysis for \algname{APDTR-I}.} We define the Lyapunov function 
\begin{eqnarray*}
    \Psi_{k+1} &=& \frac{1}{2\eta_x}\|x^{k+1}-\xs\|^2 + \frac{1}{2\eta_y}\|y^{k+1}-\ys\|^2 + \frac{1}{\eta_z} D_{{\myblue f}}(z^{k+1};\zs)\\
    && -\la y^{k+1}-\ys, \mathsf{K}(x^{k+1}-\xs)\ra  + \la \nabla {\myblue f}(z^{k+1}) - \nabla{\myblue f}(\zs), x^{k+1} - \xs\ra.
\end{eqnarray*}
The rest of the proof is the same up to the cross terms as in the proof for \algname{ACV-I}. The difference is in  the application of Lemma~\ref{lem:descent_lemma_apdtr_i_x} (inequality \eqref{eq:descent_lemma_apdtr_i_x_2}), Lemma~\ref{lem:descent_lemma_acv_1_y} with $\mu_{\hstar}=0$ and $a_y = \nicefrac{1}{2}$ and Lemma~\ref{lem:descent_lemma_apge_z} with $a_z = \nicefrac{1}{2}$:
\begin{eqnarray}
    \Psi_{k+1} &\leq& 
    \frac{1}{2\eta_x}\|x^{k+1} - \xs\|^2  - \frac{\mu_{\mygreen g}}{2}\|x^{k+1} - \xs\|^2 - \frac{1}{2L_{\mygreen g}}\| \nabla {\mygreen g}(x^{k+1}) - \nabla {\mygreen g}(\xs)\|^2 - \frac{1}{2\eta_x}\|x^{k+1}-x^k\|^2 \notag\\
    && - \la \mathsf{K}^*(y^{k} - \ys), x^{k} - \xs \ra - \la \mathsf{K}^*(y^{k} - \ys), x^{k+1} - x^k \ra  - \la \nabla {\myblue f}(z^{k+1}) - \nabla {\myblue f}(z^k), x^{k+1} - \xs \ra\notag\\
    && + \frac{1}{2\eta_y}\|y^k -\ys\|^2 - \frac{1}{4\eta_y}\|y^{k+1}-y^k\|^2 + \|\mathsf{K}\|^2 \eta_x\eta_y \frac{1}{\eta_y}\|x^{k+1}-x^k\|^2 \notag\\
    && \frac{1}{\eta_z}D_{\myblue f}(z^{k};\zs) - D_{\myblue f}(z^{k+1};\zs) - D_{\myblue f}(\zs;z^{k+1}) - \frac{1}{2\eta_z}D_{\myblue f}(z^{k};z^{k+1})\notag\\
    && + \la y^{k}- \ys, \mathsf{K}(x^{k+1}-x^k)\ra + \la \nabla {\myblue f}(z^{k}) - \nabla {\myblue f}(\zs), x^{k} - \xs \ra + \la \nabla {\myblue f}(z^{k+1}) - \nabla {\myblue f}(z^k), x^{k+1} - \xs \ra \notag\\
    &=& \Psi_k - \frac{\mu_{{\mygreen g}}}{2}\|x^{k+1} - \xs \|^2   - D_{\myblue f}(z^{k+1};\zs) - \frac{1}{2L_{\mygreen g}}\|\nabla {\mygreen g}(x^{k+1}) - \nabla {\mygreen g}(\xs)\|^2  \notag\\
    && -\left(1-\delta  - 2\|\mathsf{K}\|^2\eta_x\eta_y - 2L_{f}\eta_x\eta_z\right) \frac{1}{2\eta_x}\|x^{k+1}-x^k\|^2 \notag\\
    && - \frac{\delta}{2\eta_x}\|x^{k+1}-x^k\|^2 -\frac{1}{4\eta_y} \|y^{k+1}- y^k\|^2 - \frac{1}{2\eta_z}D_f(z^{k+1}; z^k) - D_{\myblue f}(\zs; z^{k+1}). 
\end{eqnarray}
It remains to show the following inequality:
\begin{eqnarray*}
    -\frac{\delta}{2\eta_x}\|x^{k+1}-x^k\|^2 &=& - \frac{\delta\eta_x}{2}\|\nabla {\mygreen g}(x^{k+1}) + \mathsf{K}^* y^k + (2\nabla {\myblue f}(z^{k+1}) - \nabla {\myblue f}(z^{k}))\|^2  \\
    &\leq& - \frac{\delta\eta_x}{10}\|\mathsf{K}^*(y^{k+1}-\ys)\|^2 +  \frac{\delta\eta_x}{2}\|\mathsf{K}^*(y^{k+1}-y^k)\|^2 + \frac{\delta\eta_x}{2}\|\nabla{\mygreen g}(x^{k+1}) - \nabla{\mygreen g}(\xs)\|^2 \\
    &&+ \frac{\delta\eta_x}{2}\|\nabla{\myblue f}(z^{k+1}) - \nabla{\myblue f}(\zs)\|^2 + \frac{\delta\eta_x}{2}\|\nabla{\myblue f}(z^{k+1}) - \nabla{\myblue f}(z^k)\|^2,
\end{eqnarray*}
which is the same as inequality~\eqref{eq:proof_acv_i_conv_2_15}. 
Thus, setting $\delta = \min\left\{\frac{1}{2},\frac{1}{\eta_x(L_{\mygreen g}+L_{\myblue f})}\right\}$, we get
\begin{eqnarray}
    \label{eq:proof_apdtr_i_conv_2_16}
    \Psi_{k+1} \leq \Psi_k - \min\left\{\frac{\mu_{\mygreen g}\eta_x}{2}, \frac{1}{20} \delta\lambda_{\min}(\mathsf{K}\mathsf{K}^*) \eta_x\eta_y, \frac{\eta_z}{2}\right\}\Psi_{k+1} &\Rightarrow& \Psi_{k+1} \leq \theta \Psi_k \leq \theta^{k+1}\Psi_0,
\end{eqnarray}
where the coefficient $\theta$ is defined as 
\begin{equation*}
    \theta = \max\left\{\frac{2}{2+\mu_{{\mygreen g}}\eta_x}, \frac{40}{40+\lambda_{\min}(\mathsf{K}\mathsf{K}^*)\eta_x\eta_y}, \frac{20}{20+\frac{\lambda_{\min}(\mathsf{K}\mathsf{K}^*)}{(L_{\mygreen g}+ L_{\myblue f})}\eta_y}, \frac{2}{2+\eta_z}\right\}.
\end{equation*}

\paragraph{Lyapunov analysis for \algname{APDTR-II}.} We define the Lyapunov function 
\begin{eqnarray*}
    \Psi_{k+1} &=& \frac{1}{2\eta_x}\|x^{k+1}-\xs\|^2 + \frac{1}{2\eta_y}\|y^{k+1}-\ys\|^2 + \frac{1}{\eta_z} D_{{\myblue f}}(z^{k+1};\zs)\\
    && +\la y^{k+1}-\ys, \mathsf{K}(x^{k+1}-\xs)\ra  + \la \nabla {\myblue f}(z^{k+1}) - \nabla{\myblue f}(\zs), x^{k+1} - \xs\ra.
\end{eqnarray*}
The rest of the proof is the same up to the sign flip of the cross terms as in the proof for \algname{ACV-I}. The difference is in  the application of Lemma~\ref{lem:descent_lemma_apdtr_ii_x} (inequality \eqref{eq:descent_lemma_apdtr_ii_x_2}), Lemma~\ref{lem:descent_lemma_acv_2_y} with $\mu_{\hstar} =0$ and $a_y = \nicefrac{1}{2}$ and Lemma~\ref{lem:descent_lemma_apge_z} with $a_z = \nicefrac{1}{2}$:
\begin{eqnarray}
    \Psi_{k+1} &\leq& 
    \frac{1}{2\eta_x}\|x^{k+1} - \xs\|^2  - \frac{\mu_{\mygreen g}}{2}\|x^{k+1} - \xs\|^2 - \frac{1}{2L_{\mygreen g}}\| \nabla {\mygreen g}(x^{k+1}) - \nabla {\mygreen g}(\xs)\|^2 - \frac{1}{2\eta_x}\|x^{k+1}-x^k\|^2 \notag\\
    && - \la \mathsf{K}^*(y^{k+1} - y^k), x^{k+1} - \xs \ra - \la \nabla {\myblue f}(z^{k+1}) - \nabla {\myblue f}(z^k), x^{k+1} - \xs \ra  \notag\\
    && + \frac{1}{2\eta_y}\|y^k -\ys\|^2 - \frac{1}{4\eta_y}\|y^{k+1}-y^k\|^2 + \|\mathsf{K}\|^2 \eta_x\eta_y \frac{1}{\eta_y}\|x^{k+1}-x^k\|^2 \notag\\
    && + \la y^k - \ys, \mathsf{K}(x^k -\xs)\ra + \la y^{k+1}- y^k, \mathsf{K}(x^{k+1}-\xs)\ra \notag\\
    && \frac{1}{\eta_z}D_{\myblue f}(z^{k};\zs) - D_{\myblue f}(z^{k+1};\zs) - D_{\myblue f}(\zs;z^{k+1}) - \frac{1}{2\eta_z}D_{\myblue f}(z^{k};z^{k+1}) + L_{f}\eta_x\eta_z \frac{1}{\eta_x}\|x^{k+1}-x^k\|^2 \notag\\
    && + \la \nabla {\myblue f}(z^{k}) - \nabla {\myblue f}(\zs), x^{k} - \xs \ra + \la \nabla {\myblue f}(z^{k+1}) - \nabla {\myblue f}(z^k), x^{k+1} - \xs \ra \notag\\
    &=& \Psi_k - \frac{\mu_{{\mygreen g}}}{2}\|x^{k+1} - \xs \|^2   - D_{\myblue f}(z^{k+1};\zs) - \frac{1}{2L_{\mygreen g}}\|\nabla {\mygreen g}(x^{k+1}) - \nabla {\mygreen g}(\xs)\|^2  \notag\\
    && -\left(1-\delta  - 2\|\mathsf{K}\|^2\eta_x\eta_y - 2L_{f}\eta_x\eta_z\right) \frac{1}{2\eta_x}\|x^{k+1}-x^k\|^2 \notag\\
    && - \frac{\delta}{2\eta_x}\|x^{k+1}-x^k\|^2 -\frac{1}{4\eta_y} \|y^{k+1}- y^k\|^2 - \frac{1}{2\eta_z}D_f(z^{k+1}; z^k) - D_{\myblue f}(\zs; z^{k+1}). 
\end{eqnarray}
It remains  to show the following inequality:
\begin{eqnarray*}
    -\frac{\delta}{2\eta_x}\|x^{k+1}-x^k\|^2 &=& - \frac{\delta\eta_x}{2}\|\nabla {\mygreen g}(x^{k+1}) + \mathsf{K}^*(2y^{k+1} - y^k) + (2\nabla {\myblue f}(z^{k+1}) - \nabla {\myblue f}(z^{k}))\|^2  \\
    &\leq& - \frac{\delta\eta_x}{10}\|\mathsf{K}^*(y^{k+1}-\ys)\|^2 +  \frac{\delta\eta_x}{2}\|\mathsf{K}^*(y^{k+1}-y^k)\|^2 + \frac{\delta\eta_x}{2}\|\nabla{\mygreen g}(x^{k+1}) - \nabla{\mygreen g}(\xs)\|^2 \\
    &&+ \frac{\delta\eta_x}{2}\|\nabla{\myblue f}(z^{k+1}) - \nabla{\myblue f}(\zs)\|^2 + \frac{\delta\eta_x}{2}\|\nabla{\myblue f}(z^{k+1}) - \nabla{\myblue f}(z^k)\|^2,
\end{eqnarray*}
which is the same as inequality~\eqref{eq:proof_acv_i_conv_2_15}. 
Thus, setting $\delta = \min\left\{\frac{1}{2},\frac{1}{\eta_x(L_{\mygreen g}+L_{\myblue f})}\right\}$, we get
\begin{eqnarray}
    \label{eq:proof_apdtr_ii_conv_2_16}
    \Psi_{k+1} \leq \Psi_k - \min\left\{\frac{\mu_{\mygreen g}\eta_x}{2}, \frac{1}{20} \delta\lambda_{\min}(\mathsf{K}\mathsf{K}^*) \eta_x\eta_y, \frac{\eta_z}{2}\right\}\Psi_{k+1} &\Rightarrow& \Psi_{k+1} \leq \theta \Psi_k \leq \theta^{k+1}\Psi_0,
\end{eqnarray}
where the coefficient $\theta$ is defined as 
\begin{equation*}
    \theta = \max\left\{\frac{2}{2+\mu_{{\mygreen g}}\eta_x}, \frac{40}{40+\lambda_{\min}(\mathsf{K}\mathsf{K}^*)\eta_x\eta_y}, \frac{20}{20+\frac{\lambda_{\min}(\mathsf{K}\mathsf{K}^*)}{(L_{\mygreen g}+ L_{\myblue f})}\eta_y}, \frac{2}{2+\eta_z}\right\}.
\end{equation*}

\subsection{Proof of Corollary~\ref{cor:acv_and_apdtr_complexity_non_smooth}}
\label{sec:proof_of_cor_acv_and_apdtr_complexity_non_smooth}
In view of \eqref{eq:proof_acv_i_conv_2_16} and \eqref{eq:proof_acv_ii_conv_2_16}, \eqref{eq:proof_apdtr_i_conv_2_16} and \eqref{eq:proof_apdtr_ii_conv_2_16}, we have 
\begin{equation}
    k \geq \cO\left(\frac{1}{1-\theta}\log \frac{1}{\varepsilon}\right) \quad\Rightarrow\quad \Psi_k \leq \varepsilon \Psi_0.
\end{equation}
Set $\eta_y = \frac{1}{8\|\mathsf{K}\|^2 \eta_x}$, and $\eta_z= \frac{1}{8L_{\myblue f}\eta_x}$. Then by the definition of $\theta$, we have 
\begin{eqnarray*}
    \cO\left(\frac{1}{1-\theta}\log\frac{1}{\varepsilon}\right) 
    &=& \cO\left(\max\left\{1+\frac{1}{\mu_{{\mygreen g}}\eta_x}, 1+\frac{\|\mathsf{K}\|^2}{\lambda_{\min}(\mathsf{K}\mathsf{K}^*)}, 1+\frac{\eta_x\|\mathsf{K}\|^2(L_{\mygreen g}+L_{\myblue f})}{\lambda_{\min}(\mathsf{K}\mathsf{K}^*)}, 1+L_{\myblue f}\eta_x\right\}\log\frac{1}{\varepsilon}\right).
\end{eqnarray*}
Setting $\eta_x = \min\left\{ \frac{1}{\sqrt{L_{\myblue f}\mu_{\mygreen g}}}, \sqrt{\frac{\lambda_{\min}(\mathsf{K}\mathsf{K}^*)}{\|\mathsf{K}\|^2}}\cdot \frac{1}{\sqrt{(L_{\mygreen g}+L_{\myblue f})\mu_{\mygreen g}}}\right\} $, we have 
\begin{eqnarray*}
    \cO\left(\frac{1}{1-\theta}\log\frac{1}{\varepsilon}\right) &=& \cO\left(\max\left\{ 1+\sqrt{\frac{L_{\myblue f}}{\mu_{\mygreen g}}}, 1+\sqrt{\frac{L_{\myblue f} + L_{\mygreen g}}{\mu_{\mygreen g}}  \cdot \frac{\|\mathsf{K}\|^2}{\lambda_{\min}(\mathsf{K}\mathsf{K}^*)}}, 1+ \frac{\|\mathsf{K}\|^2}{\lambda_{\min}(\mathsf{K}\mathsf{K}^*)}\right\}\log\frac{1}{\varepsilon}\right)\\
    && \cO\left(\left(\sqrt{\frac{L_{\myblue f} + L_{\mygreen g}}{\mu_{\mygreen g}}} \cdot \frac{\|\mathsf{K}\|}{\sqrt{\lambda_{\min}(\mathsf{K}\mathsf{K}^*)}} + \frac{\|\mathsf{K}\|^2}{\lambda_{\min}(\mathsf{K}\mathsf{K}^*)}\right)\log\frac{1}{\varepsilon}\right).
\end{eqnarray*}

\end{document}